\documentclass[10pt]{amsart} 
\usepackage[portrait,margin=2.54cm]{geometry}

\usepackage{hyperref}
\usepackage{graphicx}
\usepackage{upref,setspace}
\usepackage{xcolor,colortbl}
\usepackage{enumerate}

\usepackage{bbm}
\usepackage{float}
\usepackage{bm}
\usepackage[mathscr]{eucal}
\usepackage{graphicx}
\usepackage{sidecap}
\usepackage{latexsym}
\usepackage{psfrag}
\usepackage{epsfig}
\usepackage{enumerate}
\usepackage{amsmath}
\usepackage{amsthm}
\usepackage{amssymb}
\usepackage{multirow}
\usepackage{comment}
\usepackage{amsfonts}
\usepackage{color}
\usepackage[textsize=small]{todonotes}
\usepackage{hyperref}
\usepackage{mathtools}
\mathtoolsset{showonlyrefs}

\numberwithin{equation}{section}
\numberwithin{figure}{section}
\numberwithin{table}{section}
\newtheorem{theorem}{Theorem}[section]

\newtheorem{lemma}[theorem]{Lemma}
\newtheorem{construction}[theorem]{Construction}
\newtheorem{property}[theorem]{Property}

\newtheorem{assumption}[theorem]{Assumption}
\newtheorem{proposition}[theorem]{Proposition}

\theoremstyle{remark}
\newtheorem{remark}[theorem]{Remark}

\numberwithin{equation}{section}

\newcommand{\be}{\begin{equation}}
\newcommand{\ee}{\end{equation}}

\def\bc{\begin{center}}
\def\ec{\end{center}}
\def\bea{\begin{eqnarray}}
\def\eea{\end{eqnarray}}

\newcommand{\set}[1]{\left\lbrace{#1}\right\rbrace}

\DeclarePairedDelimiterX\braket[2]{\langle}{\rangle}{#1 \delimsize\vert #2}
\DeclarePairedDelimiterX{\infdivx}[2]{(}{)}{%
  #1\;\delimsize\|\;#2%
}

\newcommand{\clp}{\mathcal{P}}
\newcommand{\clk}{\mathcal{K}}
\newcommand{\cll}{\mathcal{L}}
\newcommand{\clf}{\mathcal{F}}

\newcommand{\clt}{\mathcal{T}}
\newcommand{\clm}{\mathcal{M}}
\newcommand{\cls}{\mathcal{S}}
\newcommand{\clc}{\mathcal{C}}
\newcommand{\cla}{\mathcal{A}}

\newcommand{\N}{\mathbb{N}}
\newcommand{\Y}{\mathbb{Y}}
\newcommand{\B}{\mathbb{B}}
\newcommand{\E}{\mathbb{E}}
\newcommand{\R}{\mathbb{R}}
\newcommand{\X}{\mathbb{X}}
\newcommand{\PP}{\mathbb{P}}

\newcommand{\bfq}{\mathbf{q}}
\newcommand{\sos}{A}
\newcommand{\Om}{\Omega}
\newcommand{\om}{\omega}
\newcommand{\la}{\lambda}
\newcommand{\dd}{d}
\newcommand{\eps}{\varepsilon}

\newcommand{\al}{\alpha}
\newcommand{\cons}{\varkappa}

\newcommand{\flu}{\varsigma}
\newcommand{\IDW}{I^{\mbox{\tiny{DW}}}}

\definecolor{aqua}{rgb}{0.0, 1.0, 1.0}
\definecolor{boo}{rgb}{1.0, 0.0, 1.0}
\definecolor{stred}{rgb}{1.0, 0.2, 0.37}

\newcommand{\newtext}[1]{{\color{black}{#1}}}
\newcommand{\deletetext}[1]{}

\begin{document}

\title{Jump Processes with Self-Interactions: \\ Large Deviation Asymptotics}

\date{\today}

\author{Amarjit Budhiraja}

\author{Francesco Coghi}

\begin{abstract}
We consider a pure jump process $\{X_t\}_{t\ge 0}$ with values in a finite state space $S= \{1, \ldots, d\}$ for which the jump rates at time instant $t$ depend on the occupation measure $L_t \doteq \frac1t \int_0^t \delta_{X_s}\,ds$. Such self-interacting chains arise in many contexts within statistical physics and applied probability. Under appropriate conditions, a large deviation principle is established for the pair $(L_t, R_t)$, as $t \to \infty$, where $R_t$ is the empirical flux process associated with the jump process. We show that the rate function takes a simple form that can be viewed as a dynamical generalization of the classical Donsker and Varadhan rate function for the analogous quantities in the setting of Markov processes, in particular, unlike the Markovian case, the rate function is not convex. Since the state process is non-Markovian, different techniques are needed than in the setting of Donsker and Varadhan and our proofs rely on variational representations for functionals of Poisson random measures and stochastic control methods.\\ 

\smallskip


\smallskip


\end{abstract}

\maketitle

\section{Introduction}

We consider a pure jump process $\{X_t\}_{t\ge 0}$ with values in a finite state space $S= \{1, \ldots, d\}$ for which the jump rates at time instant $t$ depend on the occupation measure $L_t \doteq \frac1t \int_0^t \delta_{X_s}\,ds$. This dependence is specified through a collection of $d\times d$ rate matrices (infinitesimal generators) $\{Q(\gamma), \gamma \in \clp(S)\}$, where
$\clp(S)$ is the space of probability measures on $S$. Roughly speaking, at time instant $t$, conditionally on $\sigma\{X(s): 0 \le s \le t\}$, and given $X(t)=i$, the jump rate to state $j \neq i$ is   $Q_{ij}(L_t)$. A precise definition is given in terms of a certain martingale problem characterization; see Section \ref{sec:2}.  Such processes are usually referred to as self-interacting Markov jump processes and arise in many different contexts; see discussion at the end of this Introduction. We note however that the evolution of $\{X_t\}$ is not Markovian and to get a Markov state descriptor, one needs to consider the pair $\{(X_t, L_t)\}$ which defines a time inhomogeneous Markov process with state space $S\times \clp(S)$. We are interested in the large time behavior of the process under a large deviation scaling. The specific objects we study are the occupation measure process $\{L_t\}_{t\ge 0}$
and the $\R_+^{d(d-1)}$ valued {\em empirical flux} process $\{R_t\}_{t\ge 0}$ where
$R_t = (R_t^{xy}, x,y \in S, x\neq y)$ and $R_t^{xy}$ is the time averaged number of transitions from state $x$ to state $y$ over the time interval $[0,t]$; see \eqref{eq:flux} for a precise definition.

Our results give a large deviation principle (LDP) for $\{(L_t, R_t)\}_{t\ge 0}$,  in $\clp(S)\times \R_+^{d(d-1)}$, as $t\to \infty$ under suitable conditions on $Q(\cdot)$; see Section \ref{sec:2} and also, comments on these conditions later in the Introduction. In the large deviation literature such a result is referred to as a Level $2.5$ LDP, and, in comparison to an LDP for the occupation measure alone, it captures important dynamic details of the process by keeping track of not only how often different states are visited but also how often different types of state transitions occur. Through a contraction principle, such a result also gives an LDP for the $\R^{d(d-1)}$ valued {\em empirical current} process $\{J_t\}_{t\ge 0}$, where $J_t = (J_t^{xy}, x,y \in S, x\neq y)$ and $J_t^{xy} = R_t^{xy}- R_t^{yx}$. 

In the setting where $Q(\gamma) \equiv Q^0$ for all $\gamma \in \clp(S)$, for some irreducible rate matrix $Q^0$, these large deviation results are classical from the work of Donsker and Varadhan \cite{donvar1,Bertini2015}.
The rate function, $\IDW_{Q^0}: \clp(S) \times \deletetext{\R_+^{(d(d-1)}}\newtext{\R_+^{d(d-1)}} \to [0,\infty]$, governing the LDP for $(L_t, R_t)$ in this Markovian setting is given as
\be
\IDW_{Q^0}(\gamma, \flu) =
\begin{cases}
\sum_{x\neq y} \gamma_x Q^0_{xy} \ell\left(\frac{\flu^{xy}}{\gamma_xQ^0_{xy}}\right) & \mbox{ if } \sum_{y: y \neq x}\flu^{xy} = \sum_{y: y \neq x} \flu^{yx} \mbox{ for all } x \in S\\
\infty & \mbox{otherwise},
\end{cases}
\ee
where $(\gamma, \flu) \in \clp(S)\times \deletetext{\R_+^{(d(d-1)}}\newtext{\R_+^{d(d-1)}}$ and $\ell(x) = x\log x-x +1$ for $x \ge 0$.

For the self-interacting case studied in this work, the rate function we obtain can be viewed as a certain dynamical version of the above rate function. A precise definition is in \eqref{eq:Alternative} (see also \eqref{eq:IGamma} for a different formula), however roughly speaking, the rate function we obtain takes the form
\be\label{eq:155nn}
I(\gamma, \flu) = \inf \int_0^{\infty} e^{-s} \IDW_{{Q(M(s))}}(\rho(s), \eta(s)) \, ds, \; (\gamma, \flu) \in \clp(S) \times \R_+^{d(d-1)},\ee
where the infimum is taken over all maps $(\rho, \eta): \R_+ \to \clp(S) \times \R_+^{d(d-1)}$ with $\gamma = \int_0^{\infty} e^{-s} \rho(s) ds$,
$\flu = \int_0^{\infty} e^{-s} \eta(s) ds${, and $M(t) = e^t \int_t^\infty e^{-s} \rho(s) \, ds$}.

An infinite horizon discounted cost in the definition of the rate function arises in a natural fashion from the evolution equation for the empirical measure process (see e.g. \eqref{eq:TimeEvolEmpOccMeas}) which says  that changes to the empirical measure become vanishingly small as $t$ becomes large. 
{These changes of order $ t^{-1}$ ensure that contributions to the rate function at large times are suitably discounted, as captured by the exponential factor in the rate function.}  In the Markovian setting it turns out that the infimum over $(\rho, \eta)$ as above is achieved with $\rho_s \equiv \gamma$ and $\eta_s \equiv \flu$ and thus the above formula for the rate function reduces to the formula given by the Donsker-Varadhan theory (see Section \ref{sec:examples}).  The form of the rate function above also captures a certain multiscale feature of the system which says that the state process $\{X_t\}$ rapidly equilibrates locally as the environment described by $\{L_t\}$ changes slowly. These local equilibria are seen in the form of the local rate function $\IDW_{{Q(M(s))}}(\gamma(s), \eta(s))$ that defines the overall rate function $I$.

For  discrete time self-interacting finite state Markov chains, similar large deviation behavior has recently been observed in \cite{Budhiraja2025} (see also \cite{BudWat3}). These works do not consider the empirical flux process but if one restricts to the LDP for the empirical measure process $\{L_t\}$ (obtained by an application of the contraction principle), the rate function we obtain in the current work has analogies with the rate functions in these works (e.g. in terms of an exponential discount factor and  multiscale feature of the dynamics). The proof techniques in the current work, which rely on variational formulae for moments of nonnegative functionals of Poisson random measures (PRM),  are however quite different.

We now give some background for the self-interacting models studied in this work.
Self-interacting stochastic processes, such as \textit{reinforced} random walks, have been an active area of research for several decades. Since the mid-1980s, both probabilists and statistical physicists have been attracted by the properties of edge-reinforced~\cite{Coppersmith1986} and vertex-reinforced random walks~\cite{Amit1983,Pemantle1988,Pemantle1992} in discrete time, originally motivated by models of growing polymers (chains with excluded volume interaction). First continuous-time extensions appeared in~\cite{Cranston1996,Raimond1997}, based on non-normalised occupation measures, followed by a series of works  that established self-interacting diffusions in compact spaces~\cite{Benaim2002,Benaim2003,Benaim2005,Benaim2011} and later extensions to general open domains~\cite{Kurtzmann2010,Chambeu2011,Kleptsyn2012}.  

Other examples of discrete-time self-interacting processes include generalised P\'{o}lya urns, which have been used to model genetic drift in biological species~\cite{Hoppe1984,sinliv,sch01} as well as novelty dynamics~\cite{Tria2014,Iacopini2020} within the emerging field of the science of science. A further motivation for studying self-interacting processes, in particular those depending on the \emph{normalised} empirical occupation measure as in this work, comes from applications to quasi-stationary distributions (QSD) of finite-state Markov chains, where reinforced Monte Carlo schemes have been employed~\cite{AFP88}, and from more general settings in sampling algorithms~\cite{Diaconis2000,Essler2024}.  

In statistical physics, self-interacting Markov jump processes are attracting interest as a natural setting in which to extend the framework of stochastic thermodynamics~\cite{Seifert2012,Peliti2021} beyond the Markovian limit. {Although non-Markovian formulations of stochastic thermodynamics have been approached in several ways~\cite{Zwanzig1973,Lapolla2019,Esposito2012,Kanazawa2024,Brandner2025a,Brandner2025b}, self-interacting processes seem to provide a distinct, physically motivated class where memory effects can be analyzed systematically, as shown in this work. Such memory effects arise naturally when 
addressing coarse-graining and measurement limitations~\cite{Mori1965,Schilling2022}. When hidden degrees of freedom or limited temporal resolution cause the system’s observed evolution to depend on its own past---through quantities such as empirical occupation measures---the dynamics become effectively self-interacting. The theoretical framework developed here provides a starting point to analyze and quantify emergent non-Markovian behavior, revealing how microscopic information loss gives rise to macroscopic memory and irreversibility.}

Most of the research on such self-interacting systems has  focused on the study of  the law of large numbers and central limit theorems, often via stochastic approximation theory and branching-process techniques, see e.g.~\cite{AFP88,benclo,BenClo18,blanchet2016,budfrawat2}. {As previously mentioned,} large deviation properties of self-interacting discrete-time Markov chains have only recently been addressed~\cite{BudWat3,Budhiraja2025}. The latter paper considers a finite state discrete time model and proves an LDP under quite general conditions on the transition mechanism that cover  many  types of self-interacting Markov chains such as the chains used for QSD approximations\cite{AFP88}, certain variants of edge reinforced and vertex reinforced random walks, a type of personalized PageRank algorithm, and certain generalized Pólya urn schemes.
For related literature on large deviations for discrete time models with reinforcement see \cite{Franchini17,HuaLiuKai,HuaLiuKai2,Zhang14}. The current paper is the first work to study large deviation behavior of general self-interacting Markov chains in continuous time.

We now make some comments on the assumptions that are made in our work. For the upper bound the only condition we require is that the map $\gamma \to Q(\gamma)$ is a continuous map.
The lower bound requires additional conditions that are summarized in Assumption \ref{assu:lowbd}. The four parts of this assumption are analogous to conditions that are needed in the analysis of the discrete time model in \cite{Budhiraja2025}, the first condition requires a strictly positive solution to the equation $\pi^*Q(\pi^*)=0$; the second 
says that  $Q$ is an affine map; the third imposes a natural communicability structure on the rate matrices $\{Q(\gamma)\}$; and the fourth condition says that the empirical measure $L_t$ eventually charges all points $S$, a.s. 
See Remark \ref{rem:whenhold} for a discussion on these conditions.

Next we briefly discuss the proof techniques.  As noted earlier in the Introduction, the main tool used in our proofs is the variational representation for functionals of PRM obtained in \cite{budhiraja2011variational} (see also \cite[Chapter 8]{buddupbook}). This variational representation is well suited to establish the equivalent form of the large deviation principle given in terms of certain Laplace asymptotic formulas (see \cite[Corollary 1.10]{buddupbook}). Specifically, the variational formula for PRM gives certain stochastic control representations for the quantities of interest in the Laplace asymptotics (namely the left sides of \eqref{eq:Laplaceupp} and \eqref{eq:Laplacelow}). With these stochastic control representations the main step in the proof is establishing convergence of the value functions of these stochastic control problems to the optimal value of certain deterministic optimization problems associated with the rate function. One important insight is obtained by considering the scaling $t \mapsto e^t$, which, together with a time reversal step leads to dynamics and cost in the stochastic control problems that roughly resemble analogous terms in the limiting deterministic control problem (e.g. compare the evolution of the scaled and time reversed dynamics in \eqref{eq:302} with the identity in 
Property \ref{proper:1}(b) and the cost on the right side of \eqref{eq:n1029} with the cost in \eqref{eq:JTheta}). This scaling and time reversal plays an important role in the proof.  The proof is completed in two parts. 
Section \ref{sec:lapupp} proves the upper bound, namely the inequality in \eqref{eq:Laplaceupp}, while the lower bound proof (i.e. the proof of \eqref{eq:Laplacelow}), which is technically the most demanding part of the work, is completed in Sections \ref{sec:laplow} and \ref{sec:construction}.  The upper bound relies on tightness and weak limit point characterizations of certain sequences of random measures constructed from the controlled state processes and control processes. For this inequality it is convenient to consider a different representation of the rate function than the one discussed earlier in the Introduction and which is more directly connected with the weak limits of the 
costs expressed as integrals with respect to the above random measures.
The lower bound requires construction of certain feedback controls that are asymptotically near optimal for the variational problem associated with the first representation of the rate function, namely, the one that appears in \eqref{eq:Alternative}. This construction requires ensuring that the feedback controls have certain continuity and nondegeneracy properties which are achieved by a series of approximations that are carried out in Sections \ref{sec:nondeg} and \ref{sec:ctycont}. Section \ref{sec:construction} then completes the proof of the lower bound by providing a construction of the feedback controls and proving the convergence of controls, state processes, and costs to the desired limits. Section \ref{sec:cptlev} \deletetext{show}\newtext{shows} that our proposed rate function is indeed a rate function, namely, it has compact sublevel sets. Finally Section \ref{sec:examples} contains some discussion and illustrative examples.

\subsection{Notation}
 Some standard notation that we use is as follows.  For some space $\X$ and a function $h: \X \to \R$, we will denote, $\sup_{x \in \X} |h(x)| =: \|h\|_{\infty}$. The space of probability measures  on a Polish space $\X$ will be denoted as $\clp(\X)$  and will be equipped with the topology of weak convergence. For Polish spaces $\X_1, \X_2, \ldots, \X_k$ and  a $\gamma \in \clp(\X_1\times \X_2 \times \cdots \X_k)$, we define, for $i=1,2, \ldots, k$, $[\gamma]_i \in \clp(\X_i)$ as $[\gamma]_i(\cdot) = \gamma(\X_1 \times \ldots \times \X_{i-1} \times \cdot \times \X_{i+1} \times \ldots \X_k)$.  For $1\le i< j \le k$, $[\gamma]_{ij} \in \clp(\X_i\times \X_j)$ is defined similarly and analogous notation will be used for measures on product of more than two spaces obtained from $\gamma$ by integrating over suitable coordinates.
 The space of Borel measurable maps from topological space $\X_1$ to $\X_2$ will be denoted as $\B(\X_1, \X_2)$. When $\X_2 = \R$, we may sometimes simply write $\B(\X_1)$.
 For a Polish space $\X$, $\clc(\R_+\!:\!\X)$ will denote the space of continuous functions from $\R_+$ to $\X$, equipped with the topology of uniform convergence on compacts. We will denote by  $C_b(\X)$  the space of real continuous and bounded functions on $\X$. For a function $f:\R_+ \to \X$ (resp. $f:\R_+\times \Y \to \X$), we will use the notation $f_s$ (resp. $f_s(y)$) and $f(s)$ (resp. $f(s,y)$) interchangeably to denote the evaluation of $f$ on $s \in \R_+$ (resp. $(s,y) \in \R_+\times \Y)$.
 
\section{Model  and Notation}\label{sec:2}
We will consider the following family of self-interacting Markov jump processes.
Let $S = \set{1,2,\cdots, \dd}$ be the finite set of states of the process.
Recall that $\clp(S)$ is the space of probability measures on $S$ which will be identified with the $\dd-1$ dimensional simplex in $\mathbb{R}_+^{\dd}$, namely
$$\clp(S) =\set{x \in \mathbb{R}_+^d \, | \, x \cdot \mathbf{1} = 1},$$ 
equipped with the $L^1$-distance: for $x,y \in \clp(S)$, $\|x-y\| \doteq \sum_{i=1}^d|x_i-\deletetext{y_1}\newtext{y_i}|$.
Throughout, we will denote $\{(i,j) \in S \times S: i \neq j\} = \sos$.
 Let $\clk_0(S)$ be the collection of rate matrices on $S$, namely the space of $\dd\times \dd$ matrices $\bfq = (q_{ij})_{i,j \in S}$ satisfying
$q_{i,j}\ge 0$ for $(i,j) \in \sos$ and $q_{ii} = - \sum_{j:j\neq i} q_{ij}$, $i \in S$. We denote by $\clk(S)$ the subset of $\clk_0(S)$
consisting of rate matrices that are
irreducible , i.e., for any $i,j \in S$, there exist $m \in \N$ and distinct $i_1, \ldots, i_m \in S$ with $q_{i_k, i_{k+1}} > 0$ for all $k = 1, \ldots, m-1$, $q_{i, i_1}>0$ and $q_{i_m,j}>0$.

Let $Q: \clp(S) \to \clk_0(S)$ be a continuous map. From the compactness of $\clp(S)$ and continuity of $Q$ it follows that 
\begin{equation}\label{eq:cqdefn}
\sup_{\gamma \in \clp(S)} \max_{(i,j) \in \sos} Q_{ij}(\gamma) = c_{Q} <\infty.\end{equation}
We will occasionally suppress $Q$ and write $c_Q$ as $c$.

We are interested in an $S$ valued self-interacting Markov jump process, $\{X_t\}_{t\ge 0}$, starting with some state $x\in S$, which is uniquely characterized in law by the following properties: For all  $f: S\to \R$
$$f(X_t) - f(x) - \int_0^t \cll_{L_s} f(X_s) ds, t \ge 0,$$
is a $\clf_t = \sigma\{X_s: 0 \le s \le t\}$ martingale, where
\begin{equation}
        \label{eq:EmpOccMeas}
        L_t = \frac{1}{t} \int_0^t \delta_{X_s}  ds \, , t > 0, \; L_0 = \delta_x \, ,
    \end{equation}
and for $\gamma \in \clp(S)$, 
\begin{equation}
        \label{eq:Generator}
        \mathcal{L}_{\gamma} f(i) = \sum_{j: j \neq i} Q_{ij}(\gamma) (f(j) - f(i)) \, , i \in S.
    \end{equation}

Such a process can be constructed pathwise as follows. Consider a probability space
$(\Om, \clf, \PP)$ on which are given $ a\doteq d(d-1)$ mutually independent Poisson random measures (PRM) $\{N_{ij}, (i,j) \in \sos\}$ with point space as $\R_+\times [0,c] \times \R_+$ with intensity given as the Lebesgue measure $\la(ds\, du\, dr) = ds\, du\, dr$.

Define  an $S \times \clp(S)$ valued stochastic process $(X_t, L_t)_{t\ge 0}$ by the following system of equations.
\begin{equation}
        \label{eq:RepresentationSIJump}
        X_t = x + \sum_{(i,j) \in \sos} \int_{[0,t] \times [0,c]\times \R_+} (j-i) 1_{[X_{s-}=i]} 1_{[0, Q_{ij}(L_{s-})]} (u) 1_{[0,1]}(r) N_{ij}(ds\, du\, dr) \, ,
    \end{equation}
    with $\{L_t\}$ defined by  \eqref{eq:EmpOccMeas}.

    Together with the above processes, we will be interested in the 
    $\R_+^{a}$-valued empirical flux process $R_t = (R^{ij}_t)_{(i,j) \in A}$, where $R^{ij}_t$ gives the (time averaged) number of transitions from state $i$ to state $j$ during the time interval $[0,t]$. This process can be described using the above collection of PRM as follows: For $(i,j) \in A$,
   \begin{equation}
        \label{eq:flux}
        R_t^{ij} =  \frac{1}{t}\int_{[0,t] \times [0,c]\times \R_+}  1_{[X_{s-}=i]} 1_{[0, Q_{ij}(L_{s-})]} (u) 1_{[0,1]}(r) N_{ij}(ds\, du\, dr) \, . 
    \end{equation}
    We note that the process $(X_t, L_t, R_t)$ can be constructed recursively from one jump to next by a standard procedure and this process satisfies the martingale property noted above.  We also remark that one can give a  construction of these processes 
   using simpler PRM (e.g. those defined on $\R_+\times [0,c]$) or by a collection of time changed standard Poisson processes, however the above construction is useful when studying the large deviation behavior as will be seen below.
\newtext{Specifically, the additional mark $r$ is introduced  for the large-deviation control representation (see \eqref{eq:RepresentationFormula}) in which control actions that correspond to  predictable changes of the jump intensities are implemented by replacing $1_{[0,1]}(r)$ with $1_{[0,\alpha_{ij}(s,u)]}(r)$ on the same underlying PRM. This variational representation is a central ingredient in our large deviation analysis.}

\section{Main Result}\label{sec:mainres}
The goal of the work is to establish an LDP for the collection $\{(L_t, R_t)\}$ in $\clp(S) \times \R_+^a$ \newtext{(recall that $a=d(d-1)$)}, as $t \to \infty$, with an appropriate rate function $I: \clp(S) \times \R_+^a \to [0,\infty]$. We recall that such a map is called a rate function if it has compact sublevel sets, i.e. for every $m<\infty$, $\{(\gamma, \flu) \in \clp(S)\times \R_+^a: I(\gamma, \flu) \le m\}$ is a compact subset of $\clp(S)\times \R_+^a$.  

We now introduce the rate function $I$ that will govern the large deviation behavior of $\{(L_t, R_t)\}$ as $t\to \infty$.
The fact that $I$ has compact sublevel sets will be shown in Section \ref{sec:cptlev}.

 Given $\gamma \in \clp(S)$ and $\flu \in \R_+^a$, let $\theta \in \B(\R_+, \clp(S\times \R^a_+\times [0,c]))$ be such that the following properties hold.
\begin{property}\label{proper:1}
$\,$
\begin{enumerate}\deletetext{[(a)]}
\item[\newtext{(a)}] We have
$\gamma = \int_{0}^{\infty} e^{-s} [\theta(s)]_1 ds$.

\item[\newtext{(b)}] Let
$$M(t) \doteq e^t \int_{t}^{\infty} e^{-s} [\theta(s)]_1 ds, \; t \ge 0.$$
Then
$$
\int_{S\times \R^a_+\times [0,c]} \sum_{y: y \neq x}  1_{[0, Q_{xy}(M_s)]}(u) r_{xy} \theta(s)(dx\, dr\, du) <\infty, \mbox{ for a.e. } s \ge 0
$$
and
for all $f\in \B(S,\R)$,
\be
\int_{S\times \R^a_+\times [0,c]} \sum_{y: y \neq x} (f(y) - f(x)) 1_{[0, Q_{xy}(M_s)]}(u) r_{xy} \theta(s)(dx\, dr\, du) = 0, \mbox{ for a.e. } s \ge 0.
\ee
\item[\newtext{(c)}] For a.e. $s\ge 0$, $[\theta(s)]_{13}(dx\, du) = [\theta(s)]_{1}(dx) c^{-1}du$.

 \item[\newtext{(d)}] Writing $\flu = (\flu^{xy})_{(x,y) \in A}$, for $(x,y) \in A$,
\be
c\int_{\R^a_+\times [0,c]\times \R_+} e^{-s} [\theta(s)]_1(x) 1_{[0, Q_{xy}(M_s)]}(u) r_{xy} \eta(s)(dr\, du\mid x) ds= \flu^{xy},
\ee
where $M_s$ is as in part (b) and $\theta(s)(dx\, dr\, du) = \eta(s)(dr\, du\mid x) [\theta(s)]_{1}(dx)$.

\end{enumerate}
\end{property}
 We denote the collection of all such $\theta$ as $\cls_{\gamma, \flu}$, namely,
\be
\cls_{\gamma, \flu} = \{\theta \in \B(\R_+, \clp(S\times \R^a_+\times [0,c])): \mbox{Property \ref{proper:1} holds} \}.
\ee

For $\theta \in \cls_{\gamma,\flu}$, the $M \in \clc(\R_+: \clp(S))$ 
defined in part (b) of \deletetext{Condition}\newtext{Property} \ref{proper:1} will be denoted as
 $M^{\theta}$.
Let $\Xi \doteq S \times \mathbb{R}_+^a \times [0,c]\times \mathbb{R}_+$.
For $\theta \in \B(\R_+, \clp(S\times \R^a_+\times [0,c]))$, let
\be
\label{eq:JTheta}
J(\theta) = c\int_{\Xi} e^{-s} \sum_{y:y\neq x}1_{[0, Q_{xy}(M_s)]}(u) \ell(r_{xy})\,  \theta(s)(dx\, dr\, du) ds \, ,
\ee
where
\begin{equation}
    \label{eq:PoissonFunction}
    \ell(x) = x \log x - x + 1, \, \, x \geq 0 \, .
\end{equation}
 Then the rate function is defined as
\be
\label{eq:IGamma}
I(\gamma,\flu) = \inf_{\theta \in \cls_{\gamma,\flu}} J(\theta), \;\; \gamma \in \clp(S), \flu \in \R_+^a.
\ee
See also a different representation for the rate function in \eqref{eq:Alternative}.

We can now state the main result of this work. For the upper bound we do not assume any additional conditions beyond the continuity of the map $Q$ assumed previously. For the lower bound we will need the following additional assumption. We denote by $\clp_+(S)$ the collection of $m \in \clp(S)$ such that $m(x)>0$ for all $x \in S$.
\begin{assumption}\label{assu:lowbd}
\begin{enumerate}\deletetext{[(a)]} $\,$
\item[\newtext{(a)}] There is a $\pi^* \in \clp_+(S)$ such that $\pi^*Q(\pi^*)=0$.
\item[\newtext{(b)}] For all $\kappa \in [0,1]$ and $m_1, m_2 \in \clp(S)$, $Q(\kappa m_1 + (1-\kappa)m_2) = \kappa Q(m_1) + (1-\kappa) Q(m_2)$.
\item[\newtext{(c)}]
There is an $\cla^* \in \clk(S)$ such that, with $A_0 = \{(x,y) \in A: \cla^*(x,y) >0\}$, $Q_{ij}(m) = 0$ for all $m \in \clp(S)$ and $(i,j) \in A\setminus A_0$. Furthermore, for some $k_Q \in (0,\infty)$, and all $(i,j) \in A_0$ and
$m \in \clp(S)$,
$$Q_{ij}(m) \ge k_Q (\inf_{x\in S} m(x)).$$
\item[\newtext{(d)}] For each $x \in S$
$$\PP(\om \in \Om: \mbox{ for some } t> 0, L_t(\om)\{x\}>0 ) = 1.$$
\end{enumerate}
\end{assumption}
\begin{remark}\label{rem:whenhold}
The affine property in part (b) is satisfied in many settings and we refer the reader to Section \ref {sec:examples} for some illustrative examples. The first statement in part (c) in the assumption is a natural communication condition which specifies the permissible transitions in the model, e.g. the jump processes of interest may be nearest neighbor walks so that only transitions to the immediate neighbors are allowed.
If part (b) of the assumption holds, the second part of part (c) will hold as long as there is some $m \in \clp(S)$ such that $Q_{ij}(m)>0$ for all $(i,j) \in A_0$.  This follows from noting that, for $p \in \clp(S)$,
\begin{multline} \inf_{(i,j) \in A_0} Q_{ij}(p) = \inf_{(i,j) \in A_0} \sum_{x \in S} p_x Q_{ij}(\delta_x)
\ge \inf_{z \in S} p_z \inf_{(i,j) \in A_0} \sum_{x \in S}  Q_{ij}(\delta_x)\\
\ge \inf_{z \in S} p_z \inf_{(i,j) \in A_0} \sum_{x \in S} m_x Q_{ij}(\delta_x)=
\inf_{z \in S} p_z \inf_{(i,j) \in A_0}  Q_{ij}(m)
= k_Q \inf_{z \in S} p_z,
\end{multline}
where $k_Q = \inf_{(i,j) \in A_0}  Q_{ij}(m)>0$.
Suppose that we have the stronger property that $Q_{ij}(m)>0$ for all $m \in \clp(S)$ and all $(i,j) \in A_0$. In this case we have that (a) holds as well. This can be seen as follows. Define $G:\clp(S) \to \clp(S)$ as, for $m\in \clp(S)$, $G(m) = \gamma$, where $\gamma$ is the unique element in $\clp(S)$ such that $\gamma Q(m) =0$. From the above irreducibility and affine assumptions we see that the map $G: \clp(S) \to \clp(S)$ is continuous. Thus by Brouwer's fixed point theorem, $G$ has a fixed point $\pi^* \in \clp(S)$. This $\pi^*$ must satisfy $\pi^*Q(\pi^*)=0$. Since $Q(\pi^*)$ is irreducible, it then follows that $\pi^* \in \clp_+(S)$, showing that (a) holds.
In this setting (and assuming that (b) holds), an argument as in \cite[Lemma A.2]{Budhiraja2025} shows that part (d) of the assumption holds as well.  
\end{remark}

\begin{theorem}\label{thm:main}
The function $I$ is a rate function, namely it has compact sublevel sets.
As $t \to \infty$, the collection $\{(L_t, R_t), t>0\}$ satisfies the LDP upper bound with speed $t$ and rate function $I$, namely for all closed sets $F \in \clp(S) \times \R_+^a$,
$$\limsup_{t\to \infty} \frac{1}{t} \log \PP\{(L_t, R_t) \in F\} \le -\inf_{z\in F} I(z).$$
Suppose in addition that Assumption \ref{assu:lowbd} is satisfied. Then, as $t \to \infty$, the collection $\{(L_t, R_t), t>0\}$ also satisfies the LDP lower bound with speed $t$ and rate function $I$, namely for all open sets $G \in \clp(S) \times \R_+^a$,
$$\liminf_{t\to \infty} \frac{1}{t} \log \PP\{(L_t, R_t) \in G\} \ge -\inf_{z\in \deletetext{F}\newtext{G}} I(z).$$
\end{theorem}
\begin{proof}
The LDP is completed in two steps.  The upper bound is established in Section \ref{sec:lapupp} and the complementary lower bound is shown in Sections \ref{sec:laplow} and \ref{sec:construction}. The fact that $I$ is a rate function is shown in Section \ref{sec:cptlev}. 
\end{proof}
We remark that Theorem \ref{thm:main} immediately yields, through the contraction principle,  an LDP for the $\R^{d(d-1)}$ valued {\em empirical current} process $\{J_t\}_{t\ge 0}$, where $J_t = (J_t^{xy}, x,y \in S, x\neq y)$ and $J_t^{xy} = R_t^{xy}- R_t^{yx}$. We omit the details.

\section{Laplace Asymptotics and  Variational Representation}
\label{sec:varrep}
For proving the large deviation bounds, it will be convenient to consider the following reparametrization. Let $t_n \to \infty$ be an arbitrary sequence and let $T_n= e^{t_n}$. It then suffices to show the bounds in Theorem \ref{thm:main} with $t$ (resp. $t\to \infty$) replaced by $T_n$ (resp. $n\to \infty$).
Further, from equivalence between an LDP and Laplace principle (see e.g. \cite[Theorems 1.5 and 1.8]{buddupbook}), it suffices, for the large deviation upper bound, to show  for all $h \in C_b(\clp(S)\times \R_+^a)$,
the  following Laplace upper bound 
\begin{equation}
        \label{eq:Laplaceupp}
         \liminf_{n \rightarrow \infty} -\frac{1}{T_n} \log \mathbb{E} \left[ e^{- T_n h(L_{T_n}, R_{T_n})} \right] \ge \inf_{(\gamma, \flu) \in \mathcal{P}(S) \times \R_+^a} \left[ h(\gamma,\flu) + I(\gamma,\flu) \right].
    \end{equation}
    This will be proved in Section \ref{sec:lapupp}.
    Similarly, for the large deviation lower bound, it suffices to show  that, for all $h \in C_b(\clp(S)\times \R_+^a)$,
    \begin{equation}
        \label{eq:Laplacelow}
         \limsup_{n \rightarrow \infty} -\frac{1}{T_n} \log \mathbb{E} \left[ e^{- T_n h(L_{T_n}, R_{T_n})} \right] \le \inf_{(\gamma, \flu) \in \mathcal{P}(S) \times \R_+^a} \left[ h(\gamma,\flu) + I(\gamma,\flu) \right].
    \end{equation}
This Laplace lower bound will be shown in Sections \ref{sec:laplow} and \ref{sec:construction}.
    
    The starting point of our proof will be the following variational formula. 
    Let $\clp_2$ be the space of predictable maps $\alpha: \R_+\times [0,c] \times \Omega \to \R_+$. Namely $\alpha$ is a $(\B([0,c])\otimes \clt_{pr}) / \B(\R_+)$ measurable, where $\clt_{pr}$ is the sub $\sigma$-field of $\B(\R_+)\otimes \clf$ generated by real $\{\clf_t\}$ predictable processes, where $\clf_t = \sigma\{N_{ij}([0,s]\times B): 0 \le s \le t, B \in \B([0,c]\times \R_+), (i,j) \in \sos\}$. We will denote by $\clp_{2,b}$ the collection of $\al \in \clp_2$ such that  there is an $m \in (0,\infty)$ (depending on $\al$) such that $\al(t,u,\om) \in [m^{-1}, m]$ for all $(t,u) \in \R_+\times [0,c]$, for a.e.\ $\om$,.
    
    We will denote by $\Theta^n$ the collection of all $\alpha^n= (\alpha^n_{ij})_{(i,j) \in A}$ such that for each $(i,j) \in A$, $\alpha^n_{ij}: [0,T_n]\times [0,c] \times \Om \to \R_+$ is a restriction of a predictable map in $\clp_{2,b}$ to the time interval $[0,T_n]$. 
    
    For $\alpha^n \in \Theta^n$, define $\{(\tilde X^n_t, \tilde L^n_t), t \le T_n\}$ by the controlled analogues of \eqref{eq:EmpOccMeas}-\eqref{eq:RepresentationSIJump}: for $0 \le t \le T_n$,
    \begin{equation}
        \label{eq:ControlProcess}
        \begin{split}
        \tilde{X}^n_t &= x + \sum_{(i,j) \in \sos} \int_{[0,t] \times [0,c] \times \mathbb{R}_+} (j-i) 1_{[\tilde{X}^n_{s-}=i]} 1_{[0, Q_{ij}(\tilde{L}^n_{s-})]}(u) 1_{[0, \alpha^n_{ij}(s,u)]}(r) N_{ij}(ds\, du\, dr),\; \\
        \end{split}
    \end{equation}
    where $N_{ij}(ds\, du\, dr)$ are defined as before (see below  \eqref{eq:Generator}), and
    \begin{equation}
            \label{eq:ContrEmpOccMeas}
            \tilde{L}^n_t = \frac{1}{t} \int_0^t \delta_{\tilde{X}^n_s} \, ds, \; 0< t \le T_n, \; \tilde L^n_0 = \delta_x.
        \end{equation}
       Let $\tilde R^n = (\tilde R^{n, ij})_{(i,j) \in A}$ {the controlled analogue of \eqref{eq:flux}}, where
       \begin{equation}
        \label{eq:ControlFlux}
        \tilde{R}^{n,ij} =  \frac{1}{T_n}\int_{[0,T_n] \times [0,c] \times \mathbb{R}_+}  1_{[\tilde{X}^n_{s-}=i]} 1_{[0, Q_{ij}(\tilde{L}^n_{s-})]}(u) 1_{[0, \alpha^n_{ij}(s,u)]}(r) N_{ij}(ds\, du\, dr),\;  (i,j) \in A.
    \end{equation}  
        The dependence of $(\tilde X^n_t, \tilde L^n_t, \tilde R^n)$ on the control process $\alpha^n$ is suppressed in the notation.

        Then we have (see \cite[Theorem 8.12]{buddupbook}),
       \begin{equation}
        \label{eq:RepresentationFormula}
        \begin{split}
        - \frac{1}{T_n} \log \mathbb{E} \left[ e^{- T_n h({L}_{T_n}, {R}_{T_n})} \right] &= \\
        &\hspace{-3cm}\inf_{\substack{\alpha^{n} \in \Theta^n}} \mathbb{E} \left[ h(\tilde{L}^n_{T_n}, \tilde{R}^n) + \sum_{(i,j)\in \sos} \frac{1}{T_n} \int_{[0,T_n] \times [0,c]} \ell(\alpha^{n}_{ij}(s,u)) \, dsdu \right] \, .
        \end{split}
    \end{equation}

    \subsection{Rescaling and time-reversed    representation}
    For the weak convergence arguments that we will need, it will be convenient to consider the following rescaling and time reversal of the various controls and control processes.
    
Let, for $t\in \R$, 
\begin{equation}
        \label{eq:RescaledControlEmpOccMeas}
         \tilde{L}_{t}^{\text{sc},n} \doteq \tilde{L}^n_{e^{t}} , \;\; 
         \tilde{X}_{t}^{\text{sc},n}\doteq \tilde{X}^n_{e^{t}}.
\end{equation}
    Then, recalling that $T_n = e^{t_n}$, the right side of \eqref{eq:RepresentationFormula} can be rewritten as
\begin{equation}
            \label{eq:RescaledRepresentationFormula}
            \inf_{\substack{\alpha^{n} \in \Theta^n}} \mathbb{E} \left[ h(\tilde{L}^{\text{sc},n}_{t_n}, \tilde{R}^{n}) + \sum_{(i,j) \in A} e^{-t_n} \int_{[0,e^{t_n}] \times [0,c]} \ell(\alpha^{n}_{ij}(s,u)) \, dsdu \right].
        \end{equation}
        Using a change of variable in the time integral, $s = e^v$, we can write this as
        \begin{equation}
            \label{eq:RescaledChangeRepresentationFormula}
             \inf_{\substack{\alpha^{n} \in \Theta^n}} \mathbb{E} \left[ h(\tilde{L}^{\text{sc},n}_{t_n}, \tilde{R}^{n}) + \sum_{(i,j) \in A} e^{-t_n} 
             \int_{
             (-\infty,t_n] \times [0,c]} e^v \ell(\alpha^{\text{sc},n}_{ij}(v,u)) \, dv\,du \right] \, ,
        \end{equation}
       where 
       \begin{equation}
           \label{eq:ControllerRes}
            \alpha_{ij}^{\text{sc},n}(v,u) \doteq \alpha^n_{ij}(e^v,u), \, \, (v,u) \in (-\infty, t_n] \times [0,c].
       \end{equation}
         Denote for $0 \le t \le t_n$, 
        \begin{equation}
            \label{eq:RevResContEmpOccMeas} \tilde{L}^{\text{sc},n}_{t} \eqqcolon \tilde{M}^{\text{sc},n}_{t_n - t}
        \end{equation}    
        Then letting $w = t_n - v$ in the  integral in \eqref{eq:RescaledChangeRepresentationFormula}, we obtain the following  time-reversed and rescaled representation formula  from \eqref{eq:RepresentationFormula}.
        \begin{equation}
            \label{eq:RescaledChangeReverseRepresentationFormula}
            \begin{split}
        -\frac{1}{T_n} \log \mathbb{E} \left[ e^{- T_n h(L_{T_n}, R_{T_n})} \right] &= \\
        &\hspace{-3cm}\inf_{\substack{\alpha^{n} \in \Theta^n}} \mathbb{E} \left[ h(\tilde{M}^{\text{sc},n}_0, \tilde{R}^{n}) + \sum_{(i,j) \in A} \int_{[0,\infty) \times [0,c]} e^{-w} \ell(\hat{\alpha}_{ij}^{\text{sc},n}(w,u)) \, dw du \right] \, ,
        \end{split}
        \end{equation}
        where  
        \begin{equation}
            \label{eq:ControllerResInv}
            \hat{\alpha}_{ij}^{\text{sc},n}(w,u) \doteq \alpha_{ij}^{\text{sc},n}(t_n - w,u), \, \, w \in [0,\infty).
        \end{equation} 
         In the following, we will work with the representation formula \eqref{eq:RescaledChangeReverseRepresentationFormula}.

    \subsection{Dynamics of the time-reversed rescaled controlled empirical occupation measure}
    We now write the evolution for the scaled time-reversed process 
    $\{\tilde{M}^{\text{sc},n}_t, t >0\}$.
Note from \eqref{eq:ContrEmpOccMeas} that, for a.e.  $t >0$
\begin{equation}
            \label{eq:TimeEvolEmpOccMeas}
            \frac{d \tilde L^n_t}{dt} =  \frac{1}{t} (-\tilde L^n_t + \delta_{\tilde X^n_t}).
        \end{equation}
        Recalling the definition of $\tilde{L}^{\text{sc},n}_{t}$ in \eqref{eq:RescaledControlEmpOccMeas}, we now see that for a.e. $t \in \R$
        \begin{equation}
            \label{eq:TimeEvolRescContrEmpOccMeas}
            \frac{d \tilde{L}^{\text{sc},n}_t}{dt} = - \tilde{L}^{\text{sc},n}_t + \delta_{\tilde{X}^{\text{sc},n}_t} \, .
        \end{equation}
        In particular, for $t >0$,
        $$
        \tilde{L}^{\text{sc},n}_{t_n} - \tilde{L}^{\text{sc},n}_{t_n-t}
        = - \int_{t_n-t}^{t_n} \tilde{L}^{\text{sc},n}_s ds + 
        \int_{t_n-t}^{t_n} \delta_{\tilde{X}^{\text{sc},n}_s} ds.
        $$
        By reversing time, recalling the definition of $\tilde{M}^{\text{sc},n}_t$ in \eqref{eq:RevResContEmpOccMeas} and defining
        \be \label{eq:Ytilde}\tilde{Y}^{\text{sc},n}_{s}\doteq \tilde{X}^{\text{sc},n}_{t_n-s} \, 0\le s\le t, \ee we have
        $$\tilde{M}^{\text{sc},n}_{0} - \tilde{M}^{\text{sc},n}_{t}
        = - \int_{0}^{t} \tilde{M}^{\text{sc},n}_s ds + 
        \int_{0}^{t} \delta_{\tilde{Y}^{\text{sc},n}_s} ds,
        $$ 
namely,
\begin{equation} \label{eq:302}
 \tilde{M}^{\text{sc},n}_{t}
        = \tilde{M}^{\text{sc},n}_{0} + \int_{0}^{t} \tilde{M}^{\text{sc},n}_s ds - 
        \int_{0}^{t} \delta_{\tilde{Y}^{\text{sc},n}_s} ds, \;  0 \le t \le t_n.
\end{equation}
    The above evolution equation will be a key ingredient in our analysis.

    \section{Laplace upper bound}
    \label{sec:lapupp}
    In this section we will complete the proof of the upper bound in \eqref{eq:Laplaceupp}.
    \newtext{The proof has three main steps. We first choose nearly optimal controls in the variational representation and use the resulting cost bound to establish tightness of the associated occupation-control measures, time-reversed empirical occupation measures, and empirical fluxes. We then characterize every weak limit point. Specifically, Lemma \ref{lem:char} shows that the disintegration of the limit occupation-control measure satisfies the defining properties of $\cls_{\gamma,\flu}$ and also identifies the limiting empirical flux. Finally, lower semicontinuity of the control cost and the definition of $I$ give the desired Laplace upper bound. The estimates below are used for the main steps in the proof, namely, for  compactness and for characterizing the object that arises in the limit.}
    Fix $\delta \in (0,1)$. Then, using \eqref{eq:RescaledChangeReverseRepresentationFormula}
    we can find, for each $n \in \N$, $\alpha^n \in \Theta^n$ such that
    \be\label{eq:delbd}
    \begin{split}
    -\frac{1}{T_n} \log \mathbb{E} \left[ e^{- T_n h(L_{T_n},R_{T_n} )} \right] + \delta
    &\ge \\
    &\hspace{-3cm} \mathbb{E} \left[ h(\tilde{M}^{\text{sc},n}_0, \tilde{R}^{n}) + \sum_{(i,j) \in A} \int_{[0,\infty) \times [0,c]} e^{-s} \ell(\hat{\alpha}_{ij}^{\text{sc},n}(s,u)) \, ds du \right]
    \end{split}
\ee
where $\tilde{M}^{\text{sc},n}_s$, $\tilde{R}^{n}$, and $\hat{\alpha}_{ij}^{\text{sc},n}(s,u)$ as in \eqref{eq:RevResContEmpOccMeas}, \eqref{eq:ControlFlux}, and \eqref{eq:ControllerResInv}, respectively, are associated with the controls $\alpha^n$.

We thus have that
\be\label{eq:costbd}
\sup_{n \in \N} \mathbb{E} \left[\sum_{(i,j) \in A} \int_{[0,\infty) \times [0,c]} e^{-s} \ell(\hat{\alpha}_{ij}^{\text{sc},n}(s,u)) \, ds du \right] \le 2(\|h\|_{\infty} +1) \doteq K_0.
\ee
We now make an important observation.
From the definition of $(\tilde{M}^{\text{sc},n}_s, \tilde{R}^{n})$ (see \deletetext{\eqref{eq:ControlProcess}}\newtext{\eqref{eq:RevResContEmpOccMeas}} and \eqref{eq:ControlFlux}), we can replace $\hat{\alpha}_{ij}^{\text{sc},n}(s,u)$ with 
$$\hat{\alpha}_{ij}^{\text{sc},n}(s,u)1_{[0, Q_{ij}(\tilde{M}^{\text{sc},n}_s)]}(u)1_{\{\tilde{Y}^{\text{sc},n}_s =i\}} + (1- 1_{[0, Q_{ij}(\tilde{M}^{\text{sc},n}_s)]}(u)1_{\{\tilde{Y}^{\text{sc},n}_s =i\}}) $$
since this change does not affect  $(\tilde{M}^{\text{sc},n}, \tilde{R}^{n})$ and does not increase the cost.
Consequently, the right side of \eqref{eq:delbd} can be replaced by
\begin{equation}\label{eq:n1029}
\mathbb{E} \left[ h(\tilde{M}^{\text{sc},n}_0, \tilde{R}^{n}) + \sum_{(i,j) \in A} \int_{[0,\infty) \times [0,c]} e^{-s} 1_{\{\tilde{Y}^{\text{sc},n}_s =i\}}1_{[0, Q_{ij}(\tilde{M}^{\text{sc},n}_s)]}(u)\ell(\hat{\alpha}_{ij}^{\text{sc},n}(s,u)) \, ds du \right] \, .
\end{equation}
The above representation, although slightly more complicated than the right side of \eqref{eq:delbd}, is well suited for capturing the correct asymptotic behavior needed for the upper bound; see proof of Lemma \ref{lem:char}.

Recall $a = \dd(\dd-1)$ and that $\Xi = S \times \R_+^a \times [0,c] \times \R_+$. Let $\clp_1$ be the collection of all $\Gamma \in \clp(\Xi)$ such that 
$$[\Gamma]_{34}(du\,ds) = c^{-1}e^{-s} du ds.$$
Such a $\Gamma$ can be disintegrated as
$$\Gamma(dx\, dr\, du\, ds) = \Gamma_{(12)|(34)}(dx\, dr|u,s) [\Gamma]_{34}(du\, ds) =  \Gamma_{(12)|(34)}(dx\, dr|u,s) c^{-1}e^{-s} du\, ds, $$
    where for  a.e. $(u,s)$, $\Gamma_{(12)|(34)}(\cdot|u,s) \in \clp(S \times \R_+^a)$.
We note that $\clp_1$ is a closed subset of $\clp(\Xi)$.

    We now consider the following sequence $\{\nu^n, n \in \N\}$ of $\clp_1$ valued random variables:
    \be \label{eq:nundefn}
    \nu^n(dx\, dr\, du\, ds) := \delta_{(\tilde{Y}^{\text{sc},n}_s, \hat{\alpha}^{\text{sc},n}(s,u))}(dx\, dr) c^{-1}e^{-s} du\, ds.
    \ee
    The following lemma is an immediate consequence of the cost bound in \eqref{eq:costbd}.
     \begin{lemma}\label{lem:tight}
The sequence $\{(\nu^n,\tilde{M}^{\text{sc},n}, \tilde{R}^{n}),  n \in \N\}$ is a tight collection of $\clp_1\times C(\R_+\!:\!\clp(S))\times \R_+^a$ valued random variables.
    \end{lemma}

    \begin{proof}
We first argue the tightness of $\{\nu^n\}$. Since $S$ is compact and $[\nu^n]_{3,4}(du\, ds) = c^{-1}e^{-s} du\, ds$, it suffices to check that $\{[\nu^n]_{2}\}$ is a tight sequence of $\clp(\R^a_+)$-valued random variables.
Since $x\mapsto \ell(x)$ is superlinear, it is a tightness function on $\R_+$ (cf.\ \cite[Section 2.2]{buddupbook}) and thus to prove the tightness of $\{[\nu^n]_{2}\}$ it suffices to verify
(see \cite[Lemma 2.9, Theorem 2.11]{buddupbook}) that
\be \label{eq:407}
\sup_{n \in \N} \E \sum_{(i,j) \in \sos} \int_{\R_+^a} \ell(r_{ij}) [\nu^n]_{2}(dr) <\infty.\ee
However the last bound is immediate from \eqref{eq:costbd} and on noting that the quantity on the left of the display above equals
$$c^{-1}\sup_{n \in \N} \E \sum_{(i,j) \in \sos} \int_{[0,\infty)\times [0,c]} e^{-s} \ell(\hat{\alpha}_{ij}^{\text{sc},n}(s,u)) ds\, du.$$
The tightness of $\tilde{M}^{\text{sc},n}$ is immediate on noting that $\clp(S)$ is compact and that, from \eqref{eq:302}, for $0\le s \le t<\infty$
$$|\tilde{M}_t^{\text{sc},n}- \tilde{M}_s^{\text{sc},n}| \le 2|t-s|.$$

Finally, we consider $\tilde{R}^{n}$. Note that, for $(x,y) \in A$,
\begin{align*}
\E\tilde{R}^{n, xy} &
= e^{-t_n} \E\int_{[0, e^{t_n}] \times [0,c] \times \R_+} 1_{[\tilde{X}^n_{s-}=x]} 1_{[0, Q_{xy}(\tilde{L}^n_{s-})]}(u) 1_{[0, \alpha^n_{xy}(s,u)]}(r) N_{xy}(ds\, du\, dr)\\
&=  e^{-t_n} \E\int_{[0, e^{t_n}]\times [0,c]} 1_{[\tilde{X}^n_{s}=x]} 1_{[0, Q_{xy}(\tilde{L}^n_{s})]}(u)\alpha^n_{xy}(s,u) ds du 
\end{align*}
By a change of variables, and using superlinearity of $\ell$, for some $c_1 \in (0,\infty)$,
\begin{align*}
\sup_{n \in \N}\E\tilde{R}^{n, xy} &\le 
\sup_{n \in \N}\E\int_{[0,\infty)\times [0,c]} e^{-s} \hat{\alpha}_{xy}^{\text{sc},n}(s,u) ds\, du
\\
&\le c_1  \sup_{n \in \N}\E\int_{[0,\infty)\times [0,c]} e^{-s} [\ell(\hat{\alpha}_{xy}^{\text{sc},n}(s,u)) +1] ds\, du \le c_1  (K_0+c).
\end{align*}
This proves the tightness of $\{\tilde{R}^{n, xy}\}$. The result follows.
    \end{proof}

    We now give an important characterization of the weak limit points of $\nu_n$.
\newtext{The role of the next lemma is to identify any weak limit point with an admissible element of $\cls_{\gamma,\flu}$. Once this is done, lower semicontinuity of the control cost gives the Laplace upper bound; the short concluding argument is  given immediately after the lemma.}

\begin{lemma}\label{lem:char}
Suppose that, along some subsequence,
\[
(\nu^n,\tilde{M}^{\text{sc},n}, \tilde R^n)
\quad\text{converges weakly to}\quad
(\nu^*, M^*, R^*).
\]
Then, almost surely, $\nu^*$ can be disintegrated as
\[
\nu^*(dx\, dr\, du\, ds) = \theta(s)(dx\, dr\, du) e^{-s} ds
\]
for some $\theta \in \B(\R_+, \clp(S\times \R_+^a\times [0,c]))$.
Furthermore,  $\theta \in \cls_{[\nu^*]_1, R^*}$ and $M^{\theta} = M^*$ a.s.
    \end{lemma}
    \begin{proof}
\newtext{We organize the proof according to the properties defining $\cls_{[\nu^*]_1,R^*}$. We first disintegrate the limit and identify $M^*$; we then verify the balance condition in Property \ref{proper:1}(b), the marginal condition in part (c), and finally the empirical-flux condition in part (d).}

\newtext{\noindent\textit{Step 1: Disintegration and identification of $M^*$.}}
Since $[\nu^n]_{4}(ds) = e^{-s} ds$ for all $n$, the same is true of the limit $\nu^*$ as well, namely,
$[\nu^*]_{4} (ds) = e^{-s}ds$. This proves the first statement in the lemma. We now consider the last two statements.

We assume without loss of generality that the convergence of $(\nu^n,\tilde{M}^{\text{sc},n}, \tilde R^n)$ to $(\nu^*, M^*, R^*)$ happens along the full sequence and holds a.s. Note, from \eqref{eq:302} and \eqref{eq:nundefn}, that
\begin{equation} 
 \tilde{M}^{\text{sc},n}_{t}
        = \tilde{M}^{\text{sc},n}_{0} + \int_{0}^{t} \tilde{M}^{\text{sc},n}_s ds - 
        \int_{\R^a_+\times [0,c]\times [0,t]} e^s \nu^n(\cdot , \, dr\, du \, ds), \;  t \ge 0.
\end{equation}
Taking the limit as $n\to \infty$, we get, for $t \ge 0$,
    \begin{align} 
 {M}^{*}_{t}
        &= {M}^{*}_{0} + \int_{0}^{t} {M}^{*}_s ds - 
        \int_{\R^a_+\times [0,c]\times [0,t]} e^s \nu^*(\cdot , \, dr\, du \, ds)\nonumber\\
        &= {M}^{*}_{0} + \int_{0}^{t} {M}^{*}_s ds - 
        \int_{0}^t [\theta(s)]_1 ds . \label{eq:212}
\end{align}
Solving the above linear equation shows that $\theta$ satisfies (a) in Property \ref{proper:1} with $\gamma = M^*_0$ and also that, for $t\ge 0$,
\begin{equation}\label{eq:442n}
M_t^*(\cdot) = e^t \int_t^{\infty} e^{-s} [\theta(s)]_1(\cdot) ds, \; t \ge 0.\end{equation} 
Namely, $\theta$ satisfies the first identity in Property \ref{proper:1}(b),  with $M$ replaced by $M^*$. Note also that, by definition, $\gamma = [\nu^*]_1$.

\newtext{\medskip\noindent\textit{Step 2: Integrability and balance condition.}}
We now show that $\theta$ satisfies the last two statements in Property \ref{proper:1}(b). The second statement in Property \ref{proper:1}(b) is immediate from \eqref{eq:407}, and the lowersemicontinuity and superlinearity of $\ell$, on observing that
\begin{align*} \E \sum_{(i,j) \in \sos} \int_{\R_+^a} \ell(r_{ij}) \nu^*(dx\,dr\,du\,ds)&\le \liminf_{n\to \infty}\E \sum_{(i,j) \in \sos} \int_{\R_+^a} \ell(r_{ij}) \nu^n(dx\,dr\,du\,ds)\\
&\le \sup_{n \in \N} \E \sum_{(i,j) \in \sos} \int_{\R_+^a} \ell(r_{ij}) [\nu^n]_{2}(dr) <\infty.\end{align*}

Now consider the third statement in Property \ref{proper:1}(b).
Fix $f \in \B(S:\R)$. Then, using \eqref{eq:ControlProcess}, for $0 \le t \le T_n$,
\begin{equation}
        \begin{split}
        f(\tilde{X}^n_t) &= f(x) + \\
        &\hspace{-0.5cm}\sum_{(i,j) \in \sos} \int_{[0,t] \times [0,c] \times \mathbb{R}_+} (f(j)-f(i)) 1_{[\tilde{X}^n_{s-}=i]} 1_{[0, Q_{ij}(\tilde{L}^n_{s-})]}(u) 1_{[0, \alpha^n_{ij}(s,u)]}(r) N_{ij}(ds\, du\, dr) \, .
        \end{split}
    \end{equation}
Note that, for $0 \le t \le T_n$, we can write
 \begin{align}
        f(\tilde{X}^n_t) &= f(x) + \clm^n(t) + \sum_{(i,j) \in \sos} 
        \int_0^t \left( \int_{[0, Q_{ij}(\tilde{L}^n_{s})]} \alpha^n_{ij}(s,u) du\right) (f(j)-f(i)) 1_{[\tilde{X}^n_{s}=i]} ds\\
        &= f(x) + \clm^n(t) + \int_0^t \tilde \cll^n_s f(\tilde X^n_s) ds,\label{eq:254}
    \end{align}
    where
    \be
\tilde \cll^n_s f(i) = \sum_{j: j \neq i} \left(\int_{[0, Q_{ij}(\tilde{L}^n_{s})]} \alpha^n_{ij}(s,u) du\right) (f(j)-f(i))
    \ee
    and
    \begin{align}
        \clm^n(t) &= \\
        &\hspace{-1cm}\sum_{(i,j) \in \sos} \int_{[0,t] \times [0,c] \times \mathbb{R}_+} (f(j)-f(i)) 1_{[\tilde{X}^n_{s-}=i]} 1_{[0, Q_{ij}(\tilde{L}^n_{s-})]}(u) 1_{[0, \alpha^n_{ij}(s,u)]}(r) N^c_{ij}(ds\, du\, dr),
    \end{align}
    with $N^c_{ij}(ds\, du\, dr) = N_{ij}(ds\, du\, dr) - ds\, du\, dr$.

    Now fix $t_0 \in [0, t_n]$ and let $T^0_n := e^{t_n-t_0}$.
We claim that $\clm^n(T^0_n)/T^0_n \to 0$ in probability as $n \to \infty$.
    To see this, note that, for some $c_1 <\infty$, 
    \be
\E(\clm^n(t))^2 \le c_1 \sum_{(i,j) \in A} \E \int_{[0,t]\times [0,c]} \alpha^n_{ij}(s,u) ds\, du, \mbox{ for all } n \in \N \mbox{ and } t \ge 0.
    \ee
    It then follows that, for some $c_2>0$,
    \begin{multline}
\frac{1}{(T_n^0)^2} \E(\clm^n(T^0_n))^2
\le c_1 e^{-2(t_n-t_0)}\sum_{(i,j) \in A} \E\int_{[0, e^{t_n-t_0}]\times [0,c]} \alpha^n_{ij}(s,u) ds\, du\\
= c_1 e^{-2(t_n-t_0)}\sum_{(i,j) \in A} \E \int_{(-\infty, t_n-t_0]\times [0,c]} e^v \alpha^{\text{sc},n}_{ij}(v,u) dv\, du\\
= c_1 e^{-(t_n-t_0)}e^{t_0} \sum_{(i,j) \in A} \E \int_{[t_0, \infty) \times [0,c]} e^{-w} \hat \alpha^{\text{sc},n}_{ij}(w,u) dw\, du\\
\le c_2 e^{-(t_n-t_0)}e^{t_0}
\sum_{(i,j) \in A} \E \int_{[0, \infty) \times [0,c]} e^{-w} (1+\ell(\hat \alpha^{\text{sc},n}_{ij}(w,u))) dw\, du \, ,\label{eq:1043}
    \end{multline}
    where, in order, we use the definition of $T_n^0$, rescale, invert time, and eventually bound the integrand. The claim is now immediate on using \eqref{eq:costbd}.
    Using this fact in \eqref{eq:254} we have that, as $n \to \infty$,
    \be \label{eq:803n}
\frac{1}{T_n^0} \int_0^{T_n^0}\tilde \cll^n_s f(\tilde X^n_s) ds \to 0, \mbox{ in probability. }
    \ee
    \enlargethispage{-\baselineskip}
    Changing variables as before,
    \begin{multline*}
\frac{1}{T_n^0} \int_0^{T_n^0}\tilde \cll^n_s f(\tilde X^n_s) ds =
e^{-(t_n-t_0)} \int_0^{e^{t_n-t_0}} \tilde \cll^n_s f(\tilde X^n_s) ds \\
= e^{-(t_n-t_0)} \int_{-\infty}^{t_n-t_0} e^v 
\tilde \cll^{\text{sc},n}_v f(\tilde X^{\text{sc},n}_{v}) dv,
    \end{multline*}
 where
 \be
\tilde{\cll}^{\text{sc},n}_v f(i) \doteq \tilde \cll^n_{e^v} f(i)
= \sum_{j: j \neq i} \left(\int_{[0, Q_{ij}(\tilde{L}^{\text{sc},n}_{v})]} \alpha^{\text{sc},n}_{ij}(v,u) du\right) (f(j)-f(i)), \; i \in S.
\ee
By another change of variables,
\begin{equation}
\frac{1}{T_n^0} \int_0^{T_n^0}\tilde \cll^n_s f(\tilde X^n_s) ds =
e^{t_0}\int_{t_0}^{\infty} e^{-w} \tilde \cll^{\text{sc},n}_{t_n-w} f(\tilde Y^{\text{sc},n}_{w}) dw
= e^{t_0}\int_{t_0}^{\infty} e^{-w} \hat \cll^{\text{sc},n}_{w} f(\tilde Y^{\text{sc},n}_{w}) dw,
\end{equation}
   where
   \be
\hat \cll^{\text{sc},n}_{w} f(i) \doteq \tilde \cll^{\text{sc},n}_{t_n-w} f(i)
=  \sum_{j: j \neq i} \left(\int_{[0, Q_{ij}(\tilde{M}^{\text{sc},n}_{w})]} \hat\alpha^{\text{sc},n}_{ij}(w,u) du\right) (f(j)-f(i)), \; i \in S.
   \ee
   From the above observations and using the definition of $\nu^n$ we now have that, as $n\to \infty$,
   \be
   \begin{split}
\int_{S \times \R^a_+ \times [0,c] \times [t_0, \infty)}\sum_{y: y \neq x} (f(y)-f(x))
 1_{[0, Q_{xy}(\tilde{M}^{\text{sc},n}_{s})]}(u) r_{xy}
\nu^n(dx\, dr\, du\, ds) &= \\
&\hspace{-3cm}=\frac{e^{-t_0}}{T_n^0} \int_0^{T_n^0}\tilde \cll^n_s f(\tilde X^n_s) ds \to 0, \label{eq:410n}
\end{split}
\ee
in probability.
We now argue that
\be \label{eq:348n}
\limsup_{n\to \infty}
\int_{\Xi} \sum_{y: y \neq x}
 |1_{[0, Q_{xy}(\tilde{M}^{\text{sc},n}_{s})]}(u) -
 1_{[0, Q_{xy}({M}^{*}_{s})]}(u)|
 r_{xy}
\nu^n(dx\, dr\, du\, ds) =0.
\ee
We will use the inequality:
\be
ab \le \frac{1}{\sigma} (\ell(a) + \exp\{\sigma b\} -1) \mbox{ for all } a,b\ge 0, \sigma >0.
\ee
From this inequality, for any $\sigma>0$, $u\ge 0$, and $q_1, q_2 \ge 0$,
\be
\begin{split}
|1_{[0, q_1]}(u) -
 1_{[0, q_2]}(u)|
 r_{xy} \le \frac{1}{\sigma} (\ell(r_{xy}) + (e^{\sigma|1_{[0, q_1]}(u) -
 1_{[0, q_2]}(u)|}-1)) &\le \\
 &\hspace{-4cm} \frac{1}{\sigma} \ell(r_{xy})
 +  \frac{e^{\sigma}}{\sigma} |1_{[0, q_1]}(u) -
 1_{[0, q_2]}(u)|
  \, .
  \end{split}
\ee
Using this observation in the left side of \eqref{eq:348n}, we have
\deletetext{\begin{align}
&\int_{\Xi} \sum_{y: y \neq x}
 |1_{[0, Q_{xy}(\tilde{M}^{\text{sc},n}_{s})]}(u) -
 1_{[0, Q_{xy}({M}^{*}_{s})]}(u)|
 r_{xy}
\nu^n(dx\, dr\, du\, ds)\\
&\le \frac{1}{\sigma} \int_{\R^a_+} \sum_{(x,y)\in \sos}
 \ell(r_{xy})
[\nu^n]_2(dr)  +
\frac{e^{\sigma}}{\sigma}\sum_{(x,y) \in \sos}\int_{0}^{\infty} e^{-s}
 |Q_{xy}(\tilde{M}^{\text{sc},n}_{s}) -
 Q_{xy}({M}^{*}_{s})| ds\\
 &\le \frac{1}{\sigma} \sup_{n} \left(\int_{\R^a_+} \sum_{(x,y)\in \sos}
 \ell(r_{xy})
[\nu^n]_2(dr)\right) + \frac{e^{\sigma}}{\sigma}\sum_{(x,y) \in \sos}\int_{0}^{\infty} e^{-s}
 |Q_{xy}(\tilde{M}^{\text{sc},n}_{s}) -
 Q_{xy}({M}^{*}_{s})| ds.\label{eq:740nn}
\end{align}}
\newtext{\begin{align}
&\int_{\Xi} \sum_{y: y \neq x}
 |1_{[0, Q_{xy}(\tilde{M}^{\text{sc},n}_{s})]}(u) -
 1_{[0, Q_{xy}({M}^{*}_{s})]}(u)|
 r_{xy}
\nu^n(dx\, dr\, du\, ds)\\
&\le \frac{1}{\sigma} \int_{\R^a_+} \sum_{(x,y)\in \sos}
 \ell(r_{xy})[\nu^n]_2(dr)\\
&\quad + \frac{e^{\sigma}}{\sigma}\sum_{(x,y) \in \sos}\int_{0}^{\infty} e^{-s}
 |Q_{xy}(\tilde{M}^{\text{sc},n}_{s})-Q_{xy}({M}^{*}_{s})|\, ds\\
&\le \frac{1}{\sigma} \sup_{n} \left(\int_{\R^a_+} \sum_{(x,y)\in \sos}
 \ell(r_{xy})[\nu^n]_2(dr)\right)\\
&\quad + \frac{e^{\sigma}}{\sigma}\sum_{(x,y) \in \sos}\int_{0}^{\infty} e^{-s}
 |Q_{xy}(\tilde{M}^{\text{sc},n}_{s})-Q_{xy}({M}^{*}_{s})|\, ds.\label{eq:740nn}
\end{align}}
The statement in \eqref{eq:348n} now follows on recalling \eqref{eq:407},  using the continuity of $Q$,  sending $n \to \infty$ and then $\sigma \to \infty$ in the above display.

Combining \eqref{eq:410n} and \eqref{eq:348n} we now have that,
as $n\to \infty$,
   \be
\int_{S \times \R^a_+ \times [0,c] \times [t_0, \infty)}\sum_{y: y \neq x} (f(y)-f(x))
 1_{[0, Q_{xy}(M^*_{s})]}(u) r_{xy}
\nu^n(dx\, dr\, du\, ds)  \to 0, \label{eq:411n}
\ee
in probability.

Next from the convergence of $\nu^n \to \nu^*$,  property \eqref{eq:407}, and super-linearity of $\ell$, we now conclude that
\begin{align}
&\int_{S \times \R^a_+ \times [0,c] \times [t_0, \infty)}\sum_{y: y \neq x} (f(y)-f(x))
 1_{[0, Q_{xy}(M^*_{s})]}(u) r_{xy} 
\nu^n(dx\, dr\, du\, ds)\\
&\to 
\int_{S \times \R^a_+ \times [0,c] \times [t_0, \infty)}\sum_{y: y \neq x} (f(y)-f(x))
 1_{[0, Q_{xy}(M^*_{s})]}(u) r_{xy} 
\nu^*(dx\, dr\, du\, ds) \label{eq:742nn}
\end{align}
proving that the last quantity is $0$ a.s.\  for every $t_0>0$.\\
Since $\nu^*(dx\, dr\, du\, ds) =  e^{-s} \theta(s)(dx\, dr\, du) ds$, and $t_0>0$ is arbitrary, we have that
$$
\int_{S \times \R^a_+ \times [0,c]}\sum_{y: y \neq x} (f(y)-f(x))
 1_{[0, Q_{xy}(M^*_{s})]}(u) r_{xy} 
\theta(s)(dx\, dr\, du) = 0 \mbox{ for a.e.\ } s, \mbox{ a.s. }
$$
This shows that $\theta$ satisfies the last statement in Property \ref{proper:1}(b), with $M$ replaced by $M^*$.

\newtext{\medskip\noindent\textit{Step 3: Verification of the marginal property.}}
We now check Property \ref{proper:1}(c). Note that
$$[\nu^n]_{134}(dx\,du\,ds) = \delta_{\tilde{Y}^{\text{sc},n}_s}(dx) c^{-1}du e^{-s} ds =
[\nu^n]_{14}(dx\, ds)c^{-1} du.$$
Sending  $n\to \infty$, we have $[\nu^*]_{134}(dx\, du\, ds) = [\nu^*]_{14}(dx\, ds) c^{-1} du$. Disintegrating both sides, we have
$[\theta(s)]_{13}(dx\,du)e^{-s} ds = [\theta(s)]_{1}(dx)c^{-1}du e^{-s} ds$,
verifying part (c).

\newtext{\medskip\noindent\textit{Step 4: Identification of the limiting empirical flux.}}
Finally we argue that $\theta$ satisfies part (d) in Property \ref{proper:1} with $\flu = R^*$. Fix $(x,y) \in A$.
Note that
\begin{align*}
\tilde{R}^{n, xy} &
= e^{-t_n} \int_{[0, e^{t_n}] \times [0,c] \times \R_+} 1_{[\tilde{X}^n_{s-}=x]} 1_{[0, Q_{xy}(\tilde{L}^n_{s-})]}(u) 1_{[0, \alpha^n_{xy}(s,u)]}(r) N_{xy}(ds\, du\, dr)\\
&=  e^{-t_n} \int_{[0, e^{t_n}]\times [0,c]} 1_{[\tilde{X}^n_{s}=x]} 1_{[0, Q_{xy}(\tilde{L}^n_{s})]}(u)\alpha^n_{xy}(s,u) du ds + e^{-t_n} \tilde{\clm}^{n,xy}(e^{t_n}),
\end{align*}
where $\tilde{\clm}^{n,xy}(t)$ is a martingale with quadratic variation
$$\langle \tilde{\clm}^{n,xy} \rangle_t = \int_{[0, t]\times [0,c]} 1_{[\tilde{X}^n_{s-}=x]} 1_{[0, Q_{xy}(\tilde{L}^n_{s-})]}(u)\alpha^n_{xy}(s,u) du ds, \; t\ge 0.$$
Thus, by a change of variables, for some $c_1 \in (0,\infty)$,
\begin{align*}
\E[e^{-t_n}\tilde{\clm}^{n,xy}(t_n)]^2 &\le 
e^{-t_n} \E\int_{[0,\infty)\times [0,c]} e^{-s} \hat{\alpha}_{xy}^{\text{sc},n}(s,u) ds\, du
\\
&\le e^{-t_n} \E\int_{[0,\infty)\times [0,c]} e^{-s} [\ell(\hat{\alpha}_{xy}^{\text{sc},n}(s,u)) +c] ds\, du \le c_1 e^{-t_n} (K_0+1) \to 0, 
\end{align*}
as $n \to \infty$, where the last line uses \eqref{eq:costbd}.
This shows that
\be \label{eq:breakrr0}
\tilde{R}^{n, xy} = \hat{R}^{n, xy} + R^{n, xy}_0,
\ee
where $R^{n, xy}_0 \to 0$ in probability and, by a change of variables again,
$$\hat R^{n,xy} = \int_{[0, \infty) \times [0,c]} e^{-s} 1_{[\tilde{Y}^n_{s}=x]} 1_{[0, Q_{xy}(\tilde{M}^{\text{sc},n}_{s})]}(u)\hat \alpha^n_{xy}(s,u) du ds.$$
Note that $\hat R^{n,xy}$ can be written as
\begin{align}
\hat R^{n,xy} 
&= c\int_{\Xi}  1_{[z=x]} 1_{[0, Q_{xy}(\tilde{M}^{\text{sc},n}_{s})]}(u) r_{xy} \nu^n(dz\, dr\, du \, ds)\\
&= c\int_{\Xi}  1_{[z=x]} 1_{[0, Q_{xy}(M^*_{s})]}(u) r_{xy} \nu^n(dz\, dr\, du \, ds) + R^{n, xy}_1, \label{eq:803nn}
\end{align}
where
$$
R^{n, xy}_1 = c\int_{\Xi}  1_{[z=x]} [1_{[0, Q_{xy}(M^*_{s})]}(u) - 1_{[0, Q_{xy}(\tilde{M}^{\text{sc},n}_{s})]}(u)]  r_{xy} \nu^n(dz\, dr\, du \, ds).
$$
From \eqref{eq:348n}, $R^{n, xy}_1 \to 0$ in probability. As in the proof of
\eqref{eq:742nn}, from the convergence of $\nu^n \to \nu^*$,  property \eqref{eq:407}, and super-linearity of $\ell$, we now have on sending $n\to \infty$ in \eqref{eq:803nn}, that
\begin{align*}
 R^{*,xy} &= c\int_{\Xi}  1_{[z=x]} 1_{[0, Q_{xy}(M^*_{s})]}(u) r_{xy} \nu^*(dz\, dr\, du \, ds)\\
 &= c\int_{\Xi} e^{-s} 1_{[z=x]} 1_{[0, Q_{xy}(M^*_{s})]}(u) r_{xy} \theta(s)(dz\, dr\, du) ds\\
 &= c\int_{\R_+^a\times  [0,c] \times \R_+} e^{-s} [\theta(s)]_1(x)  1_{[0, Q_{xy}(M^*_{s})]}(u) r_{xy} \eta(s)(dr\, du \mid x) ds,
 \end{align*}
 where $\theta(s)(dz\, dr\, du) = [\theta(s)]_1(dz)\eta(s)(dr\, du \mid z)$.
 This shows that $\theta$ satisfies part (d) in Property \ref{proper:1} with $\flu = R^*$. We have thus shown that $\theta \in \cls_{[\nu^*]_1, R^*}$ and $M^{\theta} = M^*$ a.s. The result follows.
    \end{proof}

    We now complete the proof of the LDP upper bound by proving the inequality in \eqref{eq:Laplaceupp}.

    \noindent\textbf{Proof of the Laplace upper bound \eqref{eq:Laplaceupp}.}
    From \eqref{eq:delbd}, \eqref{eq:n1029} and \eqref{eq:nundefn} it follows that
    \begin{equation*}
    \begin{split}
    -\frac{1}{T_n} \log \mathbb{E} \left[ e^{- T_n h(L_{T_n}, R_{T_n})} \right] + \delta
&\ge \\
&\hspace{-2cm}\mathbb{E} \left[ h(\tilde{M}^{\text{sc},n}_0, \tilde R^n) +  c\int_{\Xi} \sum_{y: y \neq x} 1_{[0,Q_{xy}(\tilde{M}^{\text{sc},n}_{s})]}(u)
\ell(r_{xy}) \nu^n(dx\, dr\,du\, ds) \right]
\end{split}
\end{equation*}
Now suppose that along some subsequence $(\nu^n, \tilde{M}^{\text{sc},n},  \tilde R^n)$ converges in distribution to
$( \nu^*, M^*, R^*)$ and by a standard subsequential argument, assume without loss of generality that the convergence holds along the full sequence.
Then from Lemma \ref{lem:char} and definition of $M^{\theta}$ we have $M^*(0) = M^{\theta}(0) = [\nu^*]_1$, where $\nu^*(dx\,dr\,du\,ds) = \theta(s)(dx\,dr\,du)e^{-s}ds$.
Also, the map $\R_+^a \ni r \mapsto \sum_{(i,j) \in \sos} \ell(r_{ij})$ is lower semicontinuous. Thus passing to the limit as $n \to \infty$ in the last display we obtain that
\begin{align*}
&\liminf_{n\to \infty} -\frac{1}{T_n} \log \mathbb{E} \left[ e^{- T_n \deletetext{h(L_{T_n})}\newtext{h(L_{T_n}, R_{T_n})}} \right] + \delta\\ 
& \ge \mathbb{E} \left[ h([\nu^*]_1, R^*) + c \int_{\Xi} 
\sum_{y: y\neq x} 1_{[0,Q_{xy}({M}^{*}_{s})]}(u)
\ell(r_{xy}) \nu^*(dx\, dr\,du\, ds) \right]\\
& = \mathbb{E} \left[ h([\nu^*]_1, R^*) + c\int_{\Xi} e^{-s} \sum_{y: y\neq x}  1_{[0,Q_{xy}({M}^{*}_{s})]}(u)\ell(r_{xy})\,  \theta(s)(dx\, dr\, du) ds \right]\\
&= \mathbb{E} \left[ h([\nu^*]_1, R^*) + J(\theta)\right] \\
&\ge \mathbb{E} \left[ h([\nu^*]_1, R^*) + I([\nu^*]_1, R^*)\right] \\
&\ge \inf_{\deletetext{\gamma}\newtext{(\gamma,\flu)} \in \clp(S)\times \R_+^a}
\left[ h(\gamma, \flu) + I(\gamma, \flu)\right] \, ,
\end{align*}
where the next to last inequality uses the property $\theta \in \cls_{[\nu^*]_1, R^*}$ from Lemma \ref{lem:char} and the definition of $J$ and $I$ from \eqref{eq:JTheta} and \eqref{eq:IGamma}, respectively. This completes the proof of the upper bound in \eqref{eq:Laplaceupp}. \hfill \qed

    \section{ Preparation for the Lower bound}
\label{sec:laplow}
In this section and the next we complete the proof of the large deviation lower bound, namely \eqref{eq:Laplacelow}.  The main idea  is to use the variational formula in \eqref{eq:RepresentationFormula} and construct suitable control processes $\al^n$ such that the limsup of the expectation on the right side of \eqref{eq:RepresentationFormula} is bounded above by the expression on the right side of \eqref{eq:Laplacelow}.

We begin the section with an alternative representation of the rate function which is the topic of Section \ref{sec:altform}. This alternative representation suggests certain feedback controls $\alpha^n$ as control actions for which the associated state processes and costs converge in a suitable fashion so that the limsup of the expectations on the right side of \eqref{eq:RepresentationFormula} has the upper bound we need. \newtext{We start with a near minimizer for the limit variational problem and then carry out  a series of approximations to it
in Sections \ref{sec:prelimlow}-- \ref{sec:ctycont} that have certain regularity properties. These regularized near minimizers will be then used in Section \ref{sec:construction} to construct  feedback controls for the stochastic system and complete the proof of the lower bound.

The point of these approximations is that a generic near minimizer of the deterministic variational problem need not have the nondegeneracy and smoothness needed for the construction of a feedback control and for establishing the desired convergence properties  for  the controlled dynamics and cost. We therefore first perturb a given near minimizer suitably  to obtain uniform nondegeneracy, and then truncate and regularize it in time. Each modification preserves the required admissibility properties and the resulting  control  can then be used to construct suitable feedback controls for the stochastic system, which is done in Section \ref{sec:construction}. }

Throughout the section we will take Assumption \ref{assu:lowbd} to hold.

\subsection{Alternative form for the rate function}
\label{sec:altform}
We now provide an alternative variational formula for the rate function which will be more tractable for the proof of the lower bound.

Given $(\gamma, \flu) \in \clp(S)\times \R_+^a$, let $(\rho, H) \in \B(\R_+, \clp(S) \times \clk_0(S))$ be such that
the following properties hold.
\begin{property}\label{proper:2}
$\,$
\begin{enumerate}\deletetext{[(a)]}
\item[\newtext{(a)}] We have
$\gamma = \int_{0}^{\infty} e^{-s} \rho(s) ds$.

\item[\newtext{(b)}] Let
$$M_t \doteq e^t \int_{t}^{\infty} e^{-s} \rho(s) ds, \; t \ge 0.$$
Then,  $H_s(x,y) = 0$ if $Q_{xy}(M_s) =0$ for a.e. $s$, and all $(x,y) \in A$.

\item[\newtext{(c)}] For all $f\in \B(S,\R)$,
\be
\int_{S} \sum_{y: y \neq x} (f(y) - f(x)) H_s(x,y) \rho_s(dx) = 0, \mbox{ for a.e. } s \ge 0.
\ee

\item[\newtext{(d)}] For all $(x,y) \in A$,
$$\int_0^{\infty} e^{-s} \rho_s(x) H_s(x,y) ds = \flu^{xy}.
$$
\end{enumerate}
\end{property}
We denote the collection of all such $(\rho,H)$ as $\tilde\cls_{\gamma,\flu}$, namely,
\be
\tilde\cls_{\gamma, \flu} = \{(\rho,H) \in \B(\R_+, \clp(S) \times \clk_0(S)): \mbox{Property \ref{proper:2} holds} \}.
\ee
We will denote the $M$ defined in Property \ref{proper:2} (b) as $M^{\rho}$ (suppressing dependence on $H$).

For $(\rho,H) \in \B(\R_+, \clp(S) \times \clk_0(S))$, let
\be \label{eq:150nn}
\tilde J(\rho,H) =\int_0^{\infty}  e^{-s} \left(\int_{S}  \sum_{y: y \neq x}  Q_{xy}(M_s) \ell \left(\frac{H_s(x,y)}{Q_{xy}(M_s)}\right) \,  \rho_s(dx)\right) ds,
\ee
where throughout, $0\ell(\al/0)$ is taken to be $+\infty$ if $\al>0$ and $0$ if $\al=0$.
Let
\be
\label{eq:Alternative}
\tilde I(\gamma, \flu) = \inf_{(\rho,H) \in \tilde \cls_{\gamma, \flu}} \deletetext{J(\rho,H)}\newtext{\tilde J(\rho,H)}, \;\; \gamma \in \clp(S), \flu \in \R_+^a.
\ee

Part (c) above says that, for a.e. $s$, $\rho_s$ is the stationary distribution of a Markov chain with rate matrix $H_s$. Using this and letting $\eta_s(x,y) = \rho_s(x) H_s(x,y)$, it is easy to see that the inner integral in \eqref{eq:150nn} equals $\IDW_{{Q(M(s))}}(\gamma(s), \eta(s))$, and consequently, $\tilde I$ equals the right side of \eqref{eq:155nn}.

The following lemma shows that $I$ and $\tilde I$ are the same.
\begin{lemma}\label{lem:ratesame}
For all $(\gamma, \flu) \in \clp(S)\times \R_+^a$, $I(\gamma, \flu) = \tilde I(\gamma, \flu)$.
\end{lemma}

\begin{proof}
Fix $(\gamma, \flu) \in \clp(S)\times \R_+^a$. We first argue that $\tilde I(\gamma, \flu) \le I(\gamma, \flu)$.
Fix $\delta>0$ and let $\theta \in \cls_{\gamma, \flu}$ be such that $J(\theta) \le I(\gamma, \flu) + \delta$.
Let $\rho(s) \doteq [\theta(s)]_1$. Then from Property \ref{proper:1}(a) we see that part (a) of Property \ref{proper:2} holds with this choice of $\rho$.  Let $M \doteq M^{\theta}$.  Disintegrating $\theta(s)(dx\,dr\,du) = \rho(s)(dx) \eta(s)(dr\, du \mid x)$,
define for $(x,y) \in \sos$, $s\ge 0$,
\begin{equation}
H_s(x,y) \doteq c\int_{\R^a_+ \times [0,c]} 1_{[0,Q_{xy}(M_s)]}(u) r_{xy} \eta(s)(dr\, du \mid x), \;\; 
\end{equation}
when the right side is finite; otherwise set $H_s(x,y) \doteq Q_{xy}(M_s)$.
From Property \ref{proper:1}(b) we see that, for a.e. $s$, whenever $\rho_s(x)>0$, the right side in the above display is finite.
Note that $H$ satisfies  Property \ref{proper:2}(b).
Also note that, for a.e. $s>0$,
\begin{align*}
&\int_{S} \sum_{y: y \neq x} (f(y) - f(x)) H_s(x,y) \rho_s(dx) \\
&= c\int_{S}  \int_{\R^a_+ \times [0,c]} \sum_{y: y \neq x} (f(y) - f(x)) 1_{[0,Q_{xy}(M_s)]}(u) r_{xy} \eta(s)(dr\, du \mid x)     \rho_s(dx) \\
&= c\int_{S}  \int_{\R^a_+ \times [0,c]} \sum_{y: y \neq x} (f(y) - f(x)) 1_{[0,Q_{xy}(M_s)]}(u) r_{xy} \theta(s)(dx\, dr\, du)     =0 \, ,
\end{align*}
where the last equality follows by Property \ref{proper:1}(b). This verifies 
Property \ref{proper:2}(c) for $(\rho, H)$.
Finally, Property \ref{proper:2}(d) for $(\rho, H)$, is immediate from Property \ref{proper:1}(d) and the definition of $\rho, H$.
Thus it follows that $(\rho, H) \in \tilde \cls_{\gamma, \flu}$.

Next, by first disintegrating $\theta(s)(dx dr du)$, using the definition of $H_s(x,y)$, the convexity of $\ell$, and eventually using \eqref{eq:JTheta}, we get the following chain of inequalities:
\begin{align*}
\tilde I(\gamma, \flu) & \le \int_{[0,\infty)\times S}  \sum_{y: y \neq x} e^{-s}  Q_{xy}(M_s) \ell \left(\frac{H_s(x,y)}{Q_{xy}(M_s)}\right) \,  \rho_s(dx) ds\\
&\hspace{-1cm}= 
\int_{[0,\infty)\times S}  \sum_{y: y \neq x} e^{-s}  Q_{xy}(M_s) \ell \left(\frac{c}{Q_{xy}(M_s)}
\int_{\R^a_+ \times [0,Q_{xy}(M_s)]}  r_{xy} \eta(s)(dr\, du \mid x)
\right) \,  \rho_s(dx) ds\\
&\hspace{-1cm}\le 
\int_{[0,\infty)\times S}  \sum_{y: y \neq x} e^{-s}  Q_{xy}(M_s) \frac{c}{Q_{xy}(M_s)}
\int_{\R^a_+ \times [0,Q_{xy}(M_s)]} \ell \left(  r_{xy}\right) \eta(s)(dr\, du \mid x)
 \,  \rho_s(dx) ds\\
 &\hspace{-1cm}= c\int_{\Xi}  \sum_{y: y \neq x} e^{-s}  
 1_{[0,Q_{xy}(M_s)]}(u)
 \ell \left(  r_{xy}\right) \theta(s)(dx\, dr\, du)\, ds = J(\theta) \le I(\gamma, \flu) + \delta.
\end{align*}
Since, $\delta>0$ is arbitrary, we get $\tilde I(\gamma, \flu) \le I(\gamma, \flu)$.

For the reverse inequality, assume without loss of generality that $\tilde I(\gamma, \flu)<\infty$.
Suppose, for $\delta>0$, $(\rho, H) \in \tilde \cls_{\gamma, \flu}$ is such that $\tilde J(\rho, H) \le \tilde I(\gamma, \flu) + \delta$. 
Let $M \doteq  M^{\rho}$.
Let $v(s) = (v_{xy}(s), (x,y) \in \sos)$ with $v_{xy}(s) \doteq \frac{H_s(x,y)}{Q_{xy}(M_s)} 1_{\{Q_{xy}(M_s) >0\}} +  1_{\{Q_{xy}(M_s) =0\}}$, $(x,y) \in \sos$.
Define
$$\theta(s)(dx\, dr\, du) = \rho_s(dx) \delta_{v(s)}(dr) c^{-1} du
$$
Using Property \ref{proper:2}(a) we see that Property \ref{proper:1}(a) holds with this choice of $\theta$.
Noting that for a.e.\ $s \geq 0$, $H_s(x,y)=0$ if $Q_{xy}(M_s) =0$, we have by the definition of $\theta$ in the display above,  
\begin{align*}
&\int_{S\times \R^a_+\times [0,c]} \sum_{y: y \neq x} (f(y) - f(x)) 1_{[0, Q_{xy}(M_s)]}(u) r_{xy} \theta(s)(dx\, dr\, du) \\
&= c^{-1}\int_{S\times  [0,c]} \sum_{y: y \neq x} (f(y) - f(x)) 1_{[0, Q_{xy}(M_s)]}(u) \bigg(\frac{H_s(x,y)}{Q_{xy}(M_s)} 1_{\{Q_{xy}(M_s) >0\}} + \\ &\hspace{10cm}  1_{\{Q_{xy}(M_s) =0\}}\bigg) \rho_s(dx) du\\
&= c^{-1}\int_{S} \sum_{y: y \neq x} (f(y) - f(x)) H_s(x,y)   \rho_s(dx) =0 \, ,
\end{align*}
where the last equality follows from Property \ref{proper:2}(c). This verifies Property \ref{proper:1}(b). Also, clearly  $\theta(s)$ satisfies Property \ref{proper:1}(c). 
Finally note that, with our definition of $\theta$,
\begin{align*}
&c\int_{\R^a_+\times [0,c]\times \R_+} e^{-s} [\theta(s)]_1(x) 1_{[0, Q_{xy}(M_s)]}(u) r_{xy} \eta(s)(dr\, du\mid x) ds\\
&= \int_{ [0,c]\times \R_+} e^{-s} \rho_s(x) 1_{[0, Q_{xy}(M_s)]}(u) v_{xy}(s)  ds\, du\\
&= \int_{\R_+} e^{-s} \rho_s(x)H_s(x,y)  ds = \flu^{xy}
\end{align*}
verifying Property \ref{proper:1}(d).
 Thus we have shown that $\theta \in \cls_{\gamma, \flu}$.
 
Next, by disintegrating $\theta(s)(dx dr du)$ once again and using the definition of $v(s)$, 
\begin{align*}
I(\gamma, \flu) &\le c\int_{\Xi} e^{-s} \sum_{y: y \neq x} 1_{[0, Q_{xy}(M_s)]}(u) \ell(r_{xy})\,  \theta(s)(dx\, dr\, du) ds\\
&= \int_{S \times [0,c]\times \R_+} e^{-s} \sum_{y: y \neq x} 1_{[0, Q_{xy}(M_s)]}(u) \ell\left(\frac{H_s(x,y)}{Q_{xy}(M_s)}\right)\, \rho(s)(dx) \, du\,  ds\\
&=\int_{S \times \R_+ } e^{-s} \sum_{y: y \neq x} Q_{xy}(M_s) \ell\left(\frac{H_s(x,y)}{Q_{xy}(M_s)}\right)\, \rho(s)(dx) \,  ds\\
&= \tilde J(\rho, H) \le \tilde I(\gamma, \flu) +\delta.
\end{align*}
Since $\delta>0$ is arbitrary, the inequality $I(\gamma, \flu) \le \tilde I(\gamma, \flu)$ follows.
\end{proof}
\subsection{Selection of near minimizer}
\label{sec:prelimlow}
We now proceed to the proof of the lower bound, namely the proof of \eqref{eq:Laplacelow} for a function $h \in C_b(\clp(S) \times \R_+^a)$. 
We assume, without loss of generality, that $h$ is Lipschitz (see \cite[Corollary 1.10]{buddupbook}), namely, for some $h_{\mbox{\tiny{lip}}} \in (0,\infty)$,
$$|h(x_1, \flu_1)-h(x_2, \flu_2)| \le h_{\mbox{\tiny{lip}}} (|x_1-x_2| + |\flu_1 -\flu_2|), \; x_i \in \clp(S), \flu_i \in \R_+^a, \; i=1,2.$$

Now fix  $\delta>0$. Let $(\gamma^*, \flu^*) \in \clp(S)\times \R_+^a$ be such that
\be
h(\gamma^*, \flu^*) + I(\gamma^*, \flu^*) \le \inf_{(\gamma, \flu) \in \clp(S)\times \R_+^a} [h(\gamma, \flu) + I(\gamma, \flu)]+\delta.
\ee
Recalling from Lemma \ref{lem:ratesame} that $I(\gamma, \flu) = \tilde I(\gamma, \flu)$, we can find $(\rho^*, H^*) \in \tilde \cls_{\gamma^*, \flu^*}$ such that
\be \label{eq:506n}
h(\gamma^*, \flu^*) + \tilde J(\rho^*, H^*) \le h(\gamma^*, \flu^*) + I(\gamma^*, \flu^*) + \delta \le \inf_{(\gamma, \flu) \in \clp(S)\times \R_+^a} [h(\gamma, \flu) + I(\gamma, \flu)]+2\delta.
\ee
Let $M^* \doteq M^{\rho^*}$. Our approach will be to use the variational formula in \eqref{eq:RepresentationFormula} and construct suitable controls $\alpha^n$ such that corresponding controlled processes $(\tilde L^n_{T_n}, \tilde R^n)$ well approximate $(\gamma^*, \flu^*)$ and the associated integral cost in \eqref{eq:RepresentationFormula} is appropriately close to $\tilde J(\rho^*, H^*)$.
For this we will need certain modifications of $(\gamma^*, \flu^*, \rho^*, H^*)$ that have desirable nondegeneracy and continuity properties.  This is the topic of the next two sections. 
\newtext{Accordingly, the lower-bound preparation has two steps: Section \ref{sec:nondeg} perturbs a near minimizer to obtain uniform nondegeneracy, and Section \ref{sec:ctycont} regularizes the resulting pair in time while preserving near optimality. The regularized object is then used in Section \ref{sec:construction} to define the feedback controls.}

We begin with the following elementary lemma whose proof is omitted.
\begin{lemma}\label{lem:biconbex}
The map $(x,y) \mapsto x \ell(y/x)$ is convex on $\R_+\times \R_+$.
\end{lemma}

\subsection{Ensuring nondegeneracy.}
\label{sec:nondeg}
The main result of this section given below ensures that $(\gamma^*, \flu^*, \rho^*, H^*)$ can be well approximated by a $(\gamma^1, \flu^1, \rho^1, H^1)$ which has suitable nondegeneracy properties.
\begin{lemma}\label{lem:nondeg}
There is a $(\gamma^1, \flu^1) \in \clp(S)\times \R_+^a$, $(\rho^1, H^1) \in \tilde \cls_{\gamma^1, \flu^1}$ and the corresponding $M^1 = M^{\rho^1}$, and $\cons_1 \in (0,\infty)$ such that

$$
h(\gamma^{1}, \flu^1) + \tilde J(\rho^{1}, H^{1}) \le h(\gamma^{*}, \flu^*) + \tilde J(\rho^{*}, H^{*}) +  \delta
$$
and
\begin{equation}\label{eq:allbds}
\inf_{s\ge 0} \inf_{x \in S} \min\{M^{1}_s(x), \rho^{1}_s(x) \} \ge \cons_1,\;
\inf_{s\ge 0} \inf_{(x,y) \in A_0} H^{1}_s(x,y) \ge \cons_1.
\end{equation}
\end{lemma}
\begin{proof}
Fix $\al \in (0,1)$ and let
$$\rho_s^{\al} = (1-\al) \rho_s^* + \al \pi^*, \; s \ge 0.$$
Also, let $H^{\al} \in \B([0,\infty):\clk(S))$ be defined as
$$
H^{\al}_s(x,y) = \frac{1}{\rho_s^{\al}(x)}\left( (1-\al) H^*_s(x,y) \rho^*_s(x) + \al Q_{xy}(\pi^*) \pi^*(x)\right), \; (x,y) \in \sos,
$$
$$M^{\al}_s \doteq (1-\al) M^*_s + \al \pi^*, \; s \ge 0,$$
$$
\flu^{\al, xy} = \int_0^{\infty} e^{-s} \rho^{\al}_s(x) H^{\al}_s(x,y) ds = (1-\al) \flu^{*,xy} + \al\pi^*(x)Q_{xy}(\pi^*),$$
and $\gamma^{\al} = (1-\al)\gamma^* + \al \pi^*$.
Then it is easy to verify that $(\rho^{\al}, H^{\al}) \in \tilde \cls_{\gamma^{\al}, \flu^{\al}}$ and
$M^{\rho^{\al}} = M^{\al}$.
From Assumption \ref{assu:lowbd}(c)
 \begin{equation}
 \inf_{(x,y) \in A_0} Q_{xy}(m) \ge k_{Q} \inf_{x \in S} m_x, \mbox{ for all }  m \in \clp(S).
\end{equation}
Furthermore,  from this assumption, we can find
$a_{\pi^*} \in (0,\infty)$ such that
\begin{equation}\label{eq:703nnn}
\inf_{x\in S} \pi^*_x \ge a_{\pi^*}.
\end{equation}

{Using the linearity Assumption \eqref{assu:lowbd} (b)}, with $r(\al) = 2\al(1-\al) +\al^2$, write 
\begin{align*}
\rho^{\al}_s(x) Q_{xy}(M^{\al}_s)  &= 
(1-\al)^2 \rho^{*}_s(x)Q_{xy}(M^{*}_s) \\
&\hspace{-3cm} + r(\al) \left(\frac{\al(1-\al)}{r(\al)}\pi^*Q_{xy}(M^{*}_s)
+ \frac{\al(1-\al)}{r(\al)} \rho^*_s(x)Q_{xy}(\pi^*)
+ \frac{\al^2}{r(\al)}\pi^*(x)Q_{xy}(\pi^*)\right).
\end{align*}
Similarly write
\begin{align*}
\rho^{\al}_s(x)H^{\al}_s(x,y) &= (1-\al) \rho^*_s(x) H^{*}_s(x,y) + \al \pi^*(x) Q_{xy}(\pi^*)\\
&=
(1-\al)^2 \rho^{*}_s(x)H^{*}_s(x,y) + \\
&\hspace{-3cm} r(\al)\left(\frac{\al(1-\al)}{r(\al)}\rho_s^*(x)H^*_s(x,y)+ \frac{\al(1-\al)}{r(\al)} \pi^*(x)Q_{xy}(\pi^*) + \frac{\al^2}{r(\al)}\pi^*(x)Q_{xy}(\pi^*)\right)
\end{align*}

Using Lemma \ref{lem:biconbex}, for $s\ge 0$ and $(x,y) \in A_0$
\begin{multline}
Q_{xy}(M^{\al}_s) \rho^{\al}_s(x) \ell \left(\frac{H^{\al}_s(x,y)}{Q_{xy}(M^{\al}_s)}\right) 
= Q_{xy}(M^{\al}_s) \rho^{\al}_s(x) \ell \left(\frac{\rho^{\al}_s(x)H^{\al}_s(x,y)}{\rho^{\al}_s(x)Q_{xy}(M^{\al}_s)}\right) \\
\le (1-\al)^2Q_{xy}(M^{*}_s) \rho^{*}_s(x) \ell \left(\frac{H^{*}_s(x,y)}{Q_{xy}(M^{*}_s)}\right)  \\
 + (\al(1-\al)\pi^*Q_{xy}(M^{*}_s)
+ \al(1-\al) \rho^*_s(x)Q_{xy}(\pi^*)
+ \al^2\pi^*(x)Q_{xy}(\pi^*))\\
\times \ell\left(\frac{
\al(1-\al)\rho_s^*(x)H^*_s(x,y)
+ \al(1-\al) \pi^*(x)Q_{xy}(\pi^*)
+ \al^2\pi^*(x)Q_{xy}(\pi^*)
 }
 {
\al(1-\al)\pi^*(x)Q_{xy}(M^{*}_s)
+ \al(1-\al) \rho^*_s(x)Q_{xy}(\pi^*)
+ \al^2\pi^*(x)Q_{xy}(\pi^*)
 }
 \right).  \label{eq:708nn}
\end{multline}
Write the second term in the above display as
$$ A_{\al}(s,x,y) \ell\left(\frac{ B_{\al}(s,x,y)}{ A_{\al}(s,x,y)}\right).
$$
Note that, for a.e. $s>0$
$$
\lim_{\al \to 0} \sup_{(x,y)\in A_0} A_{\al}(s,x,y) = 0,\;
\lim_{\al \to 0} \sup_{(x,y)\in A_0} B_{\al}(s,x,y) |\log\al| = 0.
$$
Also, {denoting $\log^- x = \max (0, - \log x)$} (suppressing $(s,x,y)$)
\begin{align}
A_{\al} \ell\left(\frac{ B_{\al}}{ A_{\al}}\right) &=
B_{\al} \log B_{\al} - B_{\al} \log A_{\al} - B_{\al}+ A_{\al}\label{eq:1142nn}\\
&\le B_{\al} \log B_{\al} + B_{\al} \log^- A_{\al} +  A_{\al}.
\end{align}
Recalling the definition of $A^{\al}$ we see that
\begin{align*} 
B_{\al}(s,x,y) \log^- A_{\al}(s,x,y) &\le B_{\al}(s,x,y) \log^-(\al^2 \pi^*_x Q_{xy}(\pi^*)\\
&\hspace{-2cm}\le 2 B_{\al}(s,x,y) |\log \al| + B_{\al}(s,x,y) |\log \pi^*_x Q_{xy}(\pi^*)| \to 0 \mbox{ as } \al \to 0.
\end{align*}

Combining the last three displays we see that, for a.e. $s$,
$$ \lim_{\al \to 0} \sup_{(x,y)\in A_0} A_{\al}(s,x,y) \ell\left(\frac{ B_{\al}(s,x,y)}{ A_{\al}(s,x,y)}\right) = 0.
$$
We claim that for some $C: \R_+ \to \R_+$, with $\int_{[0, \infty)} e^{-s} C(s) ds < \infty$
$$
\sup_{\al \in (0, 1/2)} \sum_{(x,y) \in A_0} A_{\al}(s,x,y) \ell\left(\frac{ B_{\al}(s,x,y)}{ A_{\al}(s,x,y)}\right) \le C(s).
$$
Assuming the claim,
from \eqref{eq:708nn}, we see  that
\begin{align*}
\tilde J(\rho^{\al}, H^{\al})
\le \tilde J(\rho^{*}, H^{*})
+ \int_{[0,\infty)}  \sum_{(x,y) \in A_0} e^{-s} A_{\al}(s,x,y) \ell\left(\frac{ B_{\al}(s,x,y)}{ A_{\al}(s,x,y)}\right)   ds
\end{align*}
and, from dominated convergence, the second term on the right side converges to $0$ as $\al \to 0$.
Since $(\gamma^{\al}, \flu^{\al}) \to (\gamma^{*}, \flu^{*})$ as $\al \to 0$, and $h$ is continuous,
the result follows on taking $\al$ suitably small 
and setting $(M^{\al}, H^{\al}, \rho^{\al}, \gamma^{\al}, \flu^{\al})$ as $(M^{1}, H^{1}, \rho^{1}, \gamma^1, \flu^1)$. 

Finally we prove the claim.
From the first line of \eqref{eq:1142nn} we see that, since $\tilde J(\rho^*, H^*) < \infty$, and $Q$ is bounded, with $\varrho_s(x,y) = H^*_s(x,y) \rho^*_s(x)$
\be\label{eq:134nn}
\sum_{(x,y) \in A_0} \int_0^{\infty} e^{-s} |\varrho_s(x,y) \log \varrho_s(x,y))|  < \infty.
\ee
Using \eqref{eq:1142nn} again, for some $c_1, c_2, c_3>0$
\begin{align*}
\sup_{\al \in (0, 1/2)} A_{\al}(s,x,y) \ell \left(\frac{ B_{\al}(s,x,y)}{ A_{\al}(s,x,y)}\right)
& \\ &\hspace{-5cm}\leq c_1\sup_{\al \in (0, 1/2)}[ 1+B_{\al}(s,x,y)\log B_{\al}(s,x,y) + B_{\al}(s,x,y) \log^-A_{\al}(s,x,y) ]\\
&\hspace{-5cm}\le c_2\sup_{\al \in (0, 1/2)}[1+B_{\al}(s,x,y)\log B_{\al}(s,x,y) + 2B_{\al}(s,x,y) |\log \al|]\\
&\hspace{-5cm}\le c_3\sup_{\al \in (0, 1/2)}[1+|\varrho_s(x,y)\log \varrho_s(x,y)|].
\end{align*}
In the last line we have used the fact that $\sup_{\al \in (0, 1/2)} \al \log \al <\infty$.
The claim now follows on taking
$$C(s) = c_3\sup_{\al \in (0, 1/2)}\sum_{(x,y) \in A_0} [1+|\varrho_s(x,y)\log \varrho_s(x,y)|]$$
and using the property in \eqref{eq:134nn}. The result follows.

\end{proof}

\subsection{Finite time truncation and continuity of controls}
\label{sec:ctycont}

The following lemma says that we can modify the controls $(\rho^1, H^1)$ so that they are continuous. The fact that $\rho^1, H^1$ and $M^1$ have the nondegeneracy property in \eqref{eq:allbds} will be useful in constructing the modification.
\begin{lemma}
\label{lem:step2}
There exist $(\gamma^2, \flu^2) \in \clp(S)\times \R_+^a$, $(\rho^2, H^2) \in \tilde \cls_{\gamma^2, \flu^2}$, $\cons_2, \cons_3 \in (0,\infty)$ such that
$$h(\gamma^2, \flu^2) + \tilde J(\rho^2, H^2) \le
h(\gamma^1, \flu^1) + \tilde J(\rho^1, H^1) + \delta
$$
and, with $M^2 = M^{\rho^2}$,
\be \label{eq:713nn} \inf_{s\ge 0} \inf_{x \in S} \min\{M^{2}_s(x), \rho^{2}_s(x)\} \ge \cons_2,\;
\mbox{ and for all } s \ge 0 \mbox{ and } (x,y) \in A_0, H^{2}_s(x,y) \ge \cons_2.\ee
Furthermore, $$
\sup_{x\in S} |\rho^2_s(x)- \rho^2_t(x)| \le \cons_3|t-s|, \mbox{ for all } s,t >0
$$
and $t\mapsto H^2_t(x,y)$ is a continuous map on $\R_+$ for all $(x,y) \in A_0$.
\end{lemma}
\begin{proof}
For $\kappa>0$ define 
$$\rho^{1,\kappa}(s) \doteq \kappa^{-1} \int_{s-\kappa}^s \rho^1(u) du, \; 0 \le s \le T,$$
where we define $\rho^1(u) \doteq \rho^1(0)$ for $u \le 0$.
Also define for $s \in [0,T]$ and $(x,y) \in A_0$
\be \label{eq:504nn}
H^{1,\kappa}_s(x,y) = \frac{\int_{s-\kappa}^s \rho^1_u(x) H^1_u(x,y) du}{\int_{s-\kappa}^s \rho^1_u(x) du},
\ee
where $H^1_u(x,y) \doteq H^1_0(x,y)$ for $u\le 0$. Note that the integral in the numerator is finite from the finiteness of $\tilde J(\rho^1, H^1)$ (cf. \eqref{eq:134nn}).
By construction, $s \mapsto \rho^{1,\kappa}_s(x)$ is Lipschitz continuous and $s\mapsto H^{1, \kappa}_s(x,y)$ is continuous.
Also, with $\cons_1$ as in Lemma \ref{lem:nondeg}, 
\be \label{eq:737nn}
\inf_{s\ge 0} \inf_{x \in S} \rho^{1,\kappa}_s(x) \ge \cons_1,\;
\mbox{ and for all } s \ge 0 \mbox{ and } (x,y) \in A_0, H^{1,\kappa}_s(x,y) \ge \cons_1.
\ee
Define, for $t \ge 0$, $M^{1,\kappa}(t) = e^t \int_t^{\infty} e^{-s} \rho^{1,\kappa}_s ds$.  
Using the fact that $\rho^1_s(x) \le 1$ for all $(s,x) \in \R_+\times S$, it is easy to verify that
\be \label{eq:726nn}
\sup_{s \ge 0}\max_{x\in S} |M^{1,\kappa}_s(x) - M^1_s(x)| \le 2 \kappa.
\ee
This also shows that, with $\gamma^{1, \kappa} = M^{1,\kappa}(0)$,
$$\sup_{x\in S}|\gamma^{1, \kappa}_x - \gamma^1_x| \le 2 \kappa.$$
Let
$$\flu^{1,\kappa, xy} \doteq \int_0^{\infty} e^{-s} \rho^{1, \kappa}_s(x)
H^{1,\kappa}_s(x,y) ds.
$$
Using the inequality for a measurable map $\tau: \R_+ \to \R_+$, that is integrable under the measure $e^{-s} ds$,
\be \label{eq:717nn}
\left|\int_0^{\infty} e^{-s} \frac{1}{\kappa} \int_{s-\kappa}^s \tau(u) du -\int_0^{\infty} e^{-s} \tau(s) ds\right| \le   \kappa \left (\tau(0) + \int_0^{\infty} e^{-s} \tau(s) ds\right),
\ee
where $\tau(s) \doteq \tau(0)$ for $s \le 0$,
it is easy to verify that
$$|\flu^{1,\kappa, xy} - \flu^{1, xy}| \le \kappa (\flu^{1, xy} + \rho^1_0(x) H^1_0(x,y)).
$$
Thus we have
\be
(\gamma^{1, \kappa}, \flu^{1,\kappa}) \to (\gamma^1, \flu^1) \mbox{ as } \kappa \to 0.
\ee
Also, it is easily checked that,
$(\rho^{1, \kappa}, H^{1, \kappa}) \in \tilde \cls_{\gamma^{1, \kappa}, \flu^{1, \kappa}}$, and $M^{1, \kappa} = M^{\rho^{1, \kappa}}$.
We claim that, as $\kappa \to 0$,
$$
\tilde J(\rho^{1,\kappa}, H^{1,\kappa}) \to \tilde J(\rho^{1},H^{1}).
$$
Assuming the claim, the result now follows on taking $\kappa$ suitably small and setting
$(\gamma^2, \flu^2, M^2, \rho^2, H^2) = (\gamma^{1,\kappa}, \flu^{1,\kappa}, M^{1,\kappa}, \rho^{1,\kappa}, H^{1,\kappa})$.

We now prove the claim.
Define for $s\ge 0$ and $(x,y) \in A_0$
$$
\tilde Q^{\kappa}_{xy}(s) \doteq \frac{\int_{s-\kappa}^s \rho^1_u(x) Q_{xy}(M^1_u) du}{\int_{s-\kappa}^s \rho^1_u(x) du}.
$$
Then, using convexity from Lemma \eqref{lem:biconbex},
\begin{align*}
\tilde Q_{xy}^{\kappa}(s) \rho^{1,\kappa}_s(x) \ell \left(\frac{H^{1,\kappa}_s(x,y)}{\tilde Q_{xy}^{\kappa}(s)}\right)
&\le \rho^{1,\kappa}_s(x)\frac{1}{\int_{s-\kappa}^s \rho^1_u(x) du}
\int_{s-\kappa}^s \rho^1_u(x) Q_{xy}(M^1_u)
\ell \left(\frac{H^{1}_u(x,y)}{Q_{xy}(M^1_u)}\right)du \\
&=\frac{1}{\kappa}
\int_{s-\kappa}^s \rho^1_u(x) Q_{xy}(M^1_u)
\ell \left(\frac{H^{1}_u(x,y)}{Q_{xy}(M^1_u)}\right) du.
\end{align*}
Once again using \eqref{eq:717nn}, we have
\begin{align*}
\hat J^{\kappa} &\doteq 
\int_{\R_+ } e^{-s} \sum_{(x,y) \in A_0} \tilde Q^{\kappa}_{xy}(s) \ell\left(\frac{H^{1,\kappa}_s(x,y)}{\tilde Q^{\kappa}_{xy}(s)}\right)\, \rho^{1,\kappa}_s(x) \,  ds
\end{align*}
satisfies
\begin{align*}
|\hat J^{\kappa} -\tilde J(\rho^1, H^1)| \le  \kappa \left(\tilde J(\rho^1, H^1) +
\sum_{(x,y) \in A_0}  \rho^1_0(x)Q_{xy}(M^1_0) \ell\left(\frac{H^{1}_0(x,y)}{Q_{xy}(M^1_0)}\right)\right).
\end{align*}
Note that the  term on the right side converges to $0$ as $\kappa \to 0$.
Thus to complete the proof it suffices to show that
\be \label{eq:801nn}
\limsup_{\kappa \to 0}\left|\hat J^{\kappa}- \tilde J(\rho^{1,\kappa}, H^{1,\kappa})\right| = 0.
\ee
For this, we have from \eqref{eq:726nn}, and the affine property of $Q$, that there is a $\tilde d_1, d_1 \in (0, \infty)$  such that for all $\kappa \in (0,1)$, $s\ge 0$, and $(x,y) \in A_0$
\begin{align*}
 |\tilde Q^{\kappa}_{xy}(s) -
Q_{xy}(M^{1,\kappa}_s)| &\le 
|\tilde Q^{\kappa}_{xy}(s) -
Q_{xy}(M^{1}_s)| + |Q_{xy}(M^{1}_s) -
Q_{xy}(M^{1,\kappa}_s)|\\
& \le \tilde d_1 \frac{\int_{s-\kappa}^s \rho^1_u(x) |Q_{xy}(M^1_u) - Q_{xy}(M^1_s)|  du}{\int_{s-\kappa}^s \rho^1_u(x)} + \tilde d_1 \kappa \le 
d_1\kappa.
\end{align*}
Also, from Assumption \ref{assu:lowbd}(c), and lower bounds on $\rho^1$ and $M^1$ from Lemma \ref{lem:nondeg}, for some $d_2>0$
$$\inf_{s\ge 0} \min_{(x,y) \in A_0}\min\{\tilde Q^{\kappa}_{xy}(s),
Q_{xy}(M^{1,\kappa}_s)\} \ge d_2.$$
We note the following inequality: For $b \in (0,\infty)$ and $a, a' \in (\varpi, \infty)$ for some $\varpi >0$,
\be \label{eq:511nn}
|a\ell(b/a) - a'\ell(b/a')| \le \left(\frac{b}{\varpi}+1\right)|a-a'|.
\ee
This shows that, for some $d_3<\infty$, and all $s\ge 0$, $(x,y) \in A_0$,
\begin{equation}\label{eq:758nn}
\rho^{1,\kappa}_s(x)\left|\tilde Q^{\kappa}_{xy}(s) \ell\left(\frac{H^{1,\kappa}_s(x,y)}{\tilde Q^{\kappa}_{xy}(s)}\right) - 
Q_{xy}(M^{1,\kappa}_s) \ell\left(\frac{H^{1,\kappa}_s(x,y)}{Q_{xy}(M^{1,\kappa}_s)}\right)\right| \le d_3(H^{1,\kappa}_s(x,y) +1)\kappa.
\end{equation}
Next, using \eqref{eq:504nn} and the lower bound on $\rho^1$ from \eqref{eq:737nn}, for some $d_4 \in (0,\infty)$, and all $\kappa \in (0,1)$, $s\ge 0$, $(x,y) \in A_0$,
$$H^{1,\kappa}_s(x,y) \le d_4 \frac{1}{\kappa} \int_{s-\kappa}^s  H^1_u(x,y) du.
$$
Thus, using \eqref{eq:717nn} again, for all $\kappa \in (0,1)$
\begin{align*}
\int_0^{\infty} e^{-s} H^{1,\kappa}_s(x,y) ds
&\le d_4 \int_0^{\infty} e^{-s} \frac{1}{\kappa} \int_{s-\kappa}^s  H^1_u(x,y) du\\
&\le 2d_4\int_0^{\infty} e^{-s} H^{1}_s(x,y) ds
+ d_4\kappa H^1_0(x,y) <\infty.
\end{align*}
The last inequality above is a consequence of $\tilde J(\rho^1, H^1)<\infty$ and the boundedness and uniform positivity of $Q_{xy}(M^1_s)$.
Using the above estimate in \eqref{eq:758nn}, integrating both sides of that inequality with respect to $e^{-s} ds$ and sending $\kappa \to 0$, we have \eqref{eq:801nn} and the result follows.
\end{proof}
\section{Lower Bound Proof}
\label{sec:construction}
We will now use the near optimal $(\gamma^2, \flu^2)$ and the associated $(\rho^2, H^2)$ from the previous section to construct a  control sequence $\al^n \in \Theta^n$, that has suitable convergence properties, for use on the right side of \eqref{eq:RepresentationFormula}. Henceforth we will suppress the superscript $2$ in $(M^2, \rho^2, \flu^2, H^2, \gamma^2)$ and simply write $(M, \rho, \flu, H, \gamma)$. 
From Lemma \ref{lem:step2} it follows that
\be \label{eq:basinteg}
\int_0^{\infty} e^{-s} \sum_{(x,y) \in A_0} H_s(x,y) \log H_s(x,y) ds < \infty.
\ee
Recall the constants $c_Q$, $k_Q$, and $\cons_2$ from \eqref{eq:cqdefn}, Assumption \ref{assu:lowbd}(c), and Lemma \ref{lem:step2}, respectively. Choose $T>0$ such that
\be \label{eq:356nn}
\cons_2^{-1}\int_T^{\infty} e^{-s} \sum_{(x,y) \in A_0} [H_s(x,y) \log H_s(x,y)
+ H_s(x,y) \log^-(k_Q\cons_2/4) + c_Q]ds \le \delta/4.
\ee
Fix $\eps_1 \in (0,1)$ such that
\begin{align}
&
\eps_1 (\int_0^T e^{-s} \sum_{(x,y) \in A_0}  [H_s(x,y) \log H_s(x,y)
+ H_s(x,y) \log^-(k_Q\cons_2/4) + c_Q] ds) \le \delta/16 \nonumber\\
&\quad \mbox{ and } 4\eps_1 \|h\|_{\infty} \le \delta/8.  \label{eq:delta8}
\end{align}
Using Assumption \ref{assu:lowbd} (d),  we can find $a^* \in (0,\infty)$ such that, with
\be\label{eq:finsig1}
\sigma_1 = \inf\{t\ge 1: L_t(x) \ge a^* \mbox{ for all } x \in S\},\;\;
\PP(\sigma_1 <\infty) \ge (1-\eps_1/2).\ee

From \eqref{eq:finsig1}, we can find $r_1 \in (1,\infty)$ such that
\be
\label{eq:ProbaSigma1r1}
\PP(\sigma_1 > r_1) \le \eps_1.
\ee

Now choose $\eps_0 \in (0,1)$ such that
\be\label{eq1211n}
\eps_0 \le e^{-T} \cons_2/32, \;\; \eps_0e^T\left(\frac{28}{k_Q\cons_2} \sum_{(x,y) \in A_0} \int_0^T e^{-s} (H_s(x,y)+1) ds+ 3 \|h\|_{\mbox{\tiny{lip}}}\right) \le \delta/8.
\ee

Note that, from Lemma \ref{lem:step2}, $M(T) \in \clp_+(S)$ and since $(\rho, H) \in \tilde \cls_{\gamma, \flu}$ and $M = M^{\rho}$, Property \ref{proper:2}(b) holds, i.e.,
\begin{equation}\label{eq:rev1-1}
M(T) = e^T\int_T^{\infty}e^{-s} \rho(s) ds.\end{equation}
Since from Property \ref{proper:2}(c) $\sum_{x\in S} \rho_s(x) H_s(x,y)=0$ for all $y \in S$, for a.e. $s\ge 0$, we obtain that, with
\begin{equation}\label{eq:rev1-2}
V(x,y) \doteq \frac{\int_T^{\infty}e^{-s} \rho_s(x)H_s(x,y) ds}
{\int_T^{\infty}e^{-s} \rho_s(x) ds}, (x,y) \in S\times S,
\end{equation}
$V \in \clk(S)$ satisfies
$\sum_{x\in S} M_T(x)V(x,y)=0$, for all $y\in S$. Furthermore, using Property \ref{proper:2}(b) of $\{H_s\}$, $V(x,y)=0$ for all $(x,y) \in A\setminus A_0$.
Also, from Lemma \ref{lem:step2},
for all $(x,y) \in A_0$
\begin{equation}\label{eq:vbd}
 V(x,y) \ge \cons_2.
\end{equation}

Denote $q\doteq M(T)$. Consider a Markov jump process on $S$ with rate matrix $V$. When this process starts with $x \in S$, we denote the state process by $\{U^x(t)\}$. By an application of the ergodic theorem, we can find $r_2> \tilde r_1 \doteq r_1(2-\eps_0)/\eps_0$ such that
\be\label{eq:520nn}
\sum_{x\in S}\PP\left(  \sup_{t \ge r_2} \|\frac{1}{t} \int_0^t \delta_{U^x(s)} ds - q\| \ge \eps_0 \right) \le \eps_1.\ee

Also, define for $u \in [e^{t_n-T}, e^{t_n}]$, $\check H_u \in \clk(S)$ as
$$\check H_u(x,y) = H_{t_n-\log u}(x,y), \; \mbox{ for } (x,y) \in S\times S.$$

\newtext{We are now ready to construct the controls and controlled processes used to realize the regularized near minimizer $(\rho,H)$. The construction has three stages. The split of the time interval $[0, e^{t_n}]$ at time $e^{t_n-T}$ in the construction below is dictated by the reverse-time change of variables $s=t_n-\log u$: the interval up to $e^{t_n-T}$ corresponds to the tail $s\ge T$, whereas $[e^{t_n-T},e^{t_n}]$ corresponds to $s\in[0,T]$, traversed from $T$ to $0$.

In the first stage, the process follows its original, uncontrolled dynamics (that is, $\alpha^n=1$) until the stopping time $\sigma_1$, at which every coordinate of its empirical measure is bounded below by $a^*$. This provides the uniform nondegeneracy needed to suitably estimate the cost (see Lemma \ref{lem:lem2prelim}). The event $E_1=\{\sigma_1\le r_1\}$ is the good event on which this initial stage has ended by the deterministic time $r_1$; on $E_1^c$, the process remains uncontrolled.

In the second stage, on $E_1$, the process is run with the fixed rate matrix $V$ up to time $e^{t_n-T}$. The measure $q=M(T)$ is invariant for $V$, and $E_2$ is the good event (see precise definition below) on which the post-$\sigma_1$ occupation averages remain close to $q$ after the burn-in time $r_2$. Consequently, the empirical measure at time $e^{t_n-T}$ is close to $M(T)$, which supplies the terminal condition for the final tracking stage. The particular definition of $V$ also ensures that this part of the construction reproduces the contribution of the tail $s\ge T$ to the limiting empirical flux, while its control cost is asymptotically negligible.

In the third stage, on $E_1\cap E_2$, the time-dependent rates $\check H_u=H_{t_n-\log u}$ are used on $[e^{t_n-T},e^{t_n}]$. Under the reverse-time parametrization, these rates are designed to realize the pair $(\rho_s,H_s)$ and to make the rescaled empirical measure follow $M_s$ from $s=T$ to $s=0$. The event $E_3$ is the good event on which this empirical-measure tracking remains within the prescribed tolerance.

Thus, $G_n\doteq E_1\cap E_2\cap E_3$ is the overall good event on which all three stages are completed and the construction produces the desired occupation-measure, empirical-measure, and empirical-flux behavior; this event gives the main contribution to the variational estimate. In fact, the probability estimates below give $\liminf_{n\to\infty}\PP(G_n)\geq 1-4\eps_1$. If the construction exits one of the good sets, the control is switched back to $\alpha^n=1$ at the corresponding exit time. Since $\ell(1)=0$, this incurs no further control cost. The estimates below show that the contributions from $E_1^c$, $E_1\cap E_2^c$, and $E_1\cap E_2\cap E_3^c$ can all be made negligible, using the small probabilities of these events together with the bounds on the costs accumulated before the corresponding exit times.}
We now define the controlled processes $\{\tilde X^n\}$, the corresponding empirical measure $\{\tilde L^n\}$ and the associated controls $\al^n$.
\begin{construction}
\label{cons:cont}
 Let $\tilde X^n_0= x$ and $\tilde L^n_0 = \delta_x$. Let PRM $\{N_{ij}, (i,j) \in A\}$ and filtered probability space $(\Om, \clf, \PP, \{\clf_t\})$ be as in Sections \ref{sec:2} and \ref{sec:varrep}. Let $\Lambda = e^T$.  
Define the processes $\{\tilde X^n, \tilde L^n, \al^n\}$ recursively as follows.
\begin{enumerate}\deletetext{[(a)]}
\item[\newtext{(a)}] 
We take $\al^n_{ij}(t,u)=1$ for all $(i,j)\in A$, $u\in[0,c]$, and $t\leq \sigma_1$.
On the \textit{good set} $E_1=\{\sigma_1\leq r_1\}$, the construction then proceeds to part (b). On the \textit{bad set} $E_1^c=\{\sigma_1>r_1\}$, we also take $\al^n_{ij}(t,u)=1$ for all $(i,j)\in A$, $u\in[0,c]$, and $t>\sigma_1$, so the process remains uncontrolled. Note that $\PP(E_1^c)\leq\eps_1$.

\item[\newtext{(b)}] 
On the good set $E_1$ we define the process for $t \ge \sigma_1$ using the  Markov chain with rate matrix $V$ with initial condition $\tilde X^n(\sigma_1)$, and which is conditionally independent of $\tilde X^n(\cdot \wedge \sigma_1)$ given $\tilde X^n(\sigma_1)$.
This is done by setting on the set $E_1E_2$, where,
$$E_2 \doteq \{\sup_{t \in [r_2, e^{t_n-T}-\sigma_1]} \|\frac{1}{t} \int_{\sigma_1}^{t+\sigma_1} \delta_{\tilde X^n(s)} ds - q\| < \eps_0\}.$$
and for $s \in (\sigma_1, e^{t_n-T}]$, $(i,j) \in A$,
$$
\alpha^n_{ij}(s,u) =
\begin{cases}
\dfrac{V(i,j)}{Q_{ij}(\tilde L^n(s-))} &
\begin{aligned}
&\text{if } \tilde X^n(s-)=i, \; Q_{ij}(\tilde L^n(s-)) \neq 0,\\
&u \in [0, Q_{ij}(\tilde L^n(s-))], \; (i,j) \in A_0,
\end{aligned} \\[0.8em]
1 & \text{otherwise.}
\end{cases}
$$
while, on the set $E_1E_2^c$ by  defining
$\al^n_{ij}(s,u)$ by the above formula for $s \in (\sigma_1, \sigma_1+\sigma_2]$ where
$$\sigma_2 \doteq \inf \{t \ge r_2: \|\frac{1}{t} \int_{\sigma_1}^{t+\sigma_1} \delta_{\tilde X^n(s)} ds - q\| \ge \eps_0\} \wedge (e^{t_n-T} - \sigma_1).$$
and set $\al^n_{ij}(s,u)=1$ for $s>\sigma_1+\sigma_2$.

Note that, from \eqref{eq:520nn}, $\PP(E_2^c \mid E_1) \le \eps_1$.

\item[\newtext{(c)}] 
On the good set $E_1E_2$ we define the process $\tilde X^n(t)$ for $t \in (e^{t_n-T}, e^{t_n}]= (T_n/\Lambda, T_n]$ using a time non-homogeneous Markov chain  with rate matrices $\{\check H_{t+T_n/\Lambda}, t \ge 0\}$, that is conditionally independent of 
$\{\tilde X^n(\cdot \wedge e^{t_n-T})\}$ given
$\tilde X^n(e^{t_n-T})$.
This is done by setting on the set $E_1E_2E_3$, where
$$E_3 \doteq \{\sup_{s \in [T_n/\Lambda,T_n]} \|\tilde L^n(s) - L(s)\| < \cons_2/8\},
\; L(s) \doteq M(t_n-\log s), s \in [T_n/\Lambda,T_n],$$
and for $s \in (T_n/\Lambda,T_n]$, $(i,j) \in A$,
$$
\alpha^n_{ij}(s,u) =
\begin{cases}
\dfrac{\check H_s(i,j)}{Q_{ij}(\tilde L^n(s-))} &
\begin{aligned}
&\text{if } \tilde X^n(s-)=i, \; Q_{ij}(\tilde L^n(s-)) \neq 0,\\
&u \in [0, Q_{ij}(\tilde L^n(s-))], \; (i,j) \in A_0,
\end{aligned} \\[0.8em]
1 & \text{otherwise.}
\end{cases}
$$

while, on the set $E_1E_2 E_3^c$,  defining
$\al^n_{ij}(s,u)$ by the above formula for $s \in (e^{t_n-T}, \sigma_3]$ where
$$\sigma_3 \doteq \inf \{t \in [T_n/\Lambda,T_n]: \|\tilde L^n(t)- L(t)\| \ge \cons_2/8\}.$$
and set $\al^n_{ij}(s,u)=1$ for $s \in (\sigma_3, T_n]$.

\item[\newtext{(d)}] The processes $\{\tilde X^n, \tilde L^n, \tilde R^n\}$ are given by \eqref{eq:ControlProcess}, \eqref{eq:ContrEmpOccMeas}, and \eqref{eq:ControlFlux}, and these  processes, together with $\{\al^n\}$  are constructed recursively from one jump to the next.
\end{enumerate}
This completes the construction of the processes $\{\tilde X^n_s, \tilde L^n_s, \al^n_s\}$ for $s \in [0, T_n]$.
 \end{construction}

Using the notation introduced in Sections \ref{sec:varrep} and \ref{sec:lapupp} for the controls $\al^n$ defined in
Construction \ref{cons:cont}, and $\{\beta^n, n \in \N\}$ of $\clp(S\times \R_+)$-valued random variables defined as
\be
\label{eq:MeasureConstrProcLower}
\beta^n(dx\, ds) = e^{-s} \delta_{\tilde{Y}^{\text{sc},n}_s}(dx) ds \, ,
\ee
we have, from \eqref{eq:302},
\begin{equation}
\label{eq:RevResTimeEvEmpOccLower}
 \tilde{M}^{\text{sc},n}_{t}
        = \tilde{M}^{\text{sc},n}_{0} + \int_{0}^{t} \tilde{M}^{\text{sc},n}_s ds - 
        \int_0^t e^s \beta^n(\cdot ,  ds), \;  t \ge 0 \, .
\end{equation}
Also from \eqref{eq:RescaledChangeReverseRepresentationFormula}  note that 
\be \label{eq:528nn}
- \frac{1}{T_n} \log \mathbb{E} \left[ e^{- T_n \deletetext{h({L}^n_{T_n}, R_{T_n})}\newtext{h({L}_{T_n}, R_{T_n})}} \right] \le  \mathbb{E} \left[h(\tilde{M}^{\text{sc},n}_0, \tilde R^n)+ \sum_{(i,j) \in A} \int_{[0,\infty) \times [0,c]} e^{-s} \ell(\hat{\alpha}_{ij}^{\text{sc},n}(s,u)) \, ds du \right] 
\ee

\subsection{Key Estimates}

\newtext{The overall construction can be separated into the initial transient phase, the period for approach to $M(T)$, and the final tracking interval.The estimates below follow the same decomposition as the construction. We control separately the contribution from $E_1^c$, from $E_1E_2^c$, from $E_1E_2E_3^c$, and finally from the good event $E_1E_2E_3$. }

In the following, to simplify the notation, we denote
$$
\sum_{(i,j) \in A} \int_{[0,\infty) \times [0,c]} e^{-s} \ell(\hat{\alpha}_{ij}^{\text{sc},n}(s,u)) \, ds du  = J_n, \qquad
h(\tilde{M}^{\text{sc},n}_0, \tilde R^n) + J_n = B_n \, .
$$
Recalling the relation between $\hat{\alpha}^{\text{sc},n}$ and $\alpha^n$, we see  that $J_n$ can be rewritten as
$$J_n = e^{-t_n}\sum_{(i,j) \in A} \int_{[0,T_n] \times [0,c]}  \ell({\alpha}_{ij}^{n}(s,u)) \, ds du .
$$

The  lemmas below will provide us with the key estimates  needed for the lower bound proof. 
\newtext{
The first two lemmas control the contributions from the complements of 
$E_1$ and $E_2$.  The next lemma is key and  treats the empirical measure and the empirical flux jointly. It proves that the final tracking stage is completed with asymptotically full probability on $E_1E_2$, and that, on the resulting good event, $(\tilde M^{\mathrm{sc},n},\tilde R^n)$ converges to $(M,\flu)$.  The last two lemmas then use this joint convergence to estimate, respectively, the contributions from the failed-tracking event and from the good event.
}
\begin{lemma}
\label{lem:lem1prelim}
$\E[B_n 1_{E_1^c}] \le \|h\|_{\infty}\eps_1$.
\end{lemma}
\begin{proof}
The left side, using  the definitions of $B_n$ and $E_1^c$, can be written as
\be
\mathbb{E} \left[ \left( h(\tilde{M}_0^{\text{sc},n}, \tilde R^n) + e^{-t_n}\sum_{i,j \in A} \int_{[0,T_n] \times [0,c]}  \ell \left( {\al}_{ij}^{n}(s,u) \right) \, ds \, du \right) 1_{\sigma_1 > r_1} \right] \, .
\ee
From \eqref{eq:ProbaSigma1r1}, the first term on the left side can be bounded as
$$
\deletetext{\mathbb{E} \left[ h(\tilde{M}_0^{\text{sc},n}) \right]}\newtext{\mathbb{E} \left[ h(\tilde{M}_0^{\text{sc},n},\tilde R^n)1_{\{\sigma_1>r_1\}} \right]} \leq \|h\|_{\infty}\eps_1 \, .
$$
Finally, by Construction \ref{cons:cont}(a), $\alpha^n = 1$ on $E_1^c$ and thus the second term on the left side is identically zero. 

\end{proof}
The next lemma will estimate the contribution of $B_n$ on the set $E_1E_2^c$.
\begin{lemma}
\label{lem:lem2prelim}
$$
\limsup_{n\to \infty}\E[B_n1_{E_1E_2^c}] \le \|h\|_{\infty}\eps_1+ \delta/4.
$$
\end{lemma}
\begin{proof}
Since $\PP(E_2^c\mid E_1) \le \eps_1$,
$\E[(B_n-J_n) 1_{E_1E_2^c}] \le \|h\|_{\infty}\eps_1$.
Also, on $E_1E_2^c$,
writing, for $0 \le t_1 \le t_2 \le T_n$
\begin{equation}\label{eq:rev1-3}
\tilde J^n[t_1, t_2] \doteq \int_{[t_1, t_2] \times [0,c]} \ell({\alpha}^{n}_{ij}(s,u)) \, dsdu,\end{equation}
we have 
$\tilde J^n[0, \sigma_1]=0$, since $\al^n_{ij} \equiv 1 $ on $[0, \sigma_1]$.

Also, using the definition of $\sigma_1$, for $s \in [\sigma_1, \sigma_1+r_2]$,
\begin{equation}\label{eq:424n}
\inf_{x\in S}\tilde L^n_s(x) \ge \frac{\sigma_1}{\sigma_1+r_2}
\inf_{x\in S}\tilde L^n_{\sigma_1}(x)
\ge
a^*\frac{\sigma_1}{\sigma_1+r_2} \ge \frac{a^*}{r_1+r_2}.\end{equation}

Using the lower bound on $\rho_s$ from Lemma \ref{lem:step2} we see that
\begin{equation}\label{eq:425nn}
\sup_{(x,y) \in A_0} V(x,y) \le e^T \cons_2^{-1} \sup_{(x,y) \in A_0}\int_0^{\infty} e^{-s} \rho_s(x) H_s(x,y) ds = e^T \cons_2^{-1}\sup_{(x,y) \in A_0}\flu_{xy} \doteq k(T)\end{equation}

We will make use of the following estimate: for $0< r<R<\infty$, 
\begin{equation}\label{eq:crr}
\sup_{(z_1, z_2) \in [r, R]^2} \{z_2 \ell(z_1/z_2)\} \doteq c_{r,R} < \infty.
\end{equation}

From  \eqref{eq:424n}, \eqref{eq:425nn}, \eqref{eq:crr}, \eqref{eq:cqdefn}, Lemma \ref{lem:step2}, and Assumption \ref{assu:lowbd}(c), denoting 
$$v_1' = \min\{a^*k_Q/(r_1+r_2), \cons_2\}, \mbox{ and }  v'_2(T) = \max\{c_Q, k(T)\},$$
$$ Q_{ij}(\tilde L^n_s) \ell \left(\frac{V(i,j)}{Q_{ij}(\tilde L^n_s)}\right) \le c_{v_1', v'_2(T)} \mbox{ for all } (i,j) \in A_0.$$
Consequently,
\begin{equation}\label{eq:rev1-4}
\tilde J^n[\sigma_1, \sigma_1+r_2] \le a r_2 c_{v_1', v'_2(T)}.
\end{equation}

For $s \in [r_2, \sigma_2)$, recalling the lower bound on $q=M(T)$ from Lemma \ref{lem:step2}, the definition of $\sigma_2$, and our choice of $\eps_0$ (see \eqref{eq1211n})
$$
\inf_{x\in S} \frac{1}{s} \int_{\sigma_1}^{\sigma_1+s} \delta_{\tilde X^n_u}(x) du \ge \cons_2/2.
$$
Thus, for any $x \in S$, and $s \in [r_2, \sigma_2)$
\begin{multline}
\tilde L^n_{\sigma_1+s}(x) = \frac{1}{\sigma_1+s} \int_0^{\sigma_1+s} \delta_{\tilde X^n_u}(x) du\\
= \frac{\sigma_1}{\sigma_1+s} \frac{1}{\sigma_1} \int_0^{\sigma_1}  \delta_{\tilde X^n_u}(x) du +
\frac{s}{\sigma_1+s} \frac{1}{s} \int_{\sigma_1}^{\sigma_1+s}  \delta_{\tilde X^n_u}(x) du\\
\ge \frac{s}{\sigma_1+s}\frac{\cons_2}{2}
\ge \frac{r_2}{r_1+r_2}\frac{\cons_2}{2}\ge \frac{\cons_2}{4},
\label{eq:1019nn}
\end{multline}
where the last inequality follows on noting that $r_2>r_1$.

Next, using the convexity property from Lemma \ref{lem:biconbex}, for $(i,j) \in A_0$,
$$Q_{ij}(\tilde L^n_s)\ell \left(\frac{V(i,j)}{Q_{ij}(\tilde L^n_s)}\right)
\le \frac{1}
{\int_T^{\infty}e^{-u} \rho_u(i) du}\int_T^{\infty}e^{-u} \rho_u(i) Q_{ij}(\tilde L^n_s)\ell \left(\frac{H_u(i,j)}{Q_{ij}(\tilde L^n_s)}\right) du.
$$
From \eqref{eq:1019nn} we have that, for $s \in [\sigma_1+r_2, \sigma_1+\sigma_2)$, $\tilde L^n_{s}(x) \ge \cons_2/4$ for all $x \in S$. Combining this with the above display and the observation that for $\alpha, \beta >0$
\be \label{eq:753nn} \al \ell(\beta/\al) \le \beta \log \beta + \beta \log^- \al  + \al,\ee
we get
\begin{align*}
&\sum_{(i,j) \in A_0} Q_{ij}(\tilde L^n_s)\ell \left(\frac{V(i,j)}{Q_{ij}(\tilde L^n_s)}\right)\\
&\le \frac{1}
{\int_T^{\infty}e^{-u} \rho_u(i) ds}\int_T^{\infty}e^{-u}
\sum_{(i,j) \in A_0}\rho_u(i) \\
&\qquad {}\times [H_u(i,j) \log H_u(i,j)
+ H_u(i,j) \log^-(k_Q\cons_2/4) + c_Q] du\\
&\le 
e^{T} \cons_2^{-1}\int_T^{\infty}e^{-u}
\sum_{(i,j) \in A_0}\rho_u(i) \\
&\qquad {}\times [H_u(i,j) \log H_u(i,j)
+ H_u(i,j) \log^-(k_Q\cons_2/4) + c_Q] du \\
&\le e^{T} \delta/4,
\end{align*}
where the last inequality follows from \eqref{eq:356nn}.

Consequently,
\be \label{eq:755nn}
e^{-t_n}\tilde J^n[\sigma_1+r_2, \sigma_1+\sigma_2] \le 
 e^{-t_n} \int_{[\sigma_1+r_2, e^{(t_n-T)}]} (e^{T} \delta/4) ds
 \le \delta/4.\ee

Since on $E_1E_2^c$, $J^n[\sigma_1+\sigma_2, \infty]=0$, we have,
combining the above estimates,
\be \label{eq:323n}
\limsup_{n\to \infty}\E[J_n1_{E_1E_2^c}] \le \limsup_{n\to \infty} \eps_1[e^{-t_n} ar_2 c_{v_1', v'_2(T)} + \delta/4] \le  \delta/4.
\ee
We thus have
$$
\limsup_{n\to \infty}\E[B_n1_{E_1E_2^c}] \le \|h\|_{\infty}\eps_1 + \delta/4
$$
and the result follows.
\end{proof}

\newtext{
Recall the sequence $\beta^n$ defined in \eqref{eq:MeasureConstrProcLower}.  The next lemma treats the empirical-measure and empirical-flux coordinates together.  

\begin{lemma}
\label{lem:lem7.4}
Let $G_n\doteq E_1E_2E_3$.  Then
\begin{equation}\label{eq:tracking-exit-vanishes}
\PP(E_1E_2E_3^c)\to  0.
\end{equation}
Moreover,
\begin{equation}\label{eq:joint-good-convergence}
\E\left[1_{G_n}\left(
\sup_{0\le s\le T}\|\tilde M^{\mathrm{sc},n}_s-M_s\|\right) \right]
+\E\left[1_{G_n}\left(|\tilde R^n-\flu|\right)\right]\to 0.
\end{equation}
In addition, for every bounded continuous $g:S\times[0,T]\to\R$,
\begin{equation}\label{eq:beta-good-convergence}
\E\left[1_{G_n}\left|
\int_{S\times[0,T]}g(x,s)\,\beta^n(dx\,ds)
-\int_0^T e^{-s}\sum_{x\in S}g(x,s)\rho_s(x)\,ds
\right|\right] \to 0.
\end{equation}
\end{lemma}

\begin{proof}
Write $\tau_n\doteq e^{t_n-T}=T_ne^{-T}$.  We divide the proof into three steps.

\medskip\noindent\textit{Step 1: convergence at the beginning of the tracking interval.}
On $E_1E_2$, after time $\sigma_1$ the process is a Markov chain with rate matrix $V$ throughout the interval $[\sigma_1,\tau_n]$.  Since $S$ is finite, $V$ is irreducible, and $q=M(T)$ is its invariant distribution, the ergodic theorem holds in $L^1$, uniformly over the initial state.  Since $\sigma_1\le r_1$ on $E_1E_2$ and $\tau_n-\sigma_1\to\infty$, it follows that
\begin{equation}\label{eq:V-ergodic-joint}
\E\left[1_{E_1E_2}\left\|
\frac{1}{\tau_n-\sigma_1}\int_{\sigma_1}^{\tau_n}
\delta_{\tilde X_s^n}\,ds-q\right\|\right]\to 0.
\end{equation}
Using
\[
\tilde M_T^{\mathrm{sc},n}
=\frac{\sigma_1}{\tau_n}\frac1{\sigma_1}
\int_0^{\sigma_1}\delta_{\tilde X_s^n}\,ds
+\frac{\tau_n-\sigma_1}{\tau_n}
\frac1{\tau_n-\sigma_1}\int_{\sigma_1}^{\tau_n}
\delta_{\tilde X_s^n}\,ds,
\]
we therefore obtain
\begin{equation}\label{eq:terminal-measure-joint}
\E\left[1_{E_1E_2}\|\tilde M_T^{\mathrm{sc},n}-M_T\|\right]
\to 0.
\end{equation}

\medskip\noindent\textit{Step 2: convergence during the tracking interval.}
Consider an auxiliary process $\bar X^n$ as follows.  It agrees with $\tilde X^n$ up to time $\tau_n$ and, on $E_1E_2$, it continues on the whole interval $[\tau_n,T_n]$ with the rate matrices $\check H_t$.  Use the same PRMs for constructing $\bar X^n$ as were used for $\tilde X^n$.  On $(E_1E_2)^c$, let $\bar X^n$ coincide with $\tilde X^n$ throughout.  Let $\bar M^{\mathrm{sc},n}$ and $\bar\beta^n$ be its time-reversed empirical measure and occupation measure.  Then the auxiliary and original processes agree until the  time $\sigma_3$.

Set $I_n\doteq 1_{E_1E_2}$.
Note that $(\bar \beta^n, \bar M^{\mathrm{sc},n}, I_n)$ forms a tight sequence in $\clp(S\times \R_+) \times C([0,T]:\clp(S)) \times \{0,1\}$. Suppose, the sequence converges weakly, along some subsequence to $(\bar\beta^*,\bar M^*,I)$. By a usual subsequential argument and appealing to Skorohod representation theorem, we can assume that the convergence holds a.s., along the full sequence.
Disintegrate
$\bar\beta^*(dx\,ds)=\bar\beta_s^*(dx)e^{-s}ds$.
The same estimate as used in the proof of \eqref{eq:410n}, now applied to the auxiliary process, gives on $\{I=1\}$, for every $f:S\to\R$ and $t_0\in[0,T]$,
\[
\int_{t_0}^T e^{-s}\sum_{x\in S}\sum_{y:y\ne x}
(f(y)-f(x))H_s(x,y)\bar\beta_s^*(x)\,ds=0.
\]
Consequently, for a.e. $s\in[0,T]$, on $\{I=1\}$,
\[
\sum_{x\in S}\bar\beta_s^*(x)H_s(x,y)=0,
\qquad y\in S.
\]
The irreducibility of $H_s$ and Property \ref{proper:2}(c) imply
$\bar\beta_s^*=\rho_s$ for a.e. $s\in[0,T]$ on $\{I=1\}$.

Equation \eqref{eq:RevResTimeEvEmpOccLower} and \eqref{eq:terminal-measure-joint} then give, on $\{I=1\}$,
\[
\bar M_t^*=M_T-\int_t^T\bar M_s^*\,ds+\int_t^T\rho_s\,ds,
\qquad 0\le t\le T.
\]
Recalling that the function $M$ satisfies the same equation we have $\bar M^*=M$ on $\{I=1\}$.  We thus conclude
\begin{multline}\label{eq:auxiliary-joint-convergence}
\E\left[1_{E_1E_2}\sup_{0\le s\le T}
\|\bar M_s^{\mathrm{sc},n}-M_s\|\right]\\
+\E\left[1_{E_1E_2}\left|
\int_{S\times[0,T]}g(x,s)\bar\beta^n(dx\,ds)
-\int_0^T e^{-s}\sum_xg(x,s)\rho_s(x)\,ds
\right|\right]\to 0
\end{multline}
for every bounded continuous $g$.

Next, note that, on $E_1E_2E_3^c$, the original and auxiliary processes agree up to the first exit from the $\cons_2/8$ neighborhood of $M$.  Hence
\[
E_1E_2E_3^c\subseteq
E_1E_2\cap\left\{
\sup_{0\le s\le T}\|\bar M_s^{\mathrm{sc},n}-M_s\|
\ge \cons_2/8\right\}.
\]
Equation \eqref{eq:auxiliary-joint-convergence} therefore proves
\eqref{eq:tracking-exit-vanishes}.  Furthermore, on $G_n$, the original and auxiliary processes agree throughout the tracking interval, so \eqref{eq:auxiliary-joint-convergence} also proves the convergence of the first term in  \eqref{eq:joint-good-convergence} to $0$. Finally, since $\bar \beta^*_s = \rho_s$, a.e. $s \in [0,T]$, on $\{I=1\}$, we have 
\eqref{eq:beta-good-convergence}.

\medskip\noindent\textit{Step 3: identification of the empirical flux.} It remains to prove the  convergence of the second term in \eqref{eq:joint-good-convergence} to $0$.
Fix $(i,j)\in A_0$.  Writing the compensated-PRM decomposition for the integral in \eqref{eq:ControlFlux}, we have
$\tilde R^{n,ij}=C_n^{ij}+Z_n^{ij}$,
where
\[
C_n^{ij}=\frac1{T_n}\int_0^{T_n}1_{\{\tilde X_s^n=i\}}
\left(\int_0^{Q_{ij}(\tilde L_s^n)}\alpha_{ij}^n(s,u)\,du\right)ds
\]
and $Z_n^{ij}$ is the corresponding compensated-PRM integral divided by $T_n$.  The effective jump rates are bounded by $c$, by $V(i,j)$, or by $H_s(i,j)$ on the three successive stages.  Since $H$ is bounded on $[0,T]$,
\begin{equation}\label{eq:flux-martingale-lower}
\E|Z_n^{ij}|^2\le \frac{C}{T_n}\to 0.
\end{equation}

On $G_n$, split $C_n^{ij}$ according to the three stages:
\begin{align*}
C_n^{ij}
=&\frac1{T_n}\int_0^{\sigma_1}1_{\{\tilde X_s^n=i\}}
Q_{ij}(\tilde L_s^n)\,ds+\frac{V(i,j)}{T_n}\int_{\sigma_1}^{\tau_n}
1_{\{\tilde X_s^n=i\}}\,ds\\
&\quad +\frac1{T_n}\int_{\tau_n}^{T_n}1_{\{\tilde X_s^n=i\}}
H_{t_n-\log s}(i,j)\,ds.
\end{align*}
The first term is at most $cr_1/T_n$ on $G_n$.  By \eqref{eq:V-ergodic-joint}, the second term converges in $L^1$ on $G_n$ to $e^{-T}q(i)V(i,j)$.  From the definitions of $q$ and $V$ (see \eqref{eq:rev1-1}, \eqref{eq:rev1-2}),
\begin{equation}\label{eq:tail-flux-identity}
e^{-T}q(i)V(i,j)
=\int_T^\infty e^{-s}\rho_s(i)H_s(i,j)\,ds.
\end{equation}
For the third term, the change of variables $u=t_n-\log s$ gives
\[
\frac1{T_n}\int_{\tau_n}^{T_n}1_{\{\tilde X_s^n=i\}}
H_{t_n-\log s}(i,j)\,ds
=\int_{S\times[0,T]}1_{\{x=i\}}H_u(i,j)\,\beta^n(dx\,du).
\]
Taking $g(x,u)=1_{\{x=i\}}H_u(i,j)$ in
\eqref{eq:beta-good-convergence}, and combining the result with
\eqref{eq:flux-martingale-lower} and \eqref{eq:tail-flux-identity}, yields, for all $(i,j) \in A_0$,
\[
\E\left[1_{G_n}\left|\tilde R^{n,ij}
-\int_0^\infty e^{-s}\rho_s(i)H_s(i,j)\,ds\right|\right]
\to 0.
\]
The integral in the above display is $\flu^{ij}$ by Property \ref{proper:2}(d).  This proves the convergence to $0$ of the second term in
\eqref{eq:joint-good-convergence} and completes the proof.
\end{proof}
}
We now control the contribution of $B_n$ on the set $E_1E_2E_3^c$.
\begin{lemma}
\label{lem:lem3prelim}
\newtext{
\[
\lim_{n\to \infty}\E[B_n1_{E_1E_2E_3^c}]=0.
\]
}
\end{lemma}
\begin{proof}
\newtext{
By Lemma \ref{lem:lem7.4},
$\PP(E_1E_2E_3^c)\to 0$.
Consequently,
\[
\E\left[|B_n-J_n|1_{E_1E_2E_3^c}\right]
\le \|h\|_\infty\PP(E_1E_2E_3^c)\to 0.
\]
Now consider the cost $J_n$.  The estimates used in Lemma
\ref{lem:lem2prelim} give, on $E_1E_2E_3^c$ (cf. \eqref{eq:rev1-4}, \eqref{eq:755nn})
\[
\tilde J^n[0,\sigma_1]=0,
\quad
\tilde J^n[\sigma_1,\sigma_1+r_2]
\le ar_2c_{v_1',v_2'(T)},\quad e^{-t_n}\tilde J^n[\sigma_1+r_2,e^{t_n-T}]
\le \delta/4.
\]
Now consider the cost incurred during
the tracking interval before the  time $\sigma_3$.  Set
$\vartheta_n=t_n-\log\sigma_3\in[0,T]$.  A change of variables gives, on
$E_1E_2E_3^c$,
\begin{align*}
e^{-t_n}\tilde J^n[e^{t_n-T},\sigma_3]
={}&\sum_{(i,j)\in A_0}\int_0^{\vartheta_n}e^{-s}
1_{\{\tilde Y^{\mathrm{sc},n}_s=i\}}
Q_{ij}(\tilde M^{\mathrm{sc},n}(s)) 
\ell\left(\frac{H_s(i,j)}{Q_{ij}(\tilde M^{\mathrm{sc},n}(s))}\right)\,ds.
\end{align*}
By the definition of $\sigma_3$ and the lower bound on $M$ in Lemma
\ref{lem:step2},
$\tilde M^{\mathrm{sc},n}_s(x)\ge \cons_2/4$ for every $x\in S$ and
$s\in[0,\vartheta_n]$.  Hence
\[
\inf_{s\in[0,\vartheta_n]}\inf_{(x,y)\in A_0}
Q_{xy}(\tilde M^{\mathrm{sc},n}(s))\ge k_Q\cons_2/4.
\]
Since $\vartheta_n\le T$, the estimate \eqref{eq:753nn} now yields
\begin{align*}
e^{-t_n}\tilde J^n[e^{t_n-T},\sigma_3]
&\le \int_0^T e^{-s}\sum_{(i,j)\in A_0}
\bigl[H_s(i,j)\log H_s(i,j) \\
&\hspace{7em}{}+H_s(i,j)\log^-(k_Q\cons_2/4)+c_Q\bigr]\,ds
\doteq C_T<\infty,
\end{align*}
where finiteness follows from \eqref{eq:basinteg}; in particular, $C_T$ does
not depend on $n$.  Moreover, $\tilde J^n[\sigma_3,T_n]=0$.  Therefore
\begin{align*}
\E[J_n1_{E_1E_2E_3^c}]
&\le e^{-t_n}ar_2c_{v_1',v_2'(T)}
+(\delta/4+C_T)\PP(E_1E_2E_3^c)
\to 0.
\end{align*}
Combining the last two estimates proves the result.
}
\end{proof}
Finally, we control $B_n$ on $E_1E_2E_3$.
\begin{lemma}
\label{lem:lem4prelim}
\begin{align*}
\limsup_{n\to \infty} \E[B_n 1_{E_1E_2E_3}]
&\le 5\delta + \deletetext{\inf_{\gamma \in \clp(S)} [h(\gamma) + I(\gamma)]}\newtext{\inf_{(\gamma,\flu) \in \clp(S)\times \R_+^a} [h(\gamma,\flu) + I(\gamma,\flu)]}.
\end{align*}
\end{lemma}
\begin{proof}
\newtext{
Let $G_n=E_1E_2E_3$.  The joint convergence in Lemma
\ref{lem:lem7.4} gives
\begin{equation}\label{eq:h-joint-good}
\E\left[1_{G_n}\left|
h(\tilde M_0^{\mathrm{sc},n},\tilde R^n)-h(\gamma,\flu)
\right|\right] \to 0.
\end{equation}
Also, by \eqref{eq:tracking-exit-vanishes} in Lemma \ref{lem:lem7.4},
\[
\limsup_{n\to\infty}\PP(G_n^c)
\le \PP(E_1^c)+\PP(E_1E_2^c)
+\limsup_{n\to\infty}\PP(E_1E_2E_3^c)
\le 2\eps_1.
\]
  Hence
\begin{equation}\label{eq:h-good-bound-joint}
\limsup_{n\to\infty}
\E\left[h(\tilde M_0^{\mathrm{sc},n},\tilde R^n)1_{G_n}\right]
\le h(\gamma,\flu)+2\eps_1\|h\|_\infty.
\end{equation}

We next consider the cost $J_n$.  As in \eqref{eq:755nn},
\begin{equation}\label{eq:pretracking-cost-good}
\limsup_{n\to\infty}e^{-t_n}
\E\left[\tilde J^n[0,e^{t_n-T}]1_{G_n}\right]\le\delta/4.
\end{equation}
The contribution to the cost $J_n$ from the tracking interval is
\[
K_n\doteq
\int_{S\times[0,T]}\sum_{y:y\ne x}
Q_{xy}(\tilde M_s^{\mathrm{sc},n})
\ell\left(\frac{H_s(x,y)}{Q_{xy}(\tilde M_s^{\mathrm{sc},n})}\right)
\beta^n(dx\,ds).
\]
On $G_n$, both $M_s$ and $\tilde M_s^{\mathrm{sc},n}$ are uniformly
bounded away from the boundary of the simplex and therefore  the continuity estimate
\eqref{eq:511nn}, followed by \eqref{eq:joint-good-convergence}, gives, as $n\to \infty$,
\begin{equation*}
\E\left[1_{G_n}\left|K_n-
\int_{S\times[0,T]}\sum_{y:y\ne x}
Q_{xy}(M_s)
\ell\left(\frac{H_s(x,y)}{Q_{xy}(M_s)}\right)
\beta^n(dx\,ds)\right|\right]
\to 0.
\end{equation*}
The function
\[
(x,s)\longmapsto \sum_{y:y\ne x}Q_{xy}(M_s)
\ell\left(\frac{H_s(x,y)}{Q_{xy}(M_s)}\right)
\]
is bounded and continuous on $S\times[0,T]$.  Applying
\eqref{eq:beta-good-convergence} to this function therefore yields
\begin{equation}\label{eq:tracking-cost-joint}
\E\left[1_{G_n}\left|K_n-K\right|\right] \to 0,
\end{equation}
where
\[
K=\int_0^T e^{-s}\sum_{x\in S}\rho_s(x)
\sum_{y:y\ne x}Q_{xy}(M_s)
\ell\left(\frac{H_s(x,y)}{Q_{xy}(M_s)}\right)ds
\le \tilde J(\rho,H).
\]

Combining \eqref{eq:h-good-bound-joint},
\eqref{eq:pretracking-cost-good}, and
\eqref{eq:tracking-cost-joint}, we obtain
\[
\limsup_{n\to\infty}\E[B_n1_{G_n}]
\le h(\gamma,\flu)+\tilde J(\rho,H)
+2\eps_1\|h\|_\infty+\delta/4.
\]
Finally, \eqref{eq:506n}, Lemma \ref{lem:nondeg}, Lemma
\ref{lem:step2}, and the choice \eqref{eq:delta8} imply
\[
h(\gamma,\flu)+\tilde J(\rho,H)
+2\eps_1\|h\|_\infty+\delta/4
\le 5\delta+
\inf_{(\gamma',\flu')\in\clp(S)\times\R_+^a}
[h(\gamma',\flu')+I(\gamma',\flu')].
\]
This proves the result.
}
\end{proof}

\subsection{Completing the lower bound proof}
We can now complete the proof of the lower bound in \eqref{eq:Laplacelow}. Note that, from \eqref{eq:528nn},
\newtext{
and Lemmas \ref{lem:lem1prelim},
\ref{lem:lem2prelim}, \ref{lem:lem3prelim}, and
\ref{lem:lem4prelim},
\begin{multline*}
\limsup_{n\to\infty}-\frac1{T_n}\log
\E\left[e^{-T_nh(L_{T_n},R_{T_n})}\right]
\le \limsup_{n\to\infty}\E[B_n]\\
\le 2\eps_1\|h\|_\infty+\delta/4
+5\delta+
\inf_{(\gamma,\flu)\in\clp(S)\times\R_+^a}
[h(\gamma,\flu)+I(\gamma,\flu)]\\
\le 6\delta+
\inf_{(\gamma,\flu)\in\clp(S)\times\R_+^a}
[h(\gamma,\flu)+I(\gamma,\flu)],
\end{multline*}
where the last inequality uses \eqref{eq:delta8}.  Since $\delta>0$ is
arbitrary, the lower bound \eqref{eq:Laplacelow} follows.
}

\section{Compactness of Level Sets}
\label{sec:cptlev}
In this section we prove that the function $I$ defined in \eqref{eq:IGamma} is a rate function, namely it has compact level sets.
\begin{proposition}
For every $L<\infty$, the set $K_L \doteq \{(\gamma, \flu) \in \clp(S) \times \R_+^a: I(\gamma, \flu)\le L\}$ is compact.
\end{proposition}
\begin{proof}
Fix $L<\infty$ and let $K_L$ be as in the statement of the proposition. 
Let $\{(\gamma^n, \flu^n)\}_{n\in \N}$ be a sequence in $K_L$. We need to show that this sequence is relatively compact and if (along a subsequence) $(\gamma^n, \flu^n) \to (\gamma, \flu)$ then, 
$(\gamma, \flu) \in K_L$.
For each $n$, let $\theta^n \in \cls_{\gamma^n, \flu^n}$ be such that
$J(\theta^n) \le L+ n^{-1}$.
Recalling part (c) of Property \ref{proper:1}, for $s\ge 0$, we disintegrate $\theta^n(s)$ as
$$\theta^n(s)(dx\, dr\, du) = \varrho^n_s(dr\mid u,x) [\theta^n(s)]_1(dx) c^{-1} du.$$
We will denote $M^{\theta^n}$ as $M^n$.
Let, for $(x,y) \in A_0$ and $(s,u) \in \R_+ \times [0,c]$
\be 
\begin{split}
v^n_{xy}(s,u) &\doteq \\
&\hspace{-1cm}1_{[\theta^n(s)]_1(x) \neq 0}1_{[0, Q_{xy}(M^n_s)]}(u) \int_{\R_+^a}  r_{xy} \varrho^n_s(dr\mid u,x)+ (1-1_{[\theta^n(s)]_1(x) \neq 0}1_{[0, Q_{xy}(M^n_s)]}(u)) 
\end{split}
\ee
when the right side is finite.
Note that from part (b) of Property \ref{proper:1} the right side in the above expression is finite for a.e. $(s,u) \in \R_+ \times [0,c]$ and all $(x,y) \in A_0$. We set $v^n_{xy}(s,u)=1$ when the right side in the above display is infinite.

Now define $\tilde \theta^n \in \B(\R_+, \clp(S\times \R^a_+\times [0,c]))$ as
$$
\tilde \theta^n_s(dx\, dr\, du) = \delta_{v^n(s,u)}(dr)[\theta^n(s)]_1(dx) c^{-1} du.
$$
Note that since $[\tilde \theta^n(s)]_1=[\theta^n(s)]_1$, part (a) of 
Property \ref{proper:1} holds with $\gamma^n$ and from part (b) of
Property \ref{proper:1} we have that $M^n= M^{\theta^n} = M^{\tilde \theta^n}$.
Also by definition of $\tilde \theta^n$ it is easily seen that the identity in the last line of part (b) of Property \ref{proper:1} holds with $(M, \theta)$ replaced by $(M^n, \tilde \theta^n)$. By definition, part (c) of Property \ref{proper:1} holds with $\theta$ replaced by $\tilde \theta^n$.

Next note that, since $\theta^n \in \cls_{\deletetext{\gamma, \flu}\newtext{\gamma^n, \flu^n}}$, from part (d) of Property \ref{proper:1} we have that, for $(x,y) \in A_0$,
\begin{align*}
\flu^{n,xy} &=c\int_{\R^a_+\times [0,c]\times \R_+} e^{-s} [\theta^n(s)]_1(x) 1_{[0, Q_{xy}(M^n_s)]}(u) r_{xy} \eta^n(s)(dr\, du\mid x) ds\\
&= c\int_{\R^a_+\times [0,c]\times \R_+} e^{-s} [\theta^n(s)]_1(x) 1_{[0, Q_{xy}(M^n_s)]}(u) r_{xy} \varrho^n_s(dr\mid u,x) c^{-1} du ds\\
&= c\int_{\R^a_+\times [0,c]\times \R_+} e^{-s} [\theta^n(s)]_1(x) 1_{[0, Q_{xy}(M^n_s)]}(u) r_{xy} \delta_{v^n(s,u)} (dr)  c^{-1} du ds\\
&= c\int_{\R^a_+\times [0,c]\times \R_+} e^{-s} [\tilde \theta^n(s)]_1(x) 1_{[0, Q_{xy}(M^n_s)]}(u) r_{xy} \tilde \eta^n(s)(dr\, du\mid x) ds,
\end{align*}
where
$\tilde \theta^n(s)(dx\, dr\, du) = \tilde \eta^n(s)(dr\, du \mid x)[\tilde \theta^n(s)]_1(dx)$ is the disintegration of $\tilde \theta^n(s)$. Thus we  have that part (d) of Property \ref{proper:1} hold with $(\theta, \flu)$ replaced by $(\tilde \theta^n, \flu^n)$.

Combining the above observations we have that $\tilde \theta^n \in \cls_{\deletetext{\nu^n}\newtext{\gamma^n}, \flu^n}$.
Also, using the definition of $\tilde \theta^n$, and convexity of $\ell$ we see that
\begin{align}
J(\tilde \theta^n) &\le c\int_{\Xi} e^{-s} \sum_{y:y\neq x} \ell(r_{xy})\,  \tilde\theta^n(s)(dx\, dr\, du) ds\\
&\le c\int_{\Xi} e^{-s} \sum_{y:y\neq x}1_{[0, Q_{xy}(M_s)]}(u) \ell(r_{xy})\,  \theta^n(s)(dx\, dr\, du) ds \le M+ n^{-1}
\le M+1.\label{eq:959nn}
\end{align}
Now define $\nu^n \in \clp(S\times \R^a_+\times [0,c]\times \R_+)$ as 
$\nu^n(dx\,dr\,du\, ds) = \tilde \theta^n_s(dx\,dr\,du)e^{-s} ds$.
From the above estimate on $J(\tilde \theta^n)$ and superlinearity of $\ell$ we see that $\{\nu^n\}$ is a tight sequence of probability measures on $S\times \R^a_+\times [0,c]\times \R_+$
and, if $\nu^n \to \nu$ along a subsequence, then for any a.e. continuous and bounded $f: S\times \R_+\times [0,c] \to \R$
\be \label{eq:1019nnn}
\int_{\Xi} \sum_{y:y\neq x} f(x,s,u) r_{xy} \nu^n(dx\,dr\,du\, ds) \to
\int_{\Xi} f(x,s,u) r_{xy} \nu(dx\,dr\,du\, ds) \ee
as $n\to \infty$.

Note that we can write
\be\label{eq:1041nn}
\flu^{n,xy} = c\int_{\Xi}  1_{[0, Q_{xy}(M^n_s)]}(u) r_{xy}  \nu^n(dx\, dr\, du\, ds)
\ee
and thus the above tightness and uniform integrability property  says that the sequence $\{\flu^{n}\}$ is tight in $\R_+^a$. Also, as noted above, the sequence $\{\gamma^n\}$ is automatically tight and the tightness of $\{M^n\}$ in $C([0,T]: \clp(S))$ follows easily as well.
Now suppose along a subsequence (denoted again as $n$)
$$(\nu^n, M^n, \gamma^n, \flu^n) \to (\nu, M, \gamma, \flu).$$
Then using the lower semicontinuity of $\ell$ and Fatou's lemma, we have from \eqref{eq:959nn} that
\begin{align}
c\int_{\Xi}  \sum_{y:y\neq x} \ell(r_{xy})\,  \nu(dx\, dr\, du\, ds)
&\le
\liminf_{n\to \infty} 
c\int_{\Xi}  \sum_{y:y\neq x} \ell(r_{xy})\,  \nu^n(dx\, dr\, du\, ds)\\
&=
\liminf_{n\to \infty}  c\int_{\Xi} e^{-s} \sum_{y:y\neq x} \ell(r_{xy})\,  \tilde\theta^n(s)(dx\, dr\, du) ds \le M.
\label{eq:1016nn}
\end{align}
Note that $\nu$ can be disintegrated as
$\nu(dx\, dr\, du\, ds) = \theta(s)(dx\, dr\, du) e^{-s} ds$ for some $\theta \in \B(\R_+, \clp(S\times \R^a_+\times [0,c]))$.
To complete the proof it suffices to show that
$\theta \in \cls_{\gamma, \flu}$.

Note that part (a) of Property \ref{proper:1} is immediate from the convergence of $(\gamma^n, \nu^n) \to (\gamma, \nu)$ and on observing that
$\int_0^{\infty} e^{-s} [\tilde \theta^n(s)]_1 ds = [\nu^n]_1$ and
$\int_0^{\infty} e^{-s} [\theta(s)]_1 ds = [\nu]_1$.

A similar observation shows that the limit point $M$ introduced above satisfies the identity
$$M(t) = e^t \int_{t}^{\infty} e^{-s} [\theta(s)]_1 ds, \; t \ge 0.$$
From the bound in \eqref{eq:1016nn}, we see that the finiteness property in part (b) of Property \ref{proper:1} holds.
We now verify the second statement in part (b).
Since $\tilde \theta^n \in \cls_{\deletetext{\nu^n}\newtext{\gamma^n}, \flu^n}$, from part (b)
of Property \ref{proper:1} we have that for any continuous and bounded $g: \R_+ \to \R$
\be \label{eq:452nn}
\int_{\Xi} g(s)\sum_{y: y \neq x} (f(y) - f(x)) 1_{[0, Q_{xy}(M^n_s)]}(u) r_{xy} \nu^n(dx\, dr\, du\, ds) = 0.
\ee
Also, by convergence of $M^n \to M$, and from arguments similar to those used in \eqref{eq:740nn}, we have that
\begin{align*}
&\int_{\Xi} g(s) |1_{[0, Q_{xy}(M^n_s)]}(u) - 
1_{[0, Q_{xy}(M_s)]}(u)| r_{xy} \nu^n(dx\, dr\, du\, ds)
\to 0,
\end{align*}
as $n\to \infty$. Here we have also used the bound in \eqref{eq:1016nn}, the superlinearity of $\ell$ and the boundedness of $M^n, M$.
From the above convergence, \eqref{eq:452nn},  and the convergence observed in \eqref{eq:1019nnn} we now see that
\begin{align*}
0=&\lim_{n\to \infty}
\int_{\Xi} g(s)\sum_{y: y \neq x} (f(y) - f(x)) 1_{[0, Q_{xy}(M^n_s)]}(u) r_{xy} \nu^n(dx\, dr\, du\, ds) \\
&= \lim_{n\to \infty}
\int_{\Xi} g(s)\sum_{y: y \neq x} (f(y) - f(x)) 1_{[0, Q_{xy}(M_s)]}(u) r_{xy} \nu^n(dx\, dr\, du\, ds) \\
&= \int_{\Xi} g(s)\sum_{y: y \neq x} (f(y) - f(x)) 1_{[0, Q_{xy}(M_s)]}(u) r_{xy} \nu(dx\, dr\, du\, ds).
\end{align*}
Since $g$ is arbitrary, recalling the definition of $\theta$, we now see that the last statement in part (b)
of Property \ref{proper:1} holds. Part (c) of Property \ref{proper:1} is immediate from the fact that $[\nu^n]_{134}(dx\,du\,ds)= e^{-s} ds [\theta^n(s)]_1(dx) c^{-1}du$. Indeed, this identity says that $[\nu^n]_{13}(dx\, du) = [\nu^n]_{1}(dx) c^{-1}du$ and sending $n\to \infty$ we have part (c) of Property \ref{proper:1}.

Finally a similar calculation as for the verification of part (b) now shows that, as $n\to \infty$
$$c\int_{\Xi}  1_{[0, Q_{xy}(M^n_s)]}(u) r_{xy}  \nu^n(dx\, dr\, du\, ds) \to  c\int_{\Xi}  1_{[0, Q_{xy}(M_s)]}(u) r_{xy}  \nu(dx\, dr\, du\, ds).
$$
From \eqref{eq:1041nn}, and since $\flu^n \to \flu$, we now get that part (d) of  Property \ref{proper:1} is also satisfied.

Thus we have shown that $\theta \in \cls_{\gamma, \flu}$ and the result follows. 
\end{proof}
\enlargethispage{2\baselineskip}
\section{Illustrative Examples}
\label{sec:examples}
Several examples in the discrete-time setting of self-interacting chains are presented in \cite[Section 8]{Budhiraja2025}. These have natural continuous-time counterparts, which we do not detail here. Instead, we turn to a set of examples motivated by applications not covered in \cite{Budhiraja2025}. We begin with the non-interacting case, which falls within the framework of the Donsker–Varadhan theory.
\begin{description}
    \item[\textbf{Donsker--Varadhan LDP}] Suppose $Q^0 \in \mathcal{K}(S)$. The map $Q: \mathcal{P}(S) \rightarrow \mathcal{K}_0(S)$ with $Q(\gamma) = Q^0$ for all $\gamma \in \mathcal{P}(S)$ is obviously continuous and satisfies Assumptions \ref{assu:lowbd} with $\mathcal{A}^* =Q^0$. Theorem \ref{thm:main} in this case gives the large deviation principle for the occupation measure and empirical flux of an irreducible continuous-time Markov jump process on a finite state space (cf.\ \cite{donvar1,Bertini2015,Barato2015}). We now argue that the rate function \eqref{eq:Alternative} in this case 
    coincides with the Donsker-Varadhan rate function $\IDW_{Q^0}$ given in the Introduction. Fix $(\gamma, \flu) \in \clp(S) \times \R_+^{a}$. We first show that
    $\bar I(\gamma, \flu) \le \IDW_{Q^0}(\gamma, \flu)$.  Assume, without loss of generality that $\IDW_{Q^0}(\gamma, \flu)<\infty$.
    Define, for $s\ge 0$, $\rho(s) \doteq \gamma$, and
    $$H_s(x,y) = \gamma_x^{-1} \flu^{xy}1_{\gamma_x \neq 0} + Q^0_{xy}1_{\gamma_x = 0}, \; (x,y) \in A,$$
    and $H_s(x,x) \doteq - \sum_{y: y \neq x} H_s(x,y)$, for $x \in S$.

    Note that Property \ref{proper:2}(a) and (d) hold with the above definition of $(\rho_s, H_s)$. Also, $M(s)$ defined in part (b) of Property \ref{proper:2} in this case equals $\gamma$ for all $s\ge 0$. Furthermore, since $\IDW_{Q^0}(\gamma, \flu)<\infty$, we must have that, for $(x,y) \in A$, if $Q^0_{xy}=0$, $\flu^{xy}=0$ and consequently $H_s(x,y) =0$ as well. This verifies part (b) of Property \ref{proper:2}.  Again, since $\IDW_{Q^0}(\gamma, \flu)<\infty$, we have that 
    \be \label{eq:142nn} \sum_{x: x \neq y}\flu^{yx} = \sum_{x: x \neq y} \flu^{xy} \mbox{ for all } y \in S.\ee
    This says that $\sum_{y\in S} \gamma_xH_s(x,y) =0$ for all $y\in S$ which verifies (c). Thus we have shown that
    $(\rho, H) \in \tilde \cls_{\gamma, \flu}$. Also, clearly, $\tilde J(\rho, H) = \IDW_{Q^0}(\gamma, \flu)$. The inequality
    $\bar I(\gamma, \flu) \le \IDW_{Q^0}(\gamma, \flu)$ now follows.

    We now argue the reverse inequality: $\IDW_{Q^0}(\gamma, \flu)\le \bar I(\gamma, \flu)$.
    Fix $\delta>0$ and suppose $(\rho, H) \in \tilde \cls_{\gamma, \flu}$ is such that
    \be \tilde J(\rho, H) \le \bar I(\gamma, \flu) + \delta. \ee
    Using the convexity property in Lemma \ref{lem:biconbex}
    \begin{align*}
\tilde J(\rho, H) &= \int_{[0,\infty)}  \sum_{(x,y)\in A_0} e^{-s}  Q^0_{xy} \rho_s(x) \ell \left(\frac{\rho_s(x)H_s(x,y)}{\rho_s(x)Q^0_{xy}}\right) \,   ds\\
&\ge \sum_{(x,y)\in A_0} \left(\int_{[0,\infty)}e^{-s}  Q^0_{xy} \rho_s(x) ds\right)
\ell \left( \frac{\int_{[0,\infty)}e^{-s}\rho_s(x)H_s(x,y) ds}{\int_{[0,\infty)}e^{-s}  Q^0_{xy} \rho_s(x) ds}\right)\\
&= \sum_{(x,y)\in A_0} Q^0_{xy} \gamma_x \ell \left(\frac{\flu^{xy}}{\gamma_xQ^0_{xy}}\right),
    \end{align*}
    where in the last line we have used Property \ref{proper:2}(a) and (d).
    Also, from Property \ref{proper:2}(c) it is easily checked that \eqref{eq:142nn} holds.  This shows that the last term in the above display equals $\IDW_{Q^0}(\gamma, \flu)$.  Thus we have shown that
    $\IDW_{Q^0}(\gamma, \flu) \le \bar I(\gamma, \flu) + \delta$. Since $\delta>0$ is arbitrary, the desired reverse inequality now follows.
    
    \item[\textbf{Affine Autochemotaxis: Self-Attraction}] Autochemotaxis is the ability of organisms to communicate through local secretions that form
    chemical trails. We consider $S = \set{1, \cdots, d}$ the set of all locations that an autochemotactic organism can visit. Suppose $Q^0 \in \mathcal{K}(S)$. Define for $K \in (0,\infty)${, interpreted as a relative measure of the chemicals' lifetime}, $Q: \mathcal{P}(S) \rightarrow \mathcal{K}_0(S)$ as
    $$
    Q_{ij}(\gamma) = Q^0(i,j) (\gamma(j) K  + 1), \quad (i,j) \in A,
    $$ 
    with $Q_{ii}(\gamma) = -\sum_{j: j\neq i} Q_{ij}(\gamma)$. The rate matrix $Q(\gamma)$ is such that the larger the empirical weight at $j$, i.e., the more the autochemotactic organism visit state $j$, the higher the jump rate into it. Note that $Q$ is affine,  irreducible and satisfies Assumption \ref{assu:lowbd} with $\mathcal{A}^* =Q^0$ (see discussion in Remark \ref{rem:whenhold}).  Thus Theorem \ref{thm:main} applies in this setting.

    \item[\textbf{Packet Routing: Source-Destination Congestion}] Consider a network of routers represented by the finite set $S = \{1,\dots,d\}$. 
    A tracer packet moves through this network, being forwarded from router to router according to baseline routing rates $Q^0 \in \mathcal{K}(S)$. We measure the congestion of a router by the fraction of time the packet has spent at that router up to the current time. The congestion in the network, in turn, modifies the routing rates of the packet so that if congestion at routers $i$ and $j$ is high, the rate at which the packet moves from router $i$ to $j$ is low. This motivates a self-interacting rate matrix 
    $$
    Q_{ij}(\gamma) \coloneqq \bigl(1 - \alpha_i \gamma(i) - \beta_j \gamma(j)\bigr) Q^0_{ij}, \quad (i,j) \in S \times S, \; i \neq j,
    $$ 
    with $Q_{ii}(\gamma) = -\sum_{j: j \neq i} Q_{ij}(\gamma)$ for $\gamma \in \mathcal{P}(S)$. 
    Here $\alpha_i,\beta_j \ge 0$ are congestion parameters satisfying $\alpha_i + \beta_j \le 1-\varepsilon$ for some $\varepsilon > 0$. This model serves as a useful  probe for network congestion that does not require the modeling of the whole network traffic explicitly.   Since $\gamma \mapsto Q(\gamma)$ is affine and for $(i,j) \in A$, $Q_{ij}(\gamma) \;\ge\; \varepsilon Q^0_{ij}$, we see that Assumption \ref{assu:lowbd} is satisfied (see discussion in Remark \ref{rem:whenhold}) and thus Theorem  \ref{thm:main} applies for this model.

    \item[\textbf{Chemical State Switching with Catalytic Channels}] Let $S = \{1,\dots,d\}$ denote a finite set of chemical states available to a single molecule. We track the trajectory of a single molecule through the state space $S$.
    At each time $t$, if the molecule is in state $i \in S$, it may switch to state $j \in S$ through one of several possible \emph{catalytic channels} labelled by $k \in S$. Physically, this corresponds to a (different) molecule in state $k$ acting as a catalyst for the reaction.        Formally, we write
    $$
    i \xrightarrow[\;k\;]{Q^{(k)}_{ij}} j ,
    $$
    to mean that the transition $i \to j$ occurs via the $k$-th channel at rate $Q^{(k)}_{ij}$. 
    The channel index $k$ plays the role of a catalyst: it facilitates the transformation $i \to j$ but is not itself consumed.

     We use the empirical occupation mass $\gamma(k)$ as a surrogate measure for abundance of chemical of type $k$ available and thus  the reaction in channel $k$ is enhanced as $\gamma(k)$ increases.
    The effective rate matrix is then defined by
    $$
    Q_{ij}(\gamma) \coloneqq \sum_{k \in S} \gamma(k) \, Q^{(k)}_{ij}, 
    \qquad (i,j) \in S \times S,\; i \neq j,
    $$
    with 
    $$
    Q_{ii}(\gamma) = - \sum_{i \neq j} Q_{ij}(\gamma).
    $$ 
    
    Thus the trajectory of the molecule is a  self-interacting jump process on $S$, whose transition structure depends on its own history via $\gamma$. We suppose that for some $\cla^* \in \clk(S)$ and the associated $A_0$ as in Assumption \ref{assu:lowbd}, $Q^{(k)}_{ij}>0$ for all $k\in S$ and $(i,j) \in A_0$.
    The assumptions of Theorem~\eqref{thm:main} are satisfied: Assumption~\ref{assu:lowbd} (b) holds trivially, Assumption~\ref{assu:lowbd} (c) holds as well, and in particular $Q(\gamma)$ is irreducible for all $\gamma \in \clp(S)$. From the discussion in Remark \ref{rem:whenhold} we see that parts (a) and (d) are satisfied as well.
    Hence this family of models satisfies all the conditions of Theorem~\eqref{thm:main}.

\end{description}

\section*{Acknowledgments}
A.B.\ was partially supported by NSF DMS-2152577, NSF  DMS-2134107, NSF DMS-2506010. F.C.\ gratefully acknowledges support from the Blanceflor Foundation for a research visit to the Department of Statistics and Operations Research at the University of North Carolina, Chapel Hill, where this work was initiated, and thanks A.B.\ for additional financial support during the visit. This work was also supported by a Leverhulme Early Career Fellowship (ECF-2025-482) and EPSRC grant no.\ EP/V031201/1.

\bibliographystyle{plain}
\bibliography{cas-refs}

@article{benclo,
	author =  "M.~Bena{\"{i}}m and B.~Cloez",
	title = "A stochastic approximation approach to quasi-stationary distributions on finite spaces",
	journal = "Electron. Communications in Probability",
	volume = "20",
	year = "2015",
	pages = "1--13"}

@article{AFP88,
title = "Two applications of urn processes: the fringe analysis of search trees and the simulation of quasi-stationary distributions of {M}arkov chains",
journal = "Probability in the Engineering and Informational Sciences",
volume = "2",
number = "3",
pages = "293--307",
year = "1988",
author = "D. Aldous and B. Flannery and J.L. Palacios"
}

@book{buddupbook,
  title={Analysis and Approximation of Rare Events: Representations and Weak Convergence Methods},
  author={A.~Budhiraja and P.~Dupuis},
  isbn={9781493995790},
  series={Probability Theory and Stochastic Modelling},
  url={https://books.google.com/books?id=UY2oDwAAQBAJ},
  year={2019},
  publisher={Springer US}
}

@article{donvar1,
  author = "M.D. Donsker and S.R.S. Varadhan",
  title = "Asymptotic evaluation of certain {M}arkov process
  expectations for large time, {I}",
  journal = "Communications in Pure and Applied  Mathematics",
  volume = 28,
  year = 1975,
  pages = "1--47" }

@article{budhiraja2011variational,
  title={Variational representations for continuous time processes},
  author={A.~Budhiraja and P.~Dupuis and V.~Maroulas},
  journal={Annales de l'IHP Probabilit{\'e}s et statistiques},
  volume={47},
  number={3},
  pages={725--747},
  year={2011}
}

@article{HuaLiuKai,
  url = {https://arxiv.org/abs/2206.12801},
  author = {X.~Huang and Y.~Liu and K.~Xiang},
  journal = {arXiv:2206.12801},
  title = {Large deviation principle for empirical measures of once-reinforced random walks on finite graphs},
  publisher = {arXiv},
  year = {2022}
}

@article{BudWat3,
title = {Empirical measure large deviations for reinforced chains on finite spaces},
journal = {Systems \& Control Letters},
volume = {169},
pages = {105379},
year = {2022},
issn = {0167-6911},
doi = {https://doi.org/10.1016/j.sysconle.2022.105379},
url = {https://www.sciencedirect.com/science/article/pii/S0167691122001566},
author = {A.~Budhiraja and A.~Waterbury},
}

@article{blanchet2016, 
title={Analysis of a stochastic approximation algorithm for computing quasi-stationary distributions}, 
volume={48}, 
DOI={10.1017/apr.2016.28}, number={3}, 
journal={Advances in Applied Probability}, 
publisher={Cambridge University Press}, 
author={Blanchet, J. and Glynn, P. and Zheng, S.}, 
year={2016}, 
pages={792–811}}

@article{sch01,
  title={Urn models, replicator processes, and random genetic drift},
  author={S.J.~Schreiber},
  journal={SIAM Journal on Applied Mathematics},
  volume={61},
  number={6},
  pages={2148--2167},
  year={2001},
  publisher={SIAM}
}

@article{sinliv,
  title={The rock--paper--scissors game and the evolution of alternative male strategies},
  author={B.~Sinervo and C.M.~Lively},
  journal={Nature},
  volume={380},
  number={6571},
  pages={240--243},
  year={1996},
  publisher={Nature Publishing Group UK London}
}

@article{BenClo18,
 author = {M.~Bena\"{i}m and B.~Cloez and F.~Panloup},
 journal = {The Annals of Applied Probability},
 number = {4},
 pages = {2370--2416},
 title = {STOCHASTIC APPROXIMATION OF QUASI-STATIONARY DISTRIBUTIONS ON COMPACT SPACES AND APPLICATIONS},
 volume = {28},
 year = {2018}
}

@article{Franchini17,
title = {Large deviations for generalized {P}olya urns with arbitrary urn function},
journal = {Stochastic Processes and their Applications},
volume = {127},
number = {10},
pages = {3372-3411},
year = {2017},
doi = {https://doi.org/10.1016/j.spa.2017.02.010},
author = {S.~Franchini}
}

@article{Zhang14,
title = {Large deviations in the reinforced random walk model on trees},
journal = {Probability Theory and Related Fields},
volume = {160},
pages = {655-678},
year = {2014},
doi = {https://doi.org/10.1007/s00440-013-0537-5},
author = {Y.~Zhang}
}

@article{budfrawat2,
	author = {{A.~Budhiraja} and {N.~Fraiman} and {A.~Waterbury}},
	title = {Approximating quasi-stationary distributions with interacting reinforced random walks},
	DOI= "10.1051/ps/2021019",
	url= "https://doi.org/10.1051/ps/2021019",
	journal = {ESAIM: PS},
	year = 2022,
	volume = 26,
	pages = "69-125",
}

@article{Coppersmith1986,
   author = {D.~Coppersmith and P. ~Diaconis},
   journal = {Unpublished manuscript},
   title = {Random walks with reinforcement},
   year = {1986}
}

@phdthesis{Pemantle1988,
   author = {R.~Pemantle},
   school = {Massachusetts Institute of Technology},
   title = {Random processes with reinforcement},
   url = {https://www2.math.upenn.edu/~pemantle/papers/thesis00-01.pdf},
   year = {1988}
}

@article{Pemantle1992,
   author = {R.~Pemantle},
   issn = {01788051},
   issue = {1},
   journal = {Probability Theory and Related Fields},
   pages = {117-136},
   publisher = {Springer-Verlag},
   title = {Vertex-reinforced random walk},
   volume = {92},
   url = {https://link.springer.com/article/10.1007/BF01205239},
   year = {1992}
}

@article{Amit1983,
   author = {D.~ J.~Amit and G.~Parisi and L.~Peliti},
   doi = {10.1103/PhysRevB.27.1635},
   issn = {01631829},
   issue = {3},
   journal = {Physical Review B},
   month = {2},
   pages = {1635},
   publisher = {American Physical Society},
   title = {Asymptotic behavior of the "true" self-avoiding walk},
   volume = {27},
   url = {https://journals.aps.org/prb/abstract/10.1103/PhysRevB.27.1635},
   year = {1983}
}

@article{Hoppe1984,
   author = {F.~M.~Hoppe},
   doi = {10.1007/BF00275863/METRICS},
   issn = {03036812},
   issue = {1},
   journal = {Journal of Mathematical Biology},
   month = {8},
   pages = {91-94},
   publisher = {Springer-Verlag},
   title = {Pólya-like urns and the {E}wens' sampling formula},
   volume = {20},
   url = {https://link.springer.com/article/10.1007/BF00275863},
   year = {1984}
}

@article{Tria2014,
   author = {F. Tria and V. Loreto and V. D.P. Servedio and S. H. Strogatz},
   doi = {10.1038/srep05890},
   issn = {2045-2322},
   issue = {1},
   journal = {Scientific Reports},
   month = {7},
   pages = {1-8},
   pmid = {25080941},
   publisher = {Nature Publishing Group},
   title = {The dynamics of correlated novelties},
   volume = {4},
   url = {https://www.nature.com/articles/srep05890},
   year = {2014}
}

@article{Iacopini2020,
   author = {I.~Iacopini and G.~ Di Bona and E.~ Ubaldi and V.~Loreto and V.~Latora},
   doi = {10.1103/PhysRevLett.125.248301},
   journal = {Physical Review Letters},
    url ={https://doi.org/10.1103/PhysRevLett.125.248301},
   title = {Interacting Discovery Processes on Complex Networks},
   volume = {125},
   year = {2020}
}

@article{Diaconis2000,
   author = {P.~Diaconis and S.~Holmes and R.~M.~Neal},
   doi = {10.1214/AOAP/1019487508},
   issn = {1050-5164},
   issue = {3},
   journal = {Annals of Applied Probability},
   month = {8},
   pages = {726-752},
   publisher = {Institute of Mathematical Statistics},
   title = {Analysis of a nonreversible {M}arkov chain sampler},
   volume = {10},
   url = {https://projecteuclid.org/journals/annals-of-applied-probability/volume-10/issue-3/Analysis-of-a-nonreversible-Markov-chain-sampler/10.1214/aoap/1019487508.full},
   year = {2000}
}

@article{Essler2024,
   author = {F.~H.L.~Essler and W.~Krauth},
   doi = {10.1103/PHYSREVX.14.041035/E041035_28.EPS/MEDIUM},
   issn = {21603308},
   issue = {4},
   journal = {Physical Review X},
   month = {10},
   pages = {041035},
   publisher = {American Physical Society},
   title = {Lifted {T}{A}{S}{E}{P}: A Solvable Paradigm for Speeding up Many-Particle {M}arkov Chains},
   volume = {14},
   url = {https://journals.aps.org/prx/abstract/10.1103/PhysRevX.14.041035},
   year = {2024}
}

@article{Cranston1996,
   author = {M. Cranston and T. S. Mountford},
   doi = {10.1214/AOP/1065725183},
   issn = {0091-1798},
   issue = {3},
   journal = {Annals of Probability},
   month = {7},
   pages = {1300-1323},
   publisher = {Institute of Mathematical Statistics},
   title = {The strong law of large numbers for a {B}rownian polymer},
   volume = {24},
   url = {https://projecteuclid.org/journals/annals-of-probability/volume-24/issue-3/The-strong-law-of-large-numbers-for-a-Brownian-polymer/10.1214/aop/1065725183.full},
   year = {1996}
}

@article{Raimond1997,
   author = {O.~Raimond},
   doi = {10.1007/S004400050082/METRICS},
   issn = {01788051},
   issue = {2},
   journal = {Probability Theory and Related Fields},
   month = {10},
   pages = {177-196},
   publisher = {Springer New York},
   title = {Self-attracting diffusions: {C}ase of the constant interaction},
   volume = {107},
   url = {https://link.springer.com/article/10.1007/s004400050082},
   year = {1997}
}

@article{Benaim2005,
   author = {M.~Benaïm and O.~Raimond},
   doi = {10.1214/009117905000000251},
   issn = {00911798},
   issue = {5},
   journal = {Annals of Probability},
   month = {9},
   pages = {1716-1759},
   publisher = {Institute of Mathematical Statistics},
   title = {Self-interacting diffusions. {I}{I}{I}. {S}ymmetric interactions},
   volume = {33},
   url = {https://projecteuclid.org/journals/annals-of-probability/volume-33/issue-5/Self-interacting-diffusions-III-Symmetric-interactions/10.1214/009117905000000251.full},
   year = {2005}
}

@article{Benaim2011,
   author = {M.~Benaïm and O.~Raimond},
   doi = {10.1214/EJP.v16-948},
   issn = {10836489},
   issue = {66},
   journal = {Electronic Journal of Probability},
   month = {1},
   pages = {1815-1843},
   publisher = {Institute of Mathematical Statistics and Bernoulli Society},
   title = {Self-interacting diffusions {I}{V}: {R}ate of convergence},
   volume = {16},
   url = {http://www.math.washington.edu/~ejpecp/},
   year = {2011}
}

@article{Benaim2003,
   author = {M.~Benaïm and O.~Raimond},
   doi = {10.1016/S0246-0203(03)00028-1},
   issn = {0246-0203},
   issue = {6},
   journal = {Annales de l'Institut Henri Poincare (B) Probability and Statistics},
   month = {11},
   pages = {1043-1055},
   title = {Self-interacting diffusions {I}{I}: {C}onvergence in law},
   volume = {39},
   year = {2003}
}

@article{Benaim2002,
   author = {M.~Benaïm and M.~Ledoux and O.~Raimond},
   doi = {10.1007/s004400100161},
   issn = {01788051},
   issue = {1},
   journal = {Probability Theory and Related Fields},
   month = {1},
   pages = {1-41},
   publisher = {Springer},
   title = {Self-interacting diffusions},
   volume = {122},
   url = {https://link.springer.com/article/10.1007/s004400100161},
   year = {2002}
}

@article{Kurtzmann2010,
   author = {A.~Kurtzmann},
   doi = {10.1214/09-AIHP206},
   issn = {02460203},
   issue = {3},
   journal = {Annales de l'institut Henri Poincare (B) Probability and Statistics},
   month = {8},
   pages = {618-643},
   publisher = {Institut Henri Poincaré},
   title = {The {O}{D}{E} method for some self-interacting diffusions on $\mathbb{R}^d$},
   volume = {46},
   url = {https://www.numdam.org/item/AIHPB_2010__46_3_618_0/},
   year = {2010}
}

@article{Chambeu2011,
   author = {S.~Chambeu and A.~Kurtzmann},
   doi = {10.3150/10-BEJ310},
   issn = {13507265},
   issue = {4},
   journal = {Bernoulli},
   month = {11},
   pages = {1248-1267},
   title = {Some particular self-interacting diffusions: {E}rgodic behaviour and almost sure convergence},
   volume = {17},
   url = {https://hal.archives-ouvertes.fr/hal-01277805},
   year = {2011}
}

@article{Kleptsyn2012,
   author = {V.~Kleptsyn and A.~Kurtzmann},
   doi = {10.1214/EJP.V17-2121},
   issn = {1083-6489},
   issue = {none},
   journal = {Electronic Journal of Probability},
   month = {1},
   pages = {1-37},
   publisher = {Institute of Mathematical Statistics and Bernoulli Society},
   title = {Ergodicity of self-attracting motion},
   volume = {17},
   url = {https://projecteuclid.org/journals/electronic-journal-of-probability/volume-17/issue-none/Ergodicity-of-self-attracting-motion/10.1214/EJP.v17-2121.full},
   year = {2012}
}

@book{Peliti2021,
   author = {L.~Peliti and S.~Pigolotti},
   isbn = {9780691201771},
   pages = {248},
   publisher = {Princeton University Press},
   title = {Stochastic Thermodynamics : An Introduction},
   year = {2021}
}

@article{Seifert2012,
   author = {U.~Seifert},
   doi = {10.1088/0034-4885/75/12/126001},
   issn = {00344885},
   issue = {12},
   journal = {Reports on Progress in Physics},
   month = {12},
   pages = {126001},
   pmid = {23168354},
   publisher = {IOP Publishing},
   title = {Stochastic thermodynamics, fluctuation theorems and molecular machines},
   volume = {75},
   url = {https://iopscience.iop.org/article/10.1088/0034-4885/75/12/126001},
   year = {2012}
}

@article{Budhiraja2025,
   author = {A.~Budhiraja and A.~Waterbury and P.~ Zoubouloglou},
   doi = {10.1016/J.SPA.2025.104640},
   issn = {0304-4149},
   journal = {Stochastic Processes and their Applications},
   month = {8},
   pages = {104640},
   publisher = {North-Holland},
   title = {Large deviations for empirical measures of self-interacting {M}arkov chains},
   volume = {186},
   url = {https://www.sciencedirect.com/science/article/pii/S030441492500081X},
   year = {2025}
}

@article{Brandner2025a,
   author = {Kay Brandner},
   doi = {10.1103/PHYSREVLETT.134.037101},
   issn = {10797114},
   issue = {3},
   journal = {Physical Review Letters},
   month = {1},
   pages = {037101},
   publisher = {American Physical Society},
   title = {Dynamics of Microscale and Nanoscale Systems in the Weak-Memory Regime},
   volume = {134},
   url = {https://journals.aps.org/prl/abstract/10.1103/PhysRevLett.134.037101},
   year = {2025}
}

@article{Brandner2025b,
   author = {Kay Brandner},
   doi = {10.1103/PhysRevE.111.014137},
   issn = {24700053},
   issue = {1},
   journal = {Physical Review E},
   month = {1},
   pages = {014137},
   pmid = {39972827},
   publisher = {American Physical Society},
   title = {Dynamics of microscale and nanoscale systems in the weak-memory regime: A mathematical framework beyond the {M}arkov approximation},
   volume = {111},
   url = {https://journals.aps.org/pre/abstract/10.1103/PhysRevE.111.014137},
   year = {2025}
}

@article{Schilling2022,
   author = {Tanja Schilling},
   doi = {10.1016/J.PHYSREP.2022.04.006},
   issn = {0370-1573},
   journal = {Physics Reports},
   month = {8},
   pages = {1-45},
   publisher = {North-Holland},
   title = {Coarse-grained modelling out of equilibrium},
   volume = {972},
   url = {https://www.sciencedirect.com/science/article/pii/S0370157322001314},
   year = {2022}
}

@article{Esposito2012,
   author = {Massimiliano Esposito},
   doi = {10.1103/PhysRevE.85.041125},
   issn = {15393755},
   issue = {4},
   journal = {Physical Review E},
   month = {4},
   pages = {041125},
   publisher = {American Physical Society},
   title = {Stochastic thermodynamics under coarse graining},
   volume = {85},
   url = {https://journals.aps.org/pre/abstract/10.1103/PhysRevE.85.041125},
   year = {2012}
}

@article{Mori1965,
   author = {Hazime Mori},
   doi = {10.1143/PTP.33.423},
   issn = {0033-068X},
   issue = {3},
   journal = {Progress of Theoretical Physics},
   month = {3},
   pages = {423-455},
   publisher = {Oxford Academic},
   title = {Transport, Collective Motion, and {B}rownian Motion},
   volume = {33},
   url = {https://dx.doi.org/10.1143/PTP.33.423},
   year = {1965}
}

@article{Zwanzig1973,
   author = {Robert Zwanzig},
   doi = {10.1007/BF01008729/METRICS},
   issn = {00224715},
   issue = {3},
   journal = {Journal of Statistical Physics},
   month = {11},
   pages = {215-220},
   publisher = {Kluwer Academic Publishers-Plenum Publishers},
   title = {Nonlinear generalized {L}angevin equations},
   volume = {9},
   url = {https://link.springer.com/article/10.1007/BF01008729},
   year = {1973}
}

@article{Lapolla2019,
   author = {Alessio Lapolla and Aljaž Godec},
   doi = {10.3389/FPHY.2019.00182/BIBTEX},
   issn = {2296424X},
   journal = {Frontiers in Physics},
   month = {11},
   pages = {493792},
   publisher = {Frontiers Media SA},
   title = {Manifestations of Projection-Induced Memory: General Theory and the Tilted Single File},
   volume = {7},
   url = {www.frontiersin.org},
   year = {2019}
}

@article{Kanazawa2024,
   author = {Kiyoshi Kanazawa and Didier Sornette},
   doi = {10.1103/PhysRevResearch.6.023270},
   issn = {26431564},
   issue = {2},
   journal = {Physical Review Research},
   month = {6},
   pages = {023270},
   publisher = {American Physical Society},
   title = {Standard form of master equations for general non-{M}arkovian jump processes: The {L}aplace-space embedding framework and asymptotic solution},
   volume = {6},
   url = {https://journals.aps.org/prresearch/abstract/10.1103/PhysRevResearch.6.023270},
   year = {2024}
}

@article{Bertini2015,
   author = {Lorenzo Bertini and Alessandra Faggionato and Davide Gabrielli},
   doi = {10.1214/14-AIHP601},
   issn = {0246-0203},
   issue = {3},
   journal = {Ann.\ Inst.\ H.\ Poincaré Probab.\ Statist.},
   pages = {867-900},
   publisher = {Institut Henri Poincaré},
   title = {Large deviations of the empirical flow for continuous time {M}arkov chains},
   volume = {51},
   url = {https://projecteuclid.org/journals/annales-de-linstitut-henri-poincare-probabilites-et-statistiques/volume-51/issue-3/Large-deviations-of-the-empirical-flow-for-continuous-time-Markov/10.1214/14-AIHP601.full},
   year = {2015}
}

@article{Barato2015,
   author = {Andre C. Barato and Raphaël Chetrite},
   doi = {10.1007/s10955-015-1283-0},
   issn = {00224715},
   issue = {5},
   journal = {Journal of Statistical Physics},
   pages = {1154-1172},
   publisher = {Springer New York LLC},
   title = {A formal view on level 2.5 large deviations and fluctuation relations},
   volume = {160},
   url = {https://link.springer.com/article/10.1007/s10955-015-1283-0},
   year = {2015}
}

@article{HuaLiuKai2,
   author = {X.~Huang and Y.~Liu and K.~Xiang},
   journal = {arXiv:2505.05284},
   month = {5},
   title = {Tail distributions of cover times of once-reinforced random walks},
   url = {https://arxiv.org/pdf/2505.05284},
   year = {2025}
}

\bigskip
\noindent \scriptsize{\textsc{ 
\begin{minipage}{0.5\linewidth}
A. Budhiraja \newline
Department of Statistics and Operations Research
\newline
University of North Carolina\newline
Chapel Hill, NC 27599, USA\newline
email:  budhiraj@email.unc.edu     
\end{minipage}
\hfill
\begin{minipage}{0.4\linewidth}
F.Coghi\newline
School of Physics and Astronomy,\newline
and, Centre for the Mathematics and Theoretical Physics of Quantum Non-Equilibrium Systems,\\
University of Nottingham,\newline
 Nottingham, NG7 2RD, United Kingdom\newline
 email: francesco.coghi@leicester.ac.uk
\end{minipage}}
}
\end{document}